%% file: LL-Geometric-constants-Miller.tex
\documentclass[10pt]{article}
\usepackage[english,activeacute]{babel}
\usepackage{epsfig}

\usepackage[latin1]{inputenc}
\usepackage[T1]{fontenc}
\usepackage{stmaryrd}
\usepackage{amsmath,amsfonts,amssymb,mathrsfs,amsthm}
\usepackage{xcolor}
\usepackage{dsfont}
\usepackage{graphicx}
\usepackage[mathscr]{eucal}
\usepackage{makeidx}
\usepackage{verbatim}
\usepackage{graphics,graphicx}
\usepackage{textcomp}
\usepackage{float}
\usepackage[colorlinks=true]{hyperref}
\input mymacros.tex

\newcommand\bna{\begin{eqnarray*}}
\newcommand\ena{\end{eqnarray*}}

\newcommand\bnan{\begin{eqnarray}}
\newcommand\enan{\end{eqnarray}}

\newcommand\bnp{\begin{proof}}
\newcommand\enp{\end{proof}}

\newcommand\bneq{\begin{eqnarray*}\left\lbrace \begin{array}{rcl}}
\newcommand\eneq{\end{array} \right.\end{eqnarray*}}
\newcommand\bneqn{\begin{eqnarray}\left\lbrace \begin{array}{rcl}}
\newcommand\eneqn{\end{array} \right.\end{eqnarray}}

\newcommand\nor[2]{\left\|#1\right\|_{#2}}
\newcommand\sgn{\textnormal{sgn}}

\newcommand\e{\varepsilon}
\newcommand{\co}{\ensuremath{\mathfrak K}}

\newcommand\gl[2]{\left\langle #1, #2\right\rangle_g}
\newcommand\gln[1]{\left| #1 \right|_g^{2}}
\newcommand\Lap{\Delta_{g}}
\newcommand\nablag{\nabla_{g}}
\newtheorem{formule}{Formula}
\newcommand\intS{\int_{\Sigma}}
\newcommand\rr{r} 

\begin{document}
\title{Observability of the heat equation, geometric constants in control theory, and a conjecture of Luc Miller}

\author{Camille Laurent\footnote{CNRS UMR 7598 and Sorbonne Universit\'es UPMC Univ Paris 06, Laboratoire Jacques-Louis Lions, F-75005, Paris, France, email: laurent@ann.jussieu.fr} and Matthieu L\'eautaud\footnote{\'Ecole Polytechnique, Centre de Math\'ematiques Laurent Schwartz UMR7640,  91128 Palaiseau cedex France, email: matthieu.leautaud@polytechnique.edu.} 
}

\date{}

\maketitle

\begin{abstract}
This article is concerned in the first place with the short-time observability constant of the heat equation from a subdomain $\omega$ of a bounded domain $\M$. The constant is of the form $e^{\frac{\co}{T}}$, where $\co$ depends only on the geometry of $\M$ and $\omega$. Luc Miller~\cite{Miller:04} conjectured that $\co$ is (universally) proportional to the square of the maximal distance from $\omega$ to a point of $\M$. 
We show in particular geometries that $\co$ may blow up like $|\log(r)|^2$ when $\omega$ is a ball of radius $r$, hence disproving the conjecture. 
We then prove in the general case the associated upper bound on this blowup.
We also show that the conjecture is true for positive solutions of the heat equation. 

The proofs rely on the study of the maximal vanishing rate of (sums of) eigenfunctions. They also yield lower and upper bounds for other geometric constants appearing as tunneling constants or approximate control costs.

As an intermediate step in the proofs, we provide a uniform Carleman estimate for Lipschitz metrics. The latter also implies uniform spectral inequalities and observability estimates for the heat equation in a bounded class of Lipschitz metrics, which are of independent interest.
\end{abstract}


\tableofcontents

\section{Introduction and main results}

We are interested in several constants appearing in the study of eigenfunctions concentration and control theory, and the links between them. In the whole paper, we are given a connected compact Riemannian manifold $(\M,g)$ with or without boundary $\d \M$, we denote by $\Delta_g$ the (negative) Laplace-Beltrami operator on $\M$. In case $\d \M \neq \emptyset$, we denote by $\Int(\M)$ the interior of $\M$, so that $\M = \d \M \sqcup \Int(\M)$ (see e.g.~\cite[Chapter~1]{Lee:book}).
For readability, we first focus in the next section on the results concerning the observability constant for the heat equation.

\subsection{The control cost for the heat equation}
Here, we study the so-called \emph{cost of controllability} of the heat equation. 
It is well known since the seminal papers of Lebeau-Robbiano \cite{LR:95} and Fursikov-Imanuvilov \cite{FI:96} that for any time $T>0$, the heat equation is controlable to zero.
More precisely, by duality, the controlability problem is equivalent to the observability problem for solutions of the free heat equation (see e.g. \cite[Section 2.5.2]{Cor:book}):
For any non empty open set $\omega$ and $T>0$, there exist $C_{T,\omega}$ such that we have 
\bnan
\label{e:co-heat}
 \nor{e^{T\Delta_g}u}{L^2(\M)}^2  \leq C_{T,\omega}^2 \int_0^T \nor{e^{t\Delta_g}u}{L^2(\omega)}^2 dt , \quad \text{ for all $T>0$ and all $u \in L^2(\M)$.}
 \enan
 Here, $(e^{t\Delta_g})_{t>0}$ denotes the semigroup generated by the {\em Dirichlet} Laplace operator on $\M$ (otherwise explicitely defined).
The observability constant $C_{T,\omega}$ is then directly related to the cost of the control to zero and has been the object of several studies. 

It has been proved by Seidman \cite{Seidman:84} in dimension one (in the closely related case of a boundary observation) and by Fursikov-Imanuvilov~\cite{FI:96} in general (see also \cite{Miller:10} for obtaining this result via the Lebeau-Robbiano method), that the cost in small time blows up at most exponentially:
\bnan
\label{defcoheat}
\omega \neq \emptyset \quad \implies \quad \text{ there is } C,\co>0 \text{ such that } C_{T,\omega}\leq Ce^{\frac{\co}{T}} \quad \text{for all } T>0.
\enan
Gu\"ichal \cite{Guichal:85} in one dimension and Miller~\cite{Miller:04} in the general case proved that exponential blowup indeed occurs:
$$
\overline{\omega} \neq \M \quad \implies \quad \text{ there is } c>0 \text{ such that } C_{T,\omega}\geq c e^{\frac{c}{T}}\quad \text{for all } T>0.
$$
This suggest to define 
\bnan
\label{e:def-Kheat}
\co_{heat} (\omega) =  \inf \left\{ \co >0 , \exists C>0 \text{ s.t. \eqref{e:co-heat} holds with } C_{T,\omega}= Ce^{\frac{\co}{T}}\right\},
\enan
which, according to the abovementionned results satisfies $\co_{heat}(\omega) <\infty$ as soon as $\omega \neq \emptyset$ and $\co_{heat}(\omega)>0$, as soon as $\overline{\omega} \neq \M$.
This constant depends only on the geometry of the manifold $(\M,g)$ and the subset $\omega$. It is expected to contain geometric features of short time heat propagation, and has thus received a lot of attention in the past fifteen years~\cite{Miller:04,Miller:04ARMA,Miller:06c,TT:07,Miller:10,EZ:11s,TT:11,BP:17,Darde:17,EV:17,NTTV:18,PhunglogCarl}.

In this direction, the result of Miller~\cite{Miller:04} is actually more precise and provides a geometric lower bound: for all $(\M, g), \omega$, we have 
$$\co_{heat} (\omega) \geq \frac{\L(\M,\omega)^2}{4},
$$
where, for $E \subset \M$, we write
\begin{equation}
\label{e:def-L-omega-M}
\L(\M,E) = \sup_{x \in \M}\dist_g(x,E) .
\end{equation} The proof relies on heat kernel estimates. In~\cite{Miller:04,Miller:06b}, Luc Miller also proved that in case $\omega$ satisfies the Geometric Control Condition in $(\M, g)$ (see~\cite{BLR:92}) we have 
$$\co_{heat} (\omega)\leq \alpha_{*} L_{\omega}^2,$$ 
where $L_\omega$ is the maximal length of a ``ray of geometric optics'' (i.e. geodesic curve in case $\d \M =\emptyset$) not intersecting $\omega$, and $\alpha_*\leq 2$ is an absolute constant (independent of the geometry).
Based on these results and the idea that the heat kernel provides the most concentrated solutions of the heat equation, he formulated the following conjecture~\cite[Section~2.1]{Miller:04}-\cite[Section~3.1]{Miller:06c}.
\begin{conjecture}[Luc Miller]
\label{c:Miller-conj}
For all $(\M,g)$ and $\omega \subset \M$ such that $\overline{\omega} \neq \M$, we have $\co_{heat} (\omega) = \frac{\L(\M,\omega)^2}{4}$.
\end{conjecture}

Note that it has been proved in \cite{Lissy:15} that, in the related context of the 1D heat equation with a boundary observation, the factor $\frac{1}{4}$ might not be correct (and should be replaced by $\frac12$, see Section \ref{sectpreviousres} below). Our first result disproves Conjecture~\ref{c:Miller-conj} in a stronger sense. 
\begin{theorem}[Counterexamples]
\label{thm:counterexamples}
Assume $(\M, g)$ is one of the following 
\begin{enumerate}
\item $\M =\S^n \subset \R^{n+1}$ and $g$ is the canonical metric (see Section~\ref{s:sphere});
\item $\M=\mathcal{S}\subset \R^3$ is a surface of revolution diffeomorphic to the sphere $\S^2$, and $g$ is the metric induced by the Euclidean metric on $\R^3$ (with additional non degeneracy conditions, see Section~\ref{s:revol});
\item $\M = \ID=\left\{(x_1,x_2)\in \R^2\left| \ x_1^2+x_2^2\leq 1\right.\right\}\subset \R^2$ is the unit disk, $g$ the Euclidean metric and Dirichlet conditions are taken on $\d \M$ (see Section \ref{sectDisk}).
\end{enumerate}
Then, for any $C>0$, there exists $\omega\subset \M$ so that $ \co_{heat} (\omega) \geq C \L(\M,\omega)$ and $\co_{heat} (\omega) \geq C$.

More precisely, assume that $x_0$ is either
\begin{enumerate}
\item any point in $\S^n$,
\item one of the two points that intersect the axis of revolution of $\mathcal{S}\subset \R^3$,
\item the center of $\ID$.
\end{enumerate}
Then, there exists $C>0$ and $r_0>0$ so that we have 
\bnan
\label{lowerlogheat}
\co_{heat} (B_g(x_0,r)) \geq C|\log(r)|^2
\enan
for any $0<r\leq  r_0$.
\end{theorem}
Here, $B_g(x_0,r)$ denotes the geodesic ball of $\M$ centered at $x_0$ of radius $r$. 
The results we obtain are slightly more precise. In particular, the constant $C$ is an explicit geometric constant.
The lower bounds are related to an appropriate Agmon distance associated to the problem. We refer to Corollary~\ref{t:agmon-introenkappa} below for more precise estimates. 

Note also that this blow up of $\co_{heat} (B(x_0,r))$ for small $r$ does not always happen and is due here to a particular (de)concentration phenomenum. For instance on $\M=\mathbb{T}^1$, the set $\omega=B(x_0,r)$ always satisfies the Geometric Control Condition for any time $T>1-2r$. Abstract results (see \eqref{upperboundMiller} below for more details) give $\co_{heat} (B(x_0,r))\leq \alpha^*$ for any $r>0$ and blowup does not occur.

\medskip
Our next result shows that the blow up given by~\eqref{lowerlogheat} is actually optimal as far as the asymptotics of $\co_{heat}$ for small balls is concerned.
 We prove the following observability result from small balls (closely related to previous results of Jerison-Lebeau \cite{JL:99}, see Section~\ref{subsub:LR-Jerison-Lebeau} below).
\begin{theorem}
\label{t:heatlog}
For all $x_0\in \M$, there exist $C>0$ such that for  all $r>0$ we have
$$
\co_{heat} (B(x_0,r)) \leq C|\log(r)|^2+C .
$$
\end{theorem}
Note that Bardos and Phung~\cite{BP:17,PhunglogCarl} recently proved independently that $\co_{heat} (B(x_0,r)) \leq \frac{C_\epsilon}{r^\epsilon} +C_\epsilon$ for all $\epsilon>0$ in case $\M \subset \R^n$ is star-shaped w.r.t. $x_0$.

These results seem to suggest that $\L(\M,\omega)$ is not the only appropriate parameter needed for estimating $\co_{heat} (\omega)$.
There are indeed some solutions of the heat equation concentrating more than the heat kernel for small times. 
Our last result concerning the heat equation goes actually in the opposite direction. It provides with a large class of solutions of the heat equation, namely {\em positive} solutions, that do not concentrate more than the heat kernel, thus proving Conjecture~\ref{c:Miller-conj} when restricted to this class of solutions.

Recall that $\L(\M,E)$ is defined in~\eqref{e:def-L-omega-M}.
\begin{theorem}
\label{thmpositive}
Assume that $(\M, g)$ has geodesically convex boundary $\d \M$.
Then, for any nonempty open set $\omega\subset \M$ and $z_0\in\M$, for any $\e>0$, there exists $C,D>0$ so that for any $0<T\leq D$, we have 
\bnan
\label{estimpos1}\nor{u(T)}{L^2(\M)}^2\leq  \frac{C}{T} e^{\frac{(1+\e)\mathcal{L}(M,\omega)^2}{2T}}\int_0^{T}\nor{u(t, \cdot)}{L^2(\omega)}^2~dt ,\\
\label{estimpos2}\nor{u(T)}{L^2(\M)}^2\leq  \frac{C}{T} e^{\frac{(1+\e)\mathcal{L}(M,z_0)^2}{2T}}\int_0^{T}u(t,z_0)^2~dt ,
\enan
for all $u_0 \in L^2(M)$ such that $u_0 \geq 0$ a.e. on $\M$ and associated solution $u$ to 
$$
(\d_t - \Delta_g) u =0 \text{ on }\R^+_* \times \Int(\M), \quad u|_{t=0} = u_0  \text{ in } \Int(\M), \quad \d_\nu u =0 \text{ on }\R^+ \times \d \M .
$$
\end{theorem}
Theorem~\ref{thmpositive} follows from classical Li-Yau estimates~\cite{LY:86}.
Notice that here, Neumann boundary conditions are taken ($\nu$ denotes a unit vector field normal to $\d \M$), and an additional geometric assumption is made (convexity of $\d \M$). The result still holds without the convexity assumption up to replacing $(1+\e)$ in the exponent by a geometric constant, see Remark~\ref{r:non-conv-Li-Yau}. 
We also recall that for nonnegative initial data $u_0 \geq0$, the solution of the heat equation remains nonnegative for all times. 
Of course, the counterexamples of Theorem~\ref{thm:counterexamples} prevent these estimates to hold in general. 
Estimate~\eqref{estimpos2} is particularly surprising (even without considering the value of the constants) and of course only true for positive solutions (otherwise just taking $z_0$ in a nodal set of an eigenfunction of $\Delta_g$ invalidates~\eqref{estimpos2}).
Finally, let us mention that the constants $C$ and $D$ are explicitely estimated by geometric quantities (see Remark~\ref{rkexplicitpos}).

\bigskip
Let us now put these results in a broader context, and introduce several related geometric constants appearing in tunneling estimates and control theory.

\subsection{Tunneling constants in control theory, and their links}
\label{sectlinkintro}

The lower bounds of Theorem~\ref{thm:counterexamples} are proved using very particular solutions to the heat equation arising from by eigenfunctions (exhibiting a very strong concentration far from $x_0$ as well as a strong deconcentration near $x_0$). It is therefore natural to study related constants measuring such (de)concentration properties.
In this section, we introduce all geometric constants studied in the paper and collect known links between them.

We first introduce spectral subspaces of the Laplace operator $\Delta_g$ (with Dirichlet boundary conditions if $\d \M \neq \emptyset$), which are at the core of most results presented here.
Namely, for $\lambda \in \Sp(-\Delta_g)$, the space
$$E_{\lambda} := \vect\{\psi \in L^2(\M), -\Delta_g \psi = \lambda \psi \}$$
denotes the eigenspace associated to the eigenvalue $\lambda$ and, for all $\lambda >0$, 
$$E_{\leq \lambda} :=  \vect\{ E_{\lambda_j}, \lambda_j \in \Sp(-\Delta_g), \lambda_j \leq \lambda\} ,
$$ the space of linear combinations of eigenfunctions associated to eigenvalues $\leq \lambda$.

Let us now introduce the constants studied in the article, else than that involved in~\eqref{e:co-heat}-\eqref{defcoheat}. 
For any nonempty open subset $\omega \subset \M$, we recall the following results: 
\begin{itemize}
\item Vanishing of eigenfunctions~\cite{DF:88,LR:95}: there exist $C,\co$ such that we have 
\bnan
\label{e:co-eig}
 \nor{\psi}{L^2(\M)}  \leq Ce^{\co\sqrt{\lambda}}\nor{\psi}{L^2(\omega)} , \quad \text{for all $\lambda \in \Sp(-\Delta_g)$ and $\psi \in E_\lambda$}.
 \enan
\item Vanishing of sums of eigenfunctions (so-called Lebeau-Robbiano spectral inequality)~\cite{LR:95,JL:99,LZ:98}: there exist $C,\co$ such that we have 
\bnan
\label{e:co-sum}
 \nor{u}{L^2(\M)}  \leq Ce^{\co\sqrt{\lambda}}\nor{u}{L^2(\omega)} , \quad \text{ for all $\lambda >0$ and all $u \in E_{\leq \lambda}$.}
 \enan
\item Infinite time observability of the heat equation~\cite{FI:96}: there exist $C,\co$ such that we have 
\bnan
\label{e:co-infty}
\int_{\R^+} e^{-\frac{2\co}{t}} \|e^{t\Delta_g}u\|_{L^2(\M)}^2 dt \leq C\int_{\R^+} \|e^{t\Delta_g}u\|_{L^2(\omega)}^2 dt , \quad \text{ for all $u \in L^2(\M)$.}
\enan
\item Approximate observability for the wave equation~\cite{LL:15},
\begin{align}
\label{e:wave}
(\d_t^2 - \Delta_g)u = 0, \quad  u|_{(0,T)\times \d \M}=0, \quad (u,\d_tu)|_{t=0} =(u_0, u_1):
\end{align}
For all $T> 2\L(\M,\omega)$, there exist $C,\co, \mu_0>0$ such that we have 
\begin{align}
\label{e:co-wave}
\|(u_0,u_1)\|_{L^2(\M) \times H^{-1}(\M)} \leq Ce^{\co \mu} \|u\|_{L^2((0,T)\times\omega)} + \frac{1}{\mu} \|(u_0,u_1)\|_{H^1_0(\M)\times L^2(\M)}, \nonumber \\ \text{ for all $\mu \geq \mu_0$ and all $(u_0, u_1) \in H^1_0(\M) \times L^2(\M)$, and $u$ solution to~\eqref{e:wave}}.
\end{align}
\end{itemize}
Recall the definition of $\L(\M,\omega)$ in~\eqref{e:def-L-omega-M}.
Remark that this last estimate is equivalent to (see~\cite{LL:15} or Corollary~\ref{c:lambda=mu} below)
\bnan
\label{e:co-wave-bis}
 \|(u_0,u_1)\|_{H^1_0(\M)\times L^2(\M)}\leq C' e^{\co' \Lambda} \|u\|_{L^2((0,T)\times\omega)}, \quad \Lambda = \frac{\|(u_0,u_1)\|_{H^1_0(\M)\times L^2(\M)}}{\|(u_0,u_1)\|_{L^2(\M) \times H^{-1}(\M)} }, \nonumber \\ \text{ for all $(u_0, u_1) \in H^1_0(\M) \times L^2(\M)$, and $u$ solution to~\eqref{e:wave}.}
\enan

Note that in the reference~\cite{LL:15}, the observation term in the right hand-side of these inequalities is $\|u\|_{L^2(0,T;H^1(\omega))}$ instead of $\|u\|_{L^2((0,T)\times\omega)}$. That the stronger inequalities above holds is proved in~\cite[Section~5.3]{LL:17Hypo} (see also~\cite{LL:17approx}).

\bigskip
In all these inequalities, we are interested in the ``best constant $\co$'' such that the estimate holds for some $C$.  More precisely, we are interested in the way it depends on the geometry of $(\M,g)$ and $\omega$ (and, in the case of~\eqref{e:co-wave}, the time $T$).
Let us first formulate the precise definitions of these constants. These are the analogues to that of $\co_{heat} (\omega)$ given in~\eqref{e:def-Kheat}.

\begin{definition}
\label{def-coco}
Given $\omega\subset \M$ an open set, we define $\co_{eig} (\omega),  \co_\Sigma (\omega) , \co_{\infty}(\omega), \co_{wave}(\omega,T)$ to be the best exponents in the above estimates~\eqref{e:co-eig}-\eqref{e:co-wave}, namely:
\bna
\co_{eig} (\omega) =  \inf \left\{ \co >0 , \exists C>0 \text{ s.t. \eqref{e:co-eig} holds} \right\} ,
\ena
\bna
\co_{\Sigma} (\omega) =  \inf \left\{ \co >0 , \exists C>0 \text{ s.t. \eqref{e:co-sum} holds} \right\} ,
\ena
\bna
\co_{\infty} (\omega) =  \inf \left\{ \co >0 , \exists C>0 \text{ s.t. \eqref{e:co-infty} holds} \right\} ,
\ena
\begin{align}
\label{e:co=co'}
\co_{wave} (\omega,T) & =  \inf \left\{ \co >0 , \exists C>0, \mu_0>0  \text{ s.t. \eqref{e:co-wave} holds} \right\} \nonumber \\
&=  \inf \left\{ \co' >0 , \exists C'>0,  \text{ s.t. \eqref{e:co-wave-bis} holds} \right\}.
\end{align}
\end{definition}
A proof of the equality in~\eqref{e:co=co'} is given in Corollary~\ref{c:lambda=mu} below. Note that we may write $\co_{wave} (\omega,T) = + \infty$ if $T< 2\L(\M,\omega)$ since~\eqref{e:co-wave}-\eqref{e:co-wave-bis} are known not to hold (see the discussion in~\cite{LL:15}). However, $\co_{wave} (\omega,T) < + \infty$ as soon as $T>2\L(\M,\omega)$, by virtue of~\eqref{e:co-wave}-\eqref{e:co-wave-bis}.

\bigskip
Let us now collect some known facts concerning these constants, in addition to the already discussed bound $\co_{heat}(\omega)\geq \frac{\L(\M,\omega)^2}{4}$~\cite{Miller:04}. A first trivial (but useful) fact is that $\co_{eig} (\omega) \leq \co_{\Sigma} (\omega)$. 
The following properties can also be found in the literature:
\begin{enumerate}
\item For all  $(\M, g), \omega$ such that $\overline{\omega} \neq \M$, we have $\co_{\Sigma}(\omega)\geq \frac{\L(\M,\omega)}{2}$, see~\cite[Theorem~5.3]{Miller:10} (that $\co_{\Sigma}(\omega)>0$ had already been proved in~\cite{JL:99}). 
\item $\co_{\infty} (\omega) \leq \co_{heat} (\omega)$, \cite[Theorem~1]{Miller:06c}.
\item For all  $(\M, g), \omega$, we have $\co_{\infty} (\omega) \geq \frac{d_1(\omega)^2}{4}$, with $d_1(\omega) = \sup \left\{ r>0 , \exists x\in \M, B(x,r) \subset \M\setminus \overline{\omega} \right\}$, see~\cite{FCZ:00} and~\cite[Section~4.1]{Zua:01}.
\item Assume $\omega$ satisfies the Geometric Control Condition in $(\M, g)$ and denote by $L_\omega$ the maximal length of a ray of geodesic optics not intersecting $\omega$. Then, we have 
\bnan
\label{upperboundMiller}
\co_{heat} (\omega) \leq \alpha_* L_\omega^2
\enan with $\alpha_* = 2 \left(\frac{36}{37}\right)^2$, see~\cite{Miller:04,Miller:06b} (improved to $\alpha_* =3/4$ in~\cite{TT:07} and to $0.6966$ in~\cite{Darde:17}).
\item Assume $\omega$ satisfies the Geometric Control Condition in $(\M, g)$ and denote by $L_\omega$ the maximal length of a ray of geometric optics not intersecting $\omega$. Then, we have $\co_{\infty} (\omega) \leq \frac1{16} L_\omega^2$, see~\cite[Theorem~1.1]{EZ:11s}.
\item $\co_{heat} (\omega) \leq 4 \co_{\Sigma} (\omega)^2$, see~\cite[Corollary~1, see also the discussion in Section~2.4]{Miller:10} (see also~\cite{Sei:08} for a proof $\co_{heat} (\omega) \leq 8 \co_{\Sigma} (\omega)^2$).
\item If $(\omega,T)$ satisfy the geometric control condition~\cite{BLR:92}, then $\co_{wave}(\omega,T) =0$ (more precisely,~\eqref{e:co-wave}-\eqref{e:co-wave-bis} hold with $\co=0$). Conversally, if $(\M,g)$ is real-analytic and $(\overline{\omega},T)$ does not satisfy the geometric control condition (for a ray that only intersects $\d \M$ transversally), then $\co_{wave}(\omega,T) >0$, see~\cite{Leb:Analytic}.
\end{enumerate}
Notice that in all these statements, the constants $\co_{heat}$ and $\co_\infty$ (heat equation) are homogeneous to a square of a distance (as for the heat kernel), whereas the other ones are homogeneous to a distance (as for the wave kernel). 

Remark also that every comparison statement above follows, in the associated reference, from a proper inequality (the above statements being only a weak form of those). 

Also notice that the converse inequality $ \co_{\Sigma} (\omega)^2\leq C \co_{heat} (\omega)$ for a universal constant $C$ is certainly not true in general. For instance, in the case of  boundary control on an interval $(0,1)$ (see Section \ref{sectpreviousres}), $\co_{heat} (\{0\})$ is finite while it is easy to see that $\co_{\Sigma} (\{0\})$ is infinite since no spectral inequality can be true just by dimensional analysis.

\bigskip
We first complete the above list of comparison results  by the following proposition. 
\begin{proposition}[Other links between the constants]
\label{p:link-eigenfct-heat-etc}
We have 
\bna
\frac{\co_{eig}(\omega)^2}{4} \leq \co_{heat}(\omega) , \quad 
\frac{\co_{eig}(\omega)^2}{4} \leq \co_{\infty}(\omega) .
\ena
Also for all $T>0$, we have $\co_{eig}(\omega) \leq \co_{wave}(\omega,T)$. 
\end{proposition}
Note that the last statement is empty if $T< 2\L(\M,\omega)$ since~\eqref{e:co-wave}-\eqref{e:co-wave-bis} are known not to hold (see the discussion in~\cite{LL:15}), but is nonempty if we have $\co_{wave}(\omega,T) <\infty$, that is if $T>2\L(\M,\omega)$, by virtue of~\eqref{e:co-wave}-\eqref{e:co-wave-bis}.

Hence, in order to produce lower bounds for $ \co_\Sigma(\omega), \co_{heat}(\omega), \co_{\infty}(\omega) , \co_{wave}(\omega,T)$, we shall product lower bounds for $\co_{eig}(\omega)$, i.e. construct sequences of eigenfunctions having a maximal vanishing rate on $\omega$. Note also that, summarizing the inequalities so far, we have:
\bnan
\label{e:comparison-constants}
\frac{\co_{eig}(\omega)^2}{4} \leq \co_{\infty}(\omega) \leq \co_{heat}(\omega) \leq 4 \co_\Sigma(\omega)^2,
\enan
so that the understanding of concentration properties for eigenfunctions and sums of eigenfunctions essentially contains those of the heat equation. Therefore, our main focus in the following is to produce:
\begin{itemize}
\item maximally vanishing eigenfunctions in particular geometries to yields a lower bound for $\co_{eig}$;
\item a uniform Lebeau-Robbiano spectral inequality on small balls to yields an upper bound for $\co_\Sigma$.
\end{itemize}

Note that reducing our attention to $\co_{eig}$ in the seek of lower bounds is already very restrictive!
Indeed, as soon as  the Schr\"odinger equation on $(\M,g)$ is observable from $\omega$ in finite time (in particular if $\omega$ satisfies the geometric control condition, see~\cite{BLR:92,Leb:92}), then $\co_{eig} (\omega)=0$ (more precisely, \eqref{e:co-eig} holds with $\co=0$).

\bigskip
Before starting to state these lower/upper bounds, let us give a link between $\co_{heat} (\omega)$ and $\co_{wave}(\omega,T)$, consequence of a result of Ervedoza-Zuazua~\cite{EZ:11} (weak observability with exponential cost for the wave equation implies observability of the heat equation).
\begin{proposition}
\label{propheatwave} 
There exist universal constants $\alpha_{1},\alpha_{2}>0$ so that for any $S>0$, we have
\bna
\co_{heat} (\omega)\leq \alpha_{1}S^2 + \alpha_{2}\co_{wave}(\omega,S)^{2}.
\ena \end{proposition}
The proof of this result in Section~\ref{s:proof-wave-heat} is a little more precise about this estimate. In particular, several values of $(\alpha_1, \alpha_2)$ can be deduced from it.
The value of $\alpha_{1}$ is thought to be related to the cost of the boundary control of the 1D heat equation.
Note that, as in~\eqref{e:comparison-constants}, this yields
\bna
\frac{\co_{eig}(\omega)^2}{4} \leq \co_{\infty}(\omega) \leq \co_{heat}(\omega) \leq \alpha_{1}S^2 + \alpha_{2}\co_{wave}(\omega,S)^{2}, \quad \text{ for all } S>0 .
\ena
However, this upper bound seems for the moment less useful than that of~\eqref{e:comparison-constants}, since the proof of~\eqref{e:co-wave}-\eqref{e:co-wave-bis} in~\cite{LL:15} is more technically involved than that of~\eqref{e:co-sum} in~\cite{LR:95,JL:99,LZ:98}. The computation of $\co_{wave}(\omega,S)$ seems thus more intricate than that of $\co_\Sigma(\omega)$.

\subsection{Main results}

\subsubsection{Constructing maximally vanishing eigenfunctions: lower bound for $\co_{eig}$}
\label{s:intro-const} 
In this Section, we provide lower bounds for $\co_{eig}$ in three different geometries.
This then proves Theorem~\ref{thm:counterexamples} as a direct corollary of Proposition~\ref{p:link-eigenfct-heat-etc}.

\paragraph{The sphere}
We first state the results we obtain on two dimensional sphere $\S^2$, since they are particularly simple. The higher dimensional case $\S^n$ is completely similar.
The sphere $\S^2$ is parametrized by $(s, \theta) \in (0,\pi)\times \S^1$. We denote by $N$ (resp. $S$) the north pole described by $s=0$ (resp. the south pole described by $s=\pi$), and remark that $s$ is the geodesic distance to the point $N$.
\begin{theorem}
\label{t:th-sphere}
For $k\in \N$, the function 
$$
\psi_k (s,\theta)= c_k \sin (s)^k e^{ik\theta}, \quad  \quad c_k = \frac{k^{1/4}}{2^{1/2}\pi^{3/4}}\left( 1 + O(\frac{1}{k})\right) \text{ as } k\to +\infty
$$
satisfies 
\begin{align*}
& -\Delta_g \psi_k = k(k+1) \psi_k \quad \text{on} \quad  \S^2 , \quad \psi_k \in C^\infty(\S^2) , \quad \|\psi_k\|_{L^2(\S^2)} = 1,\\
&|\psi_k (s,\theta)| = c_k \sin (s)^k \leq c_k s^k\quad \text{ for } s\in[0,\pi], k\in \N ,\\
&\|\psi_k\|_{L^2(B(N,r))}^2 = \frac{c_k^2\pi}{k+1}\frac{\sin(r)^{2k+2}}{\cos(r)} (1+R) , \quad |R|\leq \frac{\tan(r)^2}{2k+2} \quad \text{ for } r\in[0,\frac{\pi}{2}), k\in \N .
\end{align*}
\end{theorem}
This result is a much more explicit, more precise (and simpler to prove) version of the general results we obtain on surfaces of revolution. We turn to the general case and shall explain at the end of the section the links with Theorem~\ref{t:th-sphere}.

\paragraph{Surfaces of revolutions}
The precise description of the geometry of the surfaces we consider is given in Section~\ref{s:revol} and we only give here the features required to state the result.
We consider $\M =\calS \subset \R^3$ a smooth compact surface diffeomorphic to the sphere $\S^2$. We assume moreover that it has revolution invariance around an axis, that intersects $\calS$ in two points, the north and the south poles, respectively $N, S \in \calS$. These points are the only invariant points of the revolution symmetry. 
The surface is then endowed with the metric $g$ inherited from the Euclidean metric on $\R^3$, which itself enjoy the rotation invariance.
Then, we describe (almost all) the surface by two coordinates, namely $s = \dist_g(\cdot , N)$, the geodesic distance to the north pole and $\theta$, the angle of rotation. The variable $s$ is in $(0,L)$ where $L=\dist_g(N,S)$.
The surface is characterized by the function $R(s)$ associating to $s$ the Euclidean distance in $\R^3$ to the symmetry axis, which, by definition, is rotationally invariant, and satisfies $R(0) = 0= R(L)$. This function $R$ is the ``profile'' of the revolution surface $\calS$.

\medskip
We shall now assume that $R$ reaches at $s_0$ a global maximum, and introduce the relevant Agmon distance to the ``equator'' $s=s_0$, defined by the eikonal equation
\bnan
\label{e:defdA}
 \big(d_A'(s) \big)^2-\left( \frac{1}{R(s)^2}- \frac{1}{R(s_0)^2} \right) = 0 , \quad d_A(s_0)=0 ,  \quad \sgn(d_A'(s_0)) = \sgn(s-s_0) ,
\enan
or, more explicitely, for $s\in (0,L)$, by
\bnan
\label{e:defbisdA}
d_A(s) = \left| \int_{s_0}^s \sqrt{ \frac{1}{R(y)^2}- \frac{1}{R(s_0)^2}} dy \right|  .
\enan
A more intrinsic definition of $d_A$ is given in Remark~\ref{r:def-intrinsic-dA} below (and requires additional notation).
\begin{theorem}
\label{t:agmon-intro}
Assume that $s \mapsto R(s)$ admits a {\em non-degenerate strict global} maximum at $s_0 \in (0,L)$. 
Then, for all $k\in\N$, there exists $\psi_k \in C^{\infty}(\calS)$,  and $\lambda_k \geq 0$ such that 
$$
\lambda_k=  \frac{k^2}{R(s_0)^2}+k\sqrt{\frac{|R''(s_0)|}{R^3(s_0)}} + O(k^{1/2}), \qquad \|\psi_k\|_{L^2(\calS)}=1, \qquad - \Delta_g \psi_k = \lambda_k \psi_k .
$$
Moreover, there exist $C,C_0,k_0>0$ such that, for all $k \in \N$, $k \geq k_0$ and all $0\leq r \leq s_0$, we have the estimate
$$
\nor{\psi_{k}}{L^2(B(N,r))}
\leq C\lambda_k^{C_0} e^{- d_A(r)R(s_0) \left( \sqrt{\lambda_k} - C\right)}.
$$
\end{theorem}
This statement has to be completed by the asymptotic behavior of $d_A$ (proved in Lemma~\ref{lemma-prop-dA}) when $s\to 0$, namely
\bnan
\label{e:asympt-dA}
d_A(s) = -\log(s) + O(1) , \quad \text{as } s \to 0^+ .
\enan
That is to say that the equator and the poles are infinitely distant to each other for the Agmon distance $d_A$ (as opposed to the geodesic distance $\dist_g$).
Note that at first order, $d_A$ does not depend on the geometry of the surface $\calS$ close to the north pole $N$ ($s=0$). A similar statement holds close to the south pole $S$ ($s=L$).

This, together with Definition~\ref{def-coco} and Proposition~\ref{p:link-eigenfct-heat-etc}, yields the following direct corollary.
\begin{corollary}
\label{t:agmon-introenkappa}
Under the assumptions of Theorem \ref{t:agmon-intro}, for all $0\leq r \leq s_0$, we have the estimate
$$
\co_{eig} (B_g(N,r))\geq d_A(r)R(s_0).
$$
This yields also 
\bna
\co_{\Sigma} (B_g(N,r))\geq d_A(r)R(s_0),&\quad &\co_{wave} (B_g(N,r),T)\geq d_A(r)R(s_0) ,\quad  \textnormal{ for any }T>0,\\
\co_{\infty} (B_g(N,r))\geq \frac{\big(d_A(r)R(s_0)\big)^2}{4}, &\quad& \co_{heat} (B_g(N,r))\geq  \frac{\big(d_A(r)R(s_0)\big)^2}{4}.
\ena
\end{corollary}
Note also that Theorem~\ref{t:agmon-intro}, combined with the explicit asymptotic expansion~\eqref{e:asympt-dA} of the Agmon distance $d_A$ implies the following result.
\begin{corollary}[Rate of vanishing]
\label{c:rate-vanish}
With $(\lambda_k, \psi_k)$ as in Theorem~\ref{t:agmon-intro}, there exist $C,C_0,k_0>0$ such that, for all $k \in \N$, $k \geq k_0$ and all $r \geq 0$, we have
\bna
\nor{\psi_{k}}{L^2(B(N,r))}\leq Ce^{C\sqrt{\lambda_k}}r^{R(s_0)\sqrt{\lambda_k}-C} ,
\ena
and, in any local chart centered at $N$, we have $\d^{\alpha}\psi_{k}(N)=0$ for all $|\alpha|\leq R(s_0)\sqrt{\lambda_k}-C$.
\end{corollary}
As on the sphere, these eigenfunctions saturate the maximal vanishing rate predicted by the Donnelly-Fefferman Theorem~\cite{DF:88}.

Note that in these estimates, $R(s_0) \sqrt{\lambda_k} \sim k$ does not depend on the geometry.

\bigskip
The proofs rely on classical semiclassical decay estimates for eigenfunctions~\cite{Simon:83,HS:84}. We refer to the monographs \cite{Helffer:booksemiclassic,DS:book} for the historical background and more references. An additional difficulty here is linked to the degeneracy of the function $R$ close to the north and south poles.

Note also that, to our knowledge, the idea of constructing such examples on surfaces of revolution is due to Lebeau~\cite{Leb:96} and Allibert\cite{Allibert:98}.

\paragraph{The disk}
Recall that $\ID = \{(x,y)\in \R^2 , x^2+y^2 \leq 1\}$.
Our results on the disk are quite similar to the previous results on revolution surfaces. They are proved in Section \ref{sectDisk}.  
Note the construction is more explicit there since it involves Bessel functions.
As in the above example, the concentration is related to an Agmon distance to the maximum of the radius $r$, which corresponds to the boundary $\d\ID$ here. 
\begin{theorem}[Whispering galleries on the disk]
\label{t:agmon-diskintro}
Denote, for $r \in (0,1]$,
\begin{equation}
\label{e:def-dA-disk}
d_A(r)= - \left(\tanh(\alpha(r))-\alpha(r)\right) ,\quad \text{ with } \quad \alpha(r)=\cosh^{-1}(1/r).
\end{equation}
Then, for all $k\in\N$, there exists $\psi_k \in C^{\infty}(\ID)\cap H^1_0(\ID)$,  and $\lambda_k \geq 0$ such that 
$$
\lambda_k=  k^2+ O(k^{4/3}), \qquad \|\psi_k\|_{L^2(\calS)}=1, \qquad - \Delta_g \psi_k = \lambda_k \psi_k .
$$
Moreover, there exist  $C, \beta , k_0 >0$  such that for all $k \geq k_0$ and $0< r \leq 1-\beta \lambda_k^{-1/3}$, we have
\bna
\|\psi_{k}\|_{L^\infty(B(0,r))} \leq \exp \left(-\left( \sqrt{\lambda_k} - C\lambda_k^{1/6} \right)d_A(r)  + C \lambda_k^{1/6} \right) .
\ena
\end{theorem}
That $d_A$ indeed represents an Agmon distance in the present context is justified in the next paragraph. Note that $d_A$ still satisfies $d_A(r) \sim_{r\to 0^+} \log(\frac{1}{r})$ here, so that the analogues of Corollaries~\ref{t:agmon-introenkappa} and~\ref{c:rate-vanish} still holds in this setting.

\paragraph{Remarks on the Agmon distance}
In this paragraph, we compare the three geometries discussed above. In particular, we stress the fact that the results obtained on the sphere are refinements of those on general surfaces of revolution, and explain the similarities in the case of the disk. 

\begin{remark}[Agmon distance on the sphere]
\label{rksphereAgmon}
Note that the coordinates $(s, \theta)$ introduced on the unit sphere are the same as those defining general surfaces of revolution, with $L=\pi$, $s\in (0,\pi)$, $R(s) = \sin(s)$ 
and the maximum of $R$ is reached at $s_0= \frac{\pi}{2}$. 
In particular, recalling the definition of the Agmon distance in~\eqref{e:defbisdA}, we obtain, for $s\in (0,\pi)$,
\bna
d_A(s) = \left| \int_{s_0}^s \sqrt{ \frac{1}{R(y)^2}- \frac{1}{R(s_0)^2}} dy \right| = \left| \int_{\pi/2}^s \sqrt{ \frac{1}{\sin(y)^2}- 1} dy \right| = \left| \int_{\pi/2}^s \frac{\cos(y)}{\sin(y)}dy \right|  = |\log (\sin(s))|.
\ena
This can be rewritten intrinsically as 
$$
d_A(m) = - \log \big(\sin (\dist_g(m,N) )\big) , \quad m \in \S^2 ,\quad \text{(recall $\dist_g(m,N) + \dist_g(m,S) = \pi$)}.
$$
In view of this identity for the sphere, the estimates on the eigenfunctions $\psi_k$ of Theorem~\ref{t:th-sphere} can be reformulated as ($\lambda_k = k (k+1)$)
\begin{align*}
&|\psi_k (s,\theta)| = c_k e^{-k d_A(s)}  \quad \text{ for } s\in[0,\pi], k\in \N ,\\
&\|\psi_k\|_{L^2(B(N,r))}^2 = \frac{c_k^2\pi}{k+1}\frac{e^{-(2k+2)d_A(r)}}{\cos(r)} (1+R) , \quad |R|\leq \frac{\tan(r)^2}{2k+2} \quad \text{ for } r\in[0,\frac{\pi}{2}), k\in \N .
\end{align*}
These two statements (ponctual estimate and fine asymptotics of the $L^2$ norm) are much stronger than those of general result of Theorem~\ref{t:agmon-intro} on general surfaces of revolution.
\end{remark}

\begin{remark}[Agmon distance in the disk]
Recalling the definition of $d_A$ in~\eqref{e:def-dA-disk}, we have 
$\alpha'(r)=-\frac{1}{r^2}\frac{1}{\sqrt{\frac{1}{r^2}-1}}$, so that 
$$
(d_A'(r))^2=\alpha'(r)^2 \left(\frac{1}{\cosh^2(\alpha(r))}-1 \right)^2= \frac{1}{r^2}\frac{1}{1-r^2}(r^2-1)^2=\frac{1}{r^2}-1 , \quad \text{ and }\quad d_A(1)=0.
$$ 
As a consequence, $d_A$ is exactly the Agmon distance to the boundary $r=1$, and we have 
\bna
d_A'(r) = - \sqrt{\frac{1}{r^2}-1}, \quad r \in (0,1] .
\ena
Note again that $d_A(r) \sim_{r\to 0^+} \log(\frac{1}{r})$ and, in particular, the center of the disk is at infinite Agmon distance to the boundary: $d_A(0) = +\infty$.
\end{remark}

\subsubsection{Uniform Lebeau-Robbiano spectral inequalities: upper bound for $\co_\Sigma$}
\label{subsub:LR-Jerison-Lebeau}
The counterpart to Corollary~\ref{c:rate-vanish} is due to Donnelly-Fefferman~\cite{DF:88}, and roughly states that eigenfunctions vanish at most like $r^{C \sqrt{\lambda}+C}$ on balls of radius 
$r$ ($\lambda$ is the eigenvalue). It has been generalized in some sense to sums of eigenfunctions by Jerison and Lebeau~\cite{JL:99}.
We prove here a variant of this result under the form of a uniform Lebeau-Robbiano spectral inequality with observation on small balls.

\begin{theorem}[Uniform Lebeau-Robbiano spectral inequality with observation on small balls]
\label{t:unif-LR-ineq}
Let $(\M,g)$ be a compact Riemannian manifold with (or without) boundary $\d \M$. For all $x_0\in \M$, there exist constants $C_1,C_2>0$ such that for all $r>0$, $\lambda \geq 0$ and $\psi \in E_{\leq\lambda}$, we have
$$
\| \psi\|_{L^2(\M)} \leq e^{\left(C_1\sqrt{\lambda} + C_2\right) \left(1+\log(\frac{1}{r})\right)} \|\psi \|_{L^2(B(x_0,r))} .
$$
\end{theorem}
Note that a careful inspection of the proofs (of all Carleman estimates used, that are stable by small perturbations) shows that the constant $C_1,C_2$ can actually be taken independent of the point $x_0$.
Note that we prove the result in the context of a Lipschitz metric $g$ and in the case of Neumann boundary conditions as well.
This uniform Lebeau-Robbiano spectral inequality directly implies Theorem~\ref{t:heatlog} using~\cite[Corollary~1]{Miller:10} (recalled in Lemma~\ref{lmspectraldonneheat} below).

\bigskip
One of the tools we develop  for the proof of Theorem~\ref{t:heatlog} also yields a uniform Lebeau-Robbiano in a class of Lipschitz metrics. Even though not completely related to the main results of the paper, we choose to state is here since we believe it is of independent interest.

On the manifold $\M$, we denote here by $\mathfrak{g}$ a metric and  $(\lambda_j^\mathfrak{g})_{j \in \N}$ the spectrum of the associated Laplace-Beltrami operator $-\Delta_\mathfrak{g}$ (with Dirichlet boundary condition if $\d \M \neq \emptyset$) and by $(\psi_{\lambda_j}^\mathfrak{g})_{j \in \N}$ an associated Hilbert basis of eigenfunctions, in order to stress the dependence with respect to the metric. We also write
$$
E_{\leq\lambda}^\mathfrak{g} = \vect \{\psi_{\lambda_j}^\mathfrak{g}, \lambda_j^\mathfrak{g} \leq \lambda\} ,
$$
which of course, depends on the metric $\mathfrak{g}$.
 Now, given a reference Lipschitz metric $\mathfrak{g}_0$, we define
 $$
 \Gamma_{\eps, D}(\M, \mathfrak{g}_0) = \left\{ \mathfrak{g} \text{ Lipschitz continuous metric on } \M, \quad \|\mathfrak{g}\|_{W^{1,\infty} (\M)} \leq D, \quad \eps \mathfrak{g}_0 \leq \mathfrak{g} \leq D \mathfrak{g}_0\right\} .
$$
 
\begin{theorem}[Uniform Lebeau-Robbiano spectral inequality in a class of metrics]
\label{t:uniform-LR-ineq-metric}
Let $\M$ be a compact Riemannian manifold with (or without) boundary $\d \M$, $\mathfrak{g}_0$ be a Lipschitz continuous Riemannian metric on $\M$, and $\omega \subset \M$ a nonempty open set. Then, for all $D\geq\eps>0$, there exist constants $C, c>0$ such that for all $\mathfrak{g} \in \Gamma_{\eps, D}(\M, \mathfrak{g}_0)$, $\lambda \geq 0$ and $w \in E_{\leq\lambda}^{\mathfrak{g}}$, we have
\begin{equation}
\label{e:unif-LR-ineq-metric}
\| w \|_{L^2(\M)} \leq Ce^{c \sqrt{\lambda}} \| w \|_{L^2(\omega)} .
\end{equation}
\end{theorem} 
Note that the above estimate is valid whatever the choice of $L^2$-norm (i.e. w.r.t. $\mathfrak{g}$ or $\mathfrak{g}_0$) since all these norms are uniformly equivalent for metrics $\mathfrak{g}$ the class $\Gamma_{\eps, D}(\M, \mathfrak{g}_0)$.
This result could be reformulated by saying that~\eqref{e:unif-LR-ineq-metric} holds for all $w \in \bigcup_{\mathfrak{g} \in \Gamma_{\eps, D}(\M, \mathfrak{g}_0)}  E_{\leq\lambda}^{\mathfrak{g}}$.

This uniform Lebeau-Robbiano spectral inequality directly implies the following uniform estimate on the cost of the heat equation, using~\cite[Corollary~1]{Miller:10}, recalled in Lemma~\ref{lmspectraldonneheat} below (in which the constants are explicitely computed in terms of the constants in the spectral inequality).
\begin{corollary}
Let $\M$ a compact Riemannian manifold with (or without) boundary $\d \M$, $\mathfrak{g}_0 \in \mathcal{T}^2_{W^{1,\infty}(\M)}$ be a Riemannian metric on $\M$, and $\omega \subset \M$ a nonempty open set. Then, for all $D\geq\eps>0$, there exist constants $C, \mathfrak{K}>0$ such that for all $\mathfrak{g} \in \Gamma_{\eps, D}(\M, \mathfrak{g}_0)$, we have
\bna
 \nor{e^{T\Delta_{\mathfrak{g}}}u}{L^2(\M)}^2  \leq Ce^{\frac{2\mathfrak{K}}{T}} \int_0^T \nor{e^{t\Delta_{\mathfrak{g}}}u}{L^2(\omega)}^2 dt , \quad \text{ for all $T>0$ and all $u \in L^2(\M)$.}
 \ena
\end{corollary}

\subsubsection{The case of a barrel: upper bound for $\co_{wave}$ and $\co_{heat}$}
\label{subsubbarrel}
To conclude with the upper bounds on the constant, we present in this section some applications of results obtained by Allibert in \cite{Allibert:98}.
In case of a ``barrel-type surface'' with boundary (a geometric setting close to that of surfaces of revolution described above), Allibert estimates the att
tainable space for the controlled wave equation. As corollaries, we deduce from this result estimates of $\co_{wave}$ and, in view of Proposition~\ref{propheatwave}, of $\co_{heat}$.

We first present the geometric context which (quite similar to the one of surface of revolution described above). 
In this section, $\M=\calS$ is a surface of revolution of $\R^{3}$ with boundary, parametrized by the equation $$\calS = \{ (x,y,z)\in \R^3 , z\in [a,b] , x^{2}+y^{2}=\mathsf{R}(z)\} , \quad $$ where $\mathsf{R}$ is a strictly positive smooth function on $[a,b]$, that admits a unique local (and therefore global) non degenerate maximum (i.e. $\mathsf{R}''(c)<0$) at one point $c \in (a,b)$. The control is a boundary control from the  bottom side, that is $\Gamma=\{(x,y,a)\in \R^{3}; x^{2}+y^{2}=\mathsf{R}(a) \}$. We also describe $\calS$ by $(z,\theta)$, with $(x,y) =( \mathsf{R}(z)\cos\theta ,\mathsf{R}(z)\sin\theta ) $

We refer to Remark~\ref{rkautreparam} to explain the link between the two parametrizations of revolution surfaces by $s$ and $z$ (and in particular, that we may write $z=z(s)$ and $R(s)= \mathsf{R}(z(s))$).

As before, we define the Agmon distance to the point $c$. With the parametrization of the embedding into $\R^{3}$, it gives the following definition (note that it is almost the same as \eqref{e:defbisdA} but in different coordinates):
\bna
d_{A}(z)=\left| \int_{c}^z \sqrt{1+\mathsf{R}'^{2}(y)}\sqrt{ \frac{1}{\mathsf{R}(y)^2}- \frac{1}{\mathsf{R}(c)^2}} dy \right|  .
\ena 
We also need the following definition of a critical time $T_{1}$ (see Allibert \cite{Allibert:98} for more details), which, roughly speaking, represents the smallest period of the geodesic flow, modulo rotation. More precisely, the principal symbol of the wave operator on $\R\times \calS$ is given by
\bna
p(t,z,\theta,\tau,\zeta,\eta)=\frac{\zeta^{2}}{1+\mathsf{R}'^{2}(z)}+\frac{\eta^{2}}{\mathsf{R}^{2}(z)}-\tau^{2} ,
\ena
where $(\tau,\zeta,\eta)$ denote the dual variable to $(t,z,\theta)$.
For any bicaracteristic $\gamma$ of $p$, bouncing on the boundary according to the reflection law $\zeta \rightarrow -\zeta$, we denote $T(\gamma)$ the  smallest period of the function $\Pi_{z}(\gamma)$ where $\Pi_{z}$ is the projection on the component $z$.

Then, $T_{1}$ is defined by 
\bna
T_{1}=\sup_{\gamma \textnormal{ bicar}}T(\gamma), 
\ena
and we have $T_{1}\geq 2\L(\M,\Gamma)$ (this critical time is larger than the time of unique continuation from $\Gamma$).

In this context, we define similarly $\co_{heat} (\Gamma)$ and $\co_{wave} (\Gamma,T)$ with exactly the same definition as in \eqref{defcoheat} and Definition \ref{def-coco} with $\nor{u}{L^2([0,T]\times \omega)}$ replaced by $\nor{\partial_{\nu}u}{L^2([0,T]\times \Gamma)}$ in \eqref{e:co-heat} and \eqref{e:co-wave}. Note that $\partial_{\nu}u$ is in $L^2([0,T]\times \Gamma)$ for initial data in $L^2$ (resp. $H^1_0\times L^2$) for the heat (resp. wave) equation thanks to hidden regularity. We deduce from~\cite{Allibert:98} the following result.

\begin{theorem}
\label{t:positrev}
Under the above geometric assumptions, we have the estimates
\begin{align}
&\co_{wave} (\Gamma,T)\leq d_A(\Gamma), \quad \text{ for all }  T>T_1, \label{Allibertwave}\\
&\co_{heat}(\Gamma)\leq \alpha (T_1(\Gamma)^2+ d_A(\Gamma)^{2}) , \label{Allibertheat}
\end{align}
for some universal constant $\alpha>0$.
\end{theorem}
The first estimate~\eqref{Allibertwave} follows simply from \cite[Th\'eor\`eme~2]{Allibert:98} (see Proposition \ref{lienanalyticekwave} below), which is stated in terms of analytic spaces with respect to the rotation variable $\theta$. Then, \eqref{Allibertwave} implies \eqref{Allibertheat} thanks to Proposition \ref{propheatwave}.
Note that~\eqref{Allibertwave} also proves an analogue of Theorem~\ref{t:agmon-intro} in this geometry, so that in fact:
\bnan
\label{e:estim-=Allibert}
\co_{eig} (\Gamma)  = d_A(\Gamma), \quad \text{ and }\quad \co_{wave} (\Gamma,T) = d_A(\Gamma),  \quad \text{ for all }  T>T_1 .
\enan
He also proves upper and lower estimates for $T \in (2\L(\M,\Gamma), T_1)$ (which do not coincide). The proof of Theorem~\ref{t:positrev} in Proposition \ref{lienanalyticekwave} yields the according estimates of $\co_{wave} (\Gamma,T)$.

\subsection{Previous results}
\label{sectpreviousres}
Except for the bounds \eqref{e:estim-=Allibert} following from Allibert's result and the computation of $\co_\infty(\{0\})$ on $(0,L)$ in \cite{FR:71}, we are not aware of other situations in which the constants described in the previous paragraph are known exactly. 
We collect in this section previous results on the constant $\co_{heat}$ and $\co_{wave}$, which received a lot of attention in the past fifteen years. 

\paragraph{Parabolic equations in dimension one}
The most studied case concerns the constant $\co_{heat}$, with observation/control at the boundary in the one dimensional case. 
Yet, it seems that the constant $\co_{heat} (\{-1,1 \})$ is still unknown.
Note that the latter has a particular importance since it has applications to higher dimensions (with geometric conditions) via the transmutation method of Luc Miller~\cite{Miller:06b}. 

Here, we list previous results on $(-1,1)$ with Dirichlet control on both side of the interval. Note also that each improvement of the constant was also the occasion of finding new techniques.
\begin{itemize}
\item $\co_{heat} (\{-1,1 \}) \leq 2 \left(\frac{36}{37} \right)^2 $ Miller \cite{Miller:06b}, using the transmutation method;
\item $\co_{heat} (\{-1,1 \}) \leq \frac{3}{4}$ Tenenbaum-Tucsnak \cite{TT:07}, using some results of analytic number theory;
\item $\co_{heat} (\{-1,1 \})\geq \frac12$, Lissy \cite{Lissy:15}, using complex analysis arguments;
\item $\co_{heat} (\{-1,1 \}) \leq 0,7$, Dard\'e-Ervedoza \cite{Darde:17}, combining some Carleman estimates and complex analysis.
\end{itemize}
Note that in this context, the analogue of Conjecture \ref{c:Miller-conj} would be $\co_{heat} (\{-1,1 \})=\frac{1}{4}$, which \cite{Lissy:15} disproved in this context (by a factor $2$). However, this result does not prevent the existence of a universal constant $C>0$ so that $\co_{heat} (\omega) = C \L(\M,\omega)^2$. 

As noticed in~\cite{EZ:11s}, the result in~\cite{FR:71} implies that on the interval $(0,L)$, we have $\co_\infty(\{0\}) = \frac{L^2}{4}$ (and~\cite{EZ:11s} even prove~\eqref{e:co-infty} for the critical $\co=\frac{L^2}{4}$).
\paragraph{Parabolic equations in higher dimensions}

There are many papers concerning the control of the heat equation. We give here a short presentation of those giving some estimates on the constants studied in this paper.

The first computable estimates were obtained using the transmutation method to give estimates similar to \eqref{upperboundMiller}. We can find several references improving the universal constant involved: \cite{Miller:04,Miller:06b} \cite{TT:07}, \cite{Darde:17}.

In \cite{TT:07}, the authors prove $\co_\Sigma(\omega^*) \leq 3\log (\frac{(4 \pi e)^N}{|\omega^*|})$ where $\M = (0,\pi)^N$ is a cubic domain and $|\omega^*|$ is the volume of the biggest rectangle included in $\omega$. The proof of this result uses number theoretic argument of Tur\'an concerning families of the complex exponential $(e^{ikx})_{k \in \Z}$ (which can be interpretated as an estimate of $\co_{\Sigma}(I)$ for $I$ a subinterval of $\T$). Remark that in this particular flat-torus geometry, we have no idea of what the right constant should be.

In \cite{BP:17}, the authors prove $\co_\Sigma(B(0,r)) \leq \frac{C_\eps}{r^\eps}$ for all $\eps>0$ in convex geometries. This has just been extended by Phung \cite{PhunglogCarl}. Our Theorem \ref{t:heatlog} improves this result. Note also that \cite{NTTV:18} gave results related to this in a periodic setting, tracking uniformity with respect to several parameters. 

In the Euclidian  space $\R^n$ where $\Delta$ is the usual flat Laplacian, spectral estimates as \eqref{e:co-sum} can be interpretated as a manifestation of the uncertainty principle. Several results relying on this fact have been recently stated. We refer for instance to the recent articles~\cite{EV:17} and \cite{WWZZ:17} and the references therein. 

\paragraph{The wave equation}
Lebeau \cite{Leb:Analytic} proved in the analytic setting that $\co_{wave}(\omega,T)$ is finite for any open set $\omega$ and in optimal time $T>2\L(\M,\omega)$ (the result is stated in a quite different way actually). It was only very recently shown to be finite by the authors \cite{LL:15} in a general $C^\infty$ context. We refer the reader to the introduction of~\cite{LL:15} for a detailed discussion of the literature on unique continuation for waves, and estimates like~\eqref{e:co-wave}-\eqref{e:co-wave-bis}.

Estimates on analytic spaces of controllable data were computed by Allibert in the above described examples. We refer to Section \ref{subsectionanalytlink} for more details about why they have implications on the constant $\co_{wave}$ (and therefore $\co_{heat}$ by Proposition \ref{propheatwave}). In \cite{Allibert:98}, he studies the example of the barrel as we describe it in Section \ref{subsubbarrel}. In \cite{Allibert:99}, he studies the example of a cylinder $(0,\pi)\times \S^1 $. The results he obtain in that paper should imply $\co_{wave}(\Gamma,T)\leq \frac{C_{\delta}}{T^{1-\delta}}$ where $\Gamma=\{0\}\times \S^1$ and $T> 2\pi$.

\subsection{Plan of the paper}

The paper is divided in four main parts. In Section~\ref{s:prelim}, we give the links between the different constants, proving in particular Propositions~\ref{p:link-eigenfct-heat-etc} and~\ref{propheatwave}. We also interpret the description of the reachable set as an upper bound on the constant $\co_{wave}(\omega,T)$.

In Section~\ref{s:contruction}, we construct the various counterexamples on rotationally invariant geometries, presented in Section~\ref{s:intro-const}. This proves in particular Theorem~\ref{thm:counterexamples}.

Section~\ref{s:unif-LR-ineq} is devoted to the proof of the uniform Lebeau-Robbiano inequality on small balls, stated in Theorem~\ref{t:unif-LR-ineq}.

Finally, we prove in Section~\ref{s:positive} the observability inequality of Theorem~\ref{thmpositive} concerning positive solutions of the heat equation.

The paper ends with two appendices, in the first of which, Appendix~\ref{s:carleman}, we prove a uniform Carleman estimate for bounded families of Lipschitz metrics. Such an estimate is used as an intermediate in the proof of Theorem~\ref{t:unif-LR-ineq}. The result also yields Theorem~\ref{t:uniform-LR-ineq-metric}.

Note finally that in a companion paper \cite{LL:18vanish}, we will use similar techniques to disprove natural conjectures for the control cost of transport equations in the vanishing viscosity limit.

\bigskip
\noindent
{\em Acknowledgements.} 
The first author is partially supported by the Agence Nationale de la Recherche under grant  SRGI ANR-15-CE40-0018 and IPROBLEMS ANR-13-JS01-0006.
The second author is partially supported by the Agence Nationale de la Recherche under grant GERASIC ANR-13-BS01-0007-01.
Both authors are partially supported by the Agence Nationale de la Recherche under grants ISDEEC ANR-16-CE40-0013.
Part of this research was done when the second author was in CRM, CNRS UMI 3457, Universit\'e de Montr\'eal, and Universit\'e Paris Diderot, IMJ-PRG, UMR 7586.

\section{Preliminaries: links between the different constants}
\label{s:prelim} 
\subsection{Different definitions of \texorpdfstring{$\co_{wave} (\omega,T)$}{kwave(omega,T)}}
Let us start by proving equality~\eqref{e:co=co'}. This is a consequence of the following lemma.
\begin{lemma}
\label{l:lambda-mu-co-co}
Let $\mu_0\geq 0$, $\co \geq 0$ and assume that $\Lambda > 0$ and $X\geq 0$ satisfy
\bnan
\label{e:asspt-Lambda-mu}
\frac{1}{\Lambda} \leq e^{\co \mu} X + \frac{1}{\mu} , \quad \text{for all }\mu > \mu_0 .
\enan
Then, for all $\alpha >0$, we have 
\bnan
\label{e:conclu-Lambda-mu}
1 \leq \left( \mathds{1}_{\Lambda + \alpha \leq \mu_0} \frac{\mu_0 - \alpha}{\alpha} e^{ \co \mu_0}  + \mathds{1}_{\Lambda + \alpha> \mu_0}\frac{e^{\co \alpha}}{\alpha}\Lambda(\Lambda+\alpha)e^{\co \Lambda} \right)X .
\enan
Let $F : \R^+ \to \R^+$ be a nondecreasing function and assume that $\Lambda > 0$ and $X\geq 0$ satisfy
\bnan
\label{e:asspt-Lambda-mu-bis}
\Lambda \geq 1 \quad \text{and} \quad  1 \leq  F(\Lambda) X .
\enan
Then, we have 
\bnan
\label{e:conclu-Lambda-mu-bis}
\frac{1}{\Lambda} \leq F(\mu) X + \frac{1}{\mu} , \quad \text{for all }\mu > 0.
\enan
\end{lemma}
As a direct consequence of this lemma, we obtain the following corollary, clarifying the definition of $\co_{wave} (\omega,T)$.
\begin{corollary}
\label{c:lambda=mu}
Assume~\eqref{e:co-wave} with constants $\co, C, \mu_0>0$. Then, there is $C''>0$ such that
\bna
 \|(u_0,u_1)\|_{H^1_0(\M)\times L^2(\M)}\leq C''\Lambda^2 e^{\co\Lambda}\|u\|_{L^2((0,T)\times\omega)},\quad\Lambda =\frac{\|(u_0,u_1)\|_{H^1_0(\M)\times L^2(\M)}}{\|(u_0,u_1)\|_{L^2(\M) \times H^{-1}(\M)}},\nonumber \\ \text{ for all } (u_0, u_1) \in H^1_0(\M) \times L^2(\M), \text{ and }u \text{ solution to~\eqref{e:wave},}
\ena
\bna
\|(u_0,u_1)\|_{L^2(\M) \times H^{-1}(\M)} \leq C'' \mu^2 e^{\co \mu} \|u\|_{L^2((0,T)\times\omega)} + \frac{1}{\mu} \|(u_0,u_1)\|_{H^1_0(\M)\times L^2(\M)}, \nonumber \\ \text{ for all $\mu > 0$ and all $(u_0, u_1) \in H^1_0(\M) \times L^2(\M)$, and $u$ solution to~\eqref{e:wave}.}
\ena
Reciprocally, if~\eqref{e:co-wave-bis} holds with constants $\co', C'>0$, then~\eqref{e:co-wave} holds with $\co=\co'$, $C=C'$, and $\mu_0=0$ (and for all $\mu>0$).

In particular, we have
\begin{align*}
\co_{wave} (\omega,T) &=  \inf \left\{ \co >0 , \exists C>0, \mu_0>0  \text{ s.t. \eqref{e:co-wave} holds} \right\} \\
&=  \inf \left\{ \co' >0 , \exists C'>0,  \text{ s.t. \eqref{e:co-wave-bis} holds} \right\} \\
& = \inf \left\{ \co >0 , \exists C>0,  \text{ s.t. \eqref{e:co-wave} holds with } \mu_0 = 0 \text{ (and all $\mu>0$)}\right\} .
\end{align*}
\end{corollary}

\bnp[Proof of Lemma~\ref{l:lambda-mu-co-co}]
Let $\alpha>0$.
In case $\Lambda+\alpha > \mu_0$, the assumption~\eqref{e:asspt-Lambda-mu} with $\mu = \Lambda+\alpha>\mu_0$ yields
$$
\frac{1}{\Lambda}\left(1 - \frac{\Lambda}{\Lambda+\alpha} \right) \leq e^{\co (\Lambda + \alpha)} X ,
$$
and hence 
\bnan
\label{e:Lambda-part1}
1 \leq \frac{1}{\alpha}e^{\co \alpha}\Lambda(\Lambda+\alpha)e^{\co \Lambda} X .
\enan
If now $\Lambda+\alpha \leq \mu_0$ (and, in particular, $\alpha < \mu_0$), that is $\frac{1}{\Lambda} \geq \frac{1}{\mu_0-\alpha}>0$, the assumption~\eqref{e:asspt-Lambda-mu} implies 
$$
 \frac{1}{\mu_0-\alpha} \leq \frac{1}{\Lambda} \leq e^{\co \mu} X + \frac{1}{\mu}, \quad \text{for all }\mu \geq \mu_0 .
$$
This yields in particular 
$$
X \geq \left( \frac{1}{\mu_0-\alpha}- \frac{1}{\mu} \right) e^{- \co \mu} , \quad \text{for all }\mu \geq \mu_0 , 
$$
and hence $X \geq \max_{\mu \geq \mu_0} \left( \frac{1}{\mu_0-\alpha}- \frac{1}{\mu} \right) e^{- \co \mu} \geq \frac{\alpha}{\mu_0 - \alpha} e^{- \co \mu_0}  >0$. With~\eqref{e:Lambda-part1}, this proves~\eqref{e:conclu-Lambda-mu}.

Let us now prove~\eqref{e:conclu-Lambda-mu-bis}. If $\Lambda \geq \mu$, then $\frac{1}{\Lambda} \leq \frac{1}{\mu}$ and~\eqref{e:conclu-Lambda-mu-bis} holds. If $\Lambda \leq \mu$, then~\eqref{e:asspt-Lambda-mu-bis} gives $\frac{1}{\Lambda} \leq 1 \leq  F(\Lambda)X \leq  F(\mu)X$ and \eqref{e:conclu-Lambda-mu-bis} also holds in this case, concluding the proof.
\enp

\subsection{{The constant \texorpdfstring{$\co_{eig}(\omega)$}{Keig(omega,T)} as a lower bound for \texorpdfstring{$\co_{heat}(\omega),\co_{\infty}(\omega),\co_{wave}(\omega,T)$}{Kheat(omega),Kinfty(omega),Kwave(omega,T)}: Proof of Proposition~\ref{p:link-eigenfct-heat-etc}}}
We prove a slightly more precise version of Proposition~\ref{p:link-eigenfct-heat-etc}.
\begin{lemma}
Assume that~\eqref{e:co-heat} holds with constants $\co, C >0$. Then, we have 
\bnan
\label{e:comp-co-co-heat}
 \nor{\psi}{L^2(\M)}  \leq \sqrt{\frac{C}{2\lambda}} e^{2\sqrt{\co\lambda}}\nor{\psi}{L^2(\omega)} , \quad \text{for all $\lambda \in \Sp(-\Delta_g) \setminus\{0\}$ and $\psi \in E_\lambda$}.
 \enan
In particular, 
\bnan
\label{e:comp-co-co-heat-infimum}
\frac{\co_{eig}(\omega)^2}{4} \leq \co_{heat}(\omega) .
\enan
Assume that~\eqref{e:co-infty} holds with constants $\co, C >0$. Then, there exists $C''>0$ such that  
\bnan
\label{e:comp-co-co-infty}
 \nor{\psi}{L^2(\M)}  \leq \frac{C''}{\lambda^{1/8}}  e^{2\sqrt{\co\lambda}}\nor{\psi}{L^2(\omega)} , \quad \text{for all $\lambda \in \Sp(-\Delta_g) \setminus\{0\}$ and $\psi \in E_\lambda$}.
 \enan
In particular
\bnan
\label{e:comp-co-co-infty-infimum}
\frac{\co_{eig}(\omega)^2}{4} \leq \co_{\infty}(\omega) .
\enan
Assume that \eqref{e:co-wave-bis} holds in time $T$ with constants $C',\co'$. Then, we have 
\bnan
\label{e:wave-eig-idem}
\|\psi\|_{L^2(\M)} \leq \sqrt{\frac{T}{\lambda}} C' e^{\co' \sqrt{\lambda}}\|\psi\|_{L^2(\omega)}, \quad \text{for all $\lambda \in \Sp(-\Delta_g) \setminus\{0\}$ and $\psi \in E_\lambda$}.
\enan
In particular, for all $T>0$, we have $\co_{eig}(\omega)\leq \co_{wave}(\omega,T)$.
\end{lemma}

\bnp[Proof of Proposition~\ref{p:link-eigenfct-heat-etc}]
From~\eqref{e:co-heat}, applied to $u(t,x) = e^{-t\lambda}\psi$ with $\lambda \in \Sp(-\Delta_g) \setminus\{0\}$ and $\psi \in E_{\lambda}$, we have 
$$
e^{-2T\lambda}\nor{\psi}{L^2(\M)}^2  \leq Ce^{\frac{2\co}{T}} \int_0^T e^{-2t\lambda}\nor{\psi}{L^2(\omega)}^2 dt = Ce^{\frac{2\co}{T}} \frac{1-e^{-2T\lambda}}{2\lambda} \nor{\psi}{L^2(\omega)}^2 , \quad \text{ for all } T>0 .
$$
Taking $T = \frac{D}{\sqrt{\lambda}}$, with $D>0$ to be chosen, this implies 
$$
\nor{\psi}{L^2(\M)}^2  \leq Ce^{2T\lambda}e^{\frac{2\co}{T}}   \frac{1}{2\lambda} \nor{\psi}{L^2(\omega)}^2 = \frac{C}{2\lambda} e^{2\sqrt{\lambda}(D +\frac{\co}{D})}  \nor{\psi}{L^2(\omega)}^2 .
$$
Minimizing the exponent with respect to $D$ leads to choosing $D = \sqrt{\co}$, which implies~\eqref{e:comp-co-co-heat} when taking the square root. From~\eqref{e:comp-co-co-heat}, \eqref{e:comp-co-co-heat-infimum} follows directly when taking the infimum over all $\co$.

Let us now prove the second statement of the proposition. From~\eqref{e:co-infty}, again applied to $u(t,x) = e^{-t\lambda}\psi$ with $\lambda \in \Sp(-\Delta_g) \setminus\{0\}$ and $\psi \in E_{\lambda}$, we have 
\bnan
\label{E-Z-bis}
\int_{\R^+} e^{-\frac{2\co}{t}} e^{-2t\lambda}\nor{\psi}{L^2(\M)}^2  dt \leq C\int_{\R^+}e^{-2t\lambda}\nor{\psi}{L^2(\omega)}^2  dt  = \frac{C}{2\lambda} \nor{\psi}{L^2(\omega)}^2 .
\enan
The left hand-side may also be computed asymptotically for $\lambda \to + \infty$ using Laplace method, setting $\mu = \sqrt{\lambda}$, as
\bna
\int_{\R^+} e^{-\frac{2\co}{t}}e^{-2\mu^2 t} dt 
& = & \int_{\R^+} e^{-2\sqrt{\co}\mu( \frac{1}{s} + s)} \frac{\sqrt{\co}}{\mu} ds \\
& = & (1+o(1))  \frac{\sqrt{\co}}{\mu}  \int_\R e^{-2\sqrt{\co}\mu(2 + (s-1)^2 )} ds \\
& = & (1+o(1))  \frac{\sqrt{\co}}{\mu}  e^{-4\sqrt{\co} \mu} \sqrt{\frac{\pi}{2\sqrt{\co}\mu}}
 = (1+o(1)) \left( \frac{\pi \sqrt{\co}}{2 \mu^3} \right)^{\frac12}e^{-4\sqrt{\co} \mu} .
\ena
From~\eqref{E-Z-bis}, we then obtain that, for all eigenfunction $\psi$ associated to the eigenvalue $\mu^2$, for $\mu \to \infty$, we have
\bna
(1+o(1)) \left( \frac{\pi \sqrt{\co}}{2 \mu^3} \right)^{\frac12}e^{-4\sqrt{\co} \mu}  \nor{\psi}{L^2(\M)}^2   \leq \frac{C}{2\mu^2}\|\psi \|_{L^2(\omega)}^2 .
\ena
Coming back to $\lambda=\mu^2$, this implies that the existence of $\tilde{C}, \lambda_0>0$ such that for all $\lambda \geq \lambda_0$ 
\bna
 \nor{\psi}{L^2(\M)}^2   \leq \frac{\tilde{C}}{\lambda^{1/4}} e^{4\sqrt{\co \lambda}} \|\psi \|_{L^2(\omega)}^2 ,
\ena
and hence the sought result of~\eqref{e:comp-co-co-infty}. That of~\eqref{e:comp-co-co-infty-infimum} follows as above.

Let us now prove the last statement of the proposition. We want to apply~\eqref{e:co-wave-bis} to the function $u(t,x) = \cos(t \sqrt{\lambda}) \psi$ with $\lambda \in \Sp(-\Delta_g) \setminus\{0\}$ and $\psi \in E_{\lambda}$, which is a particular solution to~\eqref{e:wave}. We have 
$\Lambda = \frac{\|(u|_{t=0},\d_t u|_{t=0})\|_{H^1_0(\M)\times L^2(\M)}}{\|(u|_{t=0},\d_t u|_{t=0})\|_{L^2(\M) \times H^{-1}(\M)}} = \frac{\|\psi\|_{H^1_0(\M)}}{\|\psi\|_{L^2(\M)}}= \sqrt{\lambda}$ together with 
$$
\sqrt{\lambda} \|\psi\|_{L^2(\M)} = \|\psi\|_{H^1_0(\M)}  =\|u|_{t=0},\d_t u|_{t=0})\|_{H^1_0(\M)\times L^2(\M)}\leq C' e^{\co' \Lambda} \|u\|_{L^2((0,T)\times\omega)},  
 $$
where 
$$ 
\|u\|_{L^2((0,T)\times\omega)}^2 = \int_0^T \cos^2(t\sqrt\lambda) \|\psi\|_{L^2(\omega)}^2 dt \leq T\|\psi\|_{L^2(\omega)}^2 .
$$
This finally implies~\eqref{e:wave-eig-idem}. The last result follows from Corollary~\ref{c:lambda=mu}. This concludes the proof of the proposition.
\enp

\subsection{Link between \texorpdfstring{$\co_{heat} (\omega)$}{Kheat(omega)} and \texorpdfstring{$\co_{wave}(\omega,T)$}{Kwave(omega,T)}: Proof of Proposition \ref{propheatwave}}
\label{s:proof-wave-heat}
The proof will follow very closely the method of Ervedoza-Zuazua \cite{EZ:11}, but with a different assumption. We summarize the results of~\cite{EZ:11,EZ:11s} we need in the next proposition for readibility.
 \begin{proposition}[\cite{EZ:11,EZ:11s}]
\label{propoEZ}
Let $T,S>0$ and $\alpha > 2 S^2$. Let $\L$ be a negative self adjoint operator. Then, there exists some kernel function $k_T(t,s)$ such that
\begin{itemize}
\item
 if $\y$ is solution of the heat equation~$\partial_s\w -\L \w=0$, then $\w(s)=\int_0^T k_T(t,s)\y(t)dt$ is solution of 
\bneqn
\label{wave-w}
\partial_s^2\w -\L \w&=&0, \quad \text{ for } s \in ]-S,S[ ,  \\
(\w,\partial_s\w)|_{s=0}&=&\left(0,\int_0^T \d_s k_T (t,0) \y(t)dt\right) = \left(0,\int_0^T e^{-\alpha \left(\frac{1}{t}+\frac{1}{T-t}\right)}\y(t)dt\right) ;
\eneqn
\item for all $\delta \in ]0,1[$ and  all $(t, s) \in ]0,T[ \times ]-S,S[$, $k_T$ we have
\bnan
\label{e:estim-kT}
|k_T(t,s)|\leq  |s|\exp\left( \frac{1}{\min \left\{ t,T-t\right\}}\left(\frac{s^2}{\delta}-\frac{\alpha}{(1+\delta)}\right)\right) .
\enan
\end{itemize}
\end{proposition}
Note that this last estimate is most useful for $\delta$ sufficiently close to one so that $\alpha \geq S^2(1+\frac{1}{\delta})$.
We first prove the spectral observability property. 
\bnp[Proof of Proposition \ref{propheatwave}]
To simplify notations, we prove the existence of universal constants so that $\co_{heat} (\omega)\leq \alpha_{3}S^2 + \alpha_{4}\co_{wave}(\omega,2S)^{2}$ for all $S>0$. 

Let $C_0> \co_{wave}(\omega,2S)$ so that there exists $C>0$ so that we have the estimate (see Corollary \ref{c:lambda=mu} for the equivalence)
\bnan
\label{e:co-wave-bisspectral}
 \|(u_0,u_1)\|_{H^1_0(\M)\times L^2(\M)}\leq C e^{C_0 \Lambda} \|u\|_{L^2((-S,S)\times\omega)}, \quad \Lambda = \frac{\|(u_0,u_1)\|_{H^1_0(\M)\times L^2(\M)}}{\|(u_0,u_1)\|_{L^2(\M) \times H^{-1}(\M)} }, \nonumber \\ \text{ for all $(u_0, u_1) \in H^1_0(\M) \times L^2(\M)$, and $u$ solution to~\eqref{e:wave}.}
\enan
Note that when compared to \eqref{e:co-wave}, we have changed the interval $(0,2S)$ to $(-S,S)$ which gives the same result by conservation of energy.
 
The proof is a direct consequence of Lemma \ref{lmwavedonnespectral} and Lemma \ref{lmspectraldonneheat} below that we state separately since they have their own interest.
\enp
 \begin{lemma}
 \label{lmwavedonnespectral}
 Assume \eqref{e:co-wave-bisspectral}, then, we have
\bna
 \nor{e^{T\Delta_g}\y_0}{L^2(\M)}^2  \leq \frac{C (1+\lambda)S^{2}e^{2C_0(1+\lambda)^{\frac{1}{2}}}}{T}e^{\frac{18S^2}{T}}\int_0^T \nor{e^{t\Delta_g}\y_0}{L^2(\omega)}^2 dt , \quad \text{ for all $0<T\leq \alpha$ and all $\y_0 \in E_{\leq \lambda}$.}
 \ena
 \end{lemma}
 \bnp
 We pick $\alpha > 2S^2$ and use the kernel $k_{T}$ described in Proposition \ref{propoEZ}. 
 
Assume now that $\w(s)$ is associated to $\y$ as $\w(s)=\int_0^T k_T(t,s)\y(t)dt$, where $\y=e^{t\Delta_g}\y_0$ with $\y_0 \in E_{\leq \lambda}$. Then, in~\eqref{wave-w}, $\W_0$ is of the particular form $\W_0 =  \left(0,\int_0^T e^{-\alpha \left(\frac{1}{t}+\frac{1}{T-t}\right)}\y (t)dt\right)$, so that a calculation (see~\cite[Equation~(3.3)]{EZ:11}) yields
\bnan
\label{e:brutale-est}
\nor{\W_0}{L^2\times \H^{-1}_\L}^2 & \geq & (1+ \lambda)^{-1}\nor{\W_0}{\H^1_\L \times L^2}^2 
= (1+ \lambda)^{-1}  \nor{\int_0^T e^{-\alpha \left(\frac{1}{t}+\frac{1}{T-t}\right)}\y (t)dt}{L^2}^2 \nonumber \\
&\geq& (1+ \lambda)^{-1}  \sum_i|y_i|^2e^{-2\lambda_i T}\left|\int_0^T e^{-\alpha \left(\frac{1}{t}+\frac{1}{T-t}\right)}dt\right|^2\nonumber
\enan
The integral can be estimated by Laplace method
\bna
\int_0^T e^{-\alpha \left(\frac{1}{t}+\frac{1}{T-t}\right)}dt=T\int_0^1 e^{-\frac{\alpha}{T}\left(\frac{1}{s}+\frac{1}{1-s}\right)}ds\geq CT\left(\frac{T}{\alpha}\right)^{1/2}e^{-4\frac{\alpha}{T}}, \textnormal{ for  } \frac{\alpha}{T}\geq 1,
\ena
since the non degenerate minimum of $\frac{1}{s}+\frac{1}{1-s}$ is $4$ reached at $s=1/2$ and the fonction is positive.
\bnan 
\nor{\W_0}{L^2\times \H^{-1}_\L}^2& \geq & C (1+ \lambda)^{-1}T^3\alpha^{-1}e^{-\frac{8\alpha}{T}}\nor{\y(T)}{L^2}^2 .
\enan
Moreover, we have $\W_0 \in E_{\leq \lambda}\times E_{\leq \lambda}$ so that
\bna
\frac{\nor{\W_0}{H^{1}_0 \times L^2}}{\nor{\W_0}{L^2\times H^{-1}}} \leq (1+\lambda)^{\frac{1}{2}}.
\ena
As a consequence,~\eqref{e:co-wave-bisspectral} implies
\bnan
\label{observfrequence}
 \nor{\W_0}{L^2\times \H^{-1}_\L} \leq C e^{C_0(1+\lambda)^{\frac{1}{2}}}\nor{ \w}{L^2(]-S,S[\times \omega)}.
\enan
Using Cauchy-Schwarz inequality, we have
\begin{align}
\label{estimobserv}
\nor{\w}{L^2(]-S,S[\times \omega)}^2 &\leq \left( \int_{]0,T[\times ]-S,S[} k_T(t,s)^2 dt~ds \right) \int_{0}^T \int_{\omega}\left|\y (t,x)\right|^2 dx~dt 
\end{align}
Now, we use ~\eqref{e:estim-kT} with $\delta\in (0,1)$ fixed sufficiently close to one so that $\alpha \geq S^2\frac{(1+\delta)}{\delta}$ (which is possible since we have assumed $\alpha> 2S^2$). 
\bnan
\label{estimkT}
\int_{]0,T[\times ]-S,S[} k_T(t,s)^2 dt~ds \leq  CS^{2}
\int_{]0,T[\times ]-S,S[} \exp\left( \frac{1}{\min \left\{ t,T-t\right\}}\left(\frac{S^2}{\delta}-\frac{\alpha}{(1+\delta)}\right)\right) dt~ds\leq CS^{3}T .
\enan
Combining  \eqref{e:brutale-est}, \eqref{observfrequence}, \eqref{estimobserv} and \eqref{estimkT} gives the result since the estimate is true for any $\alpha >  2S^2$.
\enp

 \begin{lemma}[Miller \cite{Miller:10}]
 \label{lmspectraldonneheat}
 Assume 
 \bna
 \nor{e^{T\Delta_g}\y_0}{L^2(\M)}^2  \leq Ce^{2a\lambda^{\frac{1}{2}}+\frac{2b}{T}}\int_0^T \nor{e^{t\Delta_g}\y_0}{L^2(\omega)}^2 dt , \quad \text{ for all $0<T<T_{0}$ and all $\y_0 \in E_{\leq \lambda}$.}
 \ena
 Then, we have 
\bna
 \nor{e^{T\Delta_g}\y_0}{L^2(\M)}^2  \leq C'e^{2\frac{c^*}{T}}\int_0^T \nor{e^{t\Delta_g}\y_0}{L^2(\omega)}^2 dt , \quad \text{ for all $0<T<T_{0}$ and all $\y_0 \in L^2(\Omega)$,}
 \ena
 with $c_* =\left(a+\sqrt{b}+\sqrt{a^2+2a\sqrt{b}}\right)^2$ and $C'$ a constant dependent on $C$, $T_{0}$, $a$ and $b$.
 \end{lemma}
 \bnp
The result is not stated exactly, but the author proves this as an intermediate result of \cite[Theorem~2.2]{Miller:10}. More precisely, the assumptions of our Lemma are exactly estimate (10) in \cite{Miller:10}, with $\alpha=1/2$ and $\beta=1$.
It gives the result with $c_* = 4b^2 \left( \sqrt{a+2\sqrt{b}}-\sqrt{a}\right)^{-4} = \frac14 \left( \sqrt{a+2\sqrt{b}}+\sqrt{a}\right)^{4} = \left(a+\sqrt{b}+\sqrt{a^2+2a\sqrt{b}}\right)^2$.
 \enp
\subsection{Link between \texorpdfstring{$\co_{wave}(\omega,T)$}{Kwave(omega,T)} and some space of analytic functions}
\label{subsectionanalytlink}
As already mentioned, Theorem \ref{t:positrev} is a corollary of observability estimates in analytic spaces (characterizing the attainable set for the control problem) obtained by Allibert~\cite{Allibert:98}. The following proposition explains the link between such estimates and~\eqref{e:co-wave}-\eqref{e:co-wave-bis} (see also~\cite{Leb:Analytic}).
\begin{proposition}
\label{lienanalyticekwave}
Assume there is $C_0, C>0$ such that for all $(u_0,u_1)\in H^1_0(\M)\times L^2(\M)$ and associated $u$ solution of~\eqref{e:wave}, we have
\bnan
\label{e:obs-analytic-space}
\nor{e^{-C_0\sqrt{-\Delta_g}}(u_0,u_1)}{L^2(\M)\times H^{-1}(\M)}\leq C \|u\|_{L^2((0,T)\times\omega)}  \quad \text{(resp. }\leq C \|\partial_{\nu}u\|_{L^2((0,T)\times\Gamma)} \text{)}.
\enan
Then~\eqref{e:co-wave} is satisfied with constant $\co = C_0$ and all $\mu>0$. In particular, we have
\bna
\co_{wave}(\omega,T)\leq C_0 , \quad \text{(resp. }\co_{wave}(\Gamma,T)\leq C_0 \text{)}.
\ena
\end{proposition}
Again, in this statement, $\Delta_g$ denotes the Laplace operator with Dirichlet boundary conditions.
\bnp
Given $\mu>0$, we decompose the data $(u_0,u_1)$ as $u_0=\mathds{1}_{\sqrt{-\Delta_g}\leq \mu}u_0+\mathds{1}_{\sqrt{-\Delta_g}> \mu}u_0$ (and similarly for $u_1$). Here $\mathds{1}_{\sqrt{-\Delta_g}\leq \mu}$ denotes the orthogonal projector on the spectral space of $-\Delta_g$ associated to eigenfunctions $\lambda_j$ with $\sqrt{\lambda_j}\leq \mu$. 
Remarking that 
\begin{align*}
\|1_{\sqrt{-\Delta_g}>\mu}(u_0,u_1)\|_{L^2(\M) \times H^{-1}(\M)} & \leq \frac{1}{\mu}\|\mathds{1}_{\sqrt{-\Delta_g}> \mu} (u_0,u_1)\|_{H^1_0(\M)\times L^2(\M)}\\
&  \leq \frac{1}{\mu}\|(u_0,u_1)\|_{H^1_0(\M)\times L^2(\M)},
\end{align*}
we obtain
\bna
\|(u_0,u_1)\|_{L^2(\M) \times H^{-1}(\M)}&\leq& \|1_{\sqrt{-\Delta_g}\leq \mu}(u_0,u_1)\|_{L^2(\M) \times H^{-1}(\M)}+\frac{1}{\mu}\|(u_0,u_1)\|_{H^1_0(\M)\times L^2(\M)}\\
&\leq& e^{C_0\mu}\|e^{-C_0\sqrt{-\Delta_g}}(u_0,u_1)\|_{L^2(\M) \times H^{-1}(\M)}+\frac{1}{\mu}\|(u_0,u_1)\|_{H^1_0(\M)\times L^2(\M)}\\
&\leq& Ce^{C_0\mu} \|u\|_{L^2((0,T)\times\omega)} +\frac{1}{\mu}\|(u_0,u_1)\|_{H^1_0(\M)\times L^2(\M)},
\ena
where we used the assumption~\eqref{e:obs-analytic-space} in the last inequality. This concludes the proof of~\eqref{e:co-wave}, and that of the proposition.
\enp
We now extract an estimate like~\eqref{e:obs-analytic-space} on some surfaces of revolution from~\cite{Allibert:98}.
Indeed, a combination of several estimates in~\cite{Allibert:98} gives the following result.
\begin{theorem}[Allibert \cite{Allibert:98}]
\label{thmAllibert}
For any $T>T_1$ and $C_0>d_A(\Gamma)$, there exists $C>0$ so that
\bnan
\label{observanalyticAllibert}
\nor{e^{-C_0\sqrt{-\Delta_g}}(u_0,u_1)}{H^1_0\times L^2}\leq C\|\partial_{\nu}u\|_{L^2((0,T)\times\Gamma)}
\enan
for any $(u_0,u_1)\in H^1_0(\M)\times L^2(\M)$ and associated solution $u$ of~\eqref{e:wave}.
\end{theorem}
The result is not stated exactly this way in the article. It is also more precise since it involves analytic spaces only in the $\theta$ variable. More precisely, denoting $E_0^k$ the spaces of functions in $H^1_0\times L^2$ of the form $f(s)e^{i k\theta}$, the following estimate is proved in~\cite[Theor\`eme~2, D\'efinition~3 and Proposition~1]{Allibert:98}:
\bna
\nor{(u_0,u_1)}{H^1_0\times L^2}\leq C(k)\|\partial_{\nu}u\|_{L^2((0,T)\times\Gamma)}
\ena
for any $(u_0,u_1)\in E_0^k$, where $C(k)$ satisfies 
\bna
\limsup_{n\to +\infty}\frac{\ln C(k)}{k}=d_A(\Gamma).
\ena 
This gives \eqref{observanalyticAllibert} for any $C_0>d_A$, taking into account the orthogonality of the subspaces $E_0^k$ for the norm of $H^1_0\times L^2$ and the norm of the observation.

\bigskip
With Theorem~\ref{thmAllibert} in hand, Theorem~\ref{t:positrev} is now a straightforward consequence of Propositions~\ref{lienanalyticekwave} and~\ref{propheatwave}.

\subsection{Reformulation of the definition of the constants in terms of localization functions}
This section is aimed at giving an alternative definition for the geometric constants $ \co_{eig}(\omega)$, $\co_{Sigma}(\omega)$, $\co_{heat}(\omega)$ in terms of localization functions.
\begin{definition}
Let $\omega\subset \M$ be an open set. We set:
\bna
\Loc_{eig}(\omega,\lambda)=\inf \left\{\frac{\nor{\psi}{L^2(\omega)}}{\nor{\psi}{L^2(\M)}}, \psi \in E_{\lambda} \setminus \{0\} \right\} \in [0,1], \quad \lambda \in \Sp(-\Delta_g) ,
\ena
\bna
\Loc_{\Sigma}(\omega,\lambda)=\inf \left\{\frac{\nor{u}{L^2(\omega)}}{\nor{u}{L^2(\M)}}, u \in E_{\leq \lambda} \setminus \{0\} \right\}  \in [0,1],
\ena
\bna
\Loc_{heat}(\omega,T)=\inf \left\{\frac{\nor{e^{t\Delta}u_0}{L^2((0,T) \times \omega)}}{\nor{e^{T\Delta}u_0}{L^2(\M)}}, u_0 \in L^2(\M) \setminus \{0\} \right\} ,
\ena
\end{definition}
Note that if the Schr\"odinger equation is observable from $\omega$ in finite time (in particular if $\omega$ satisfies the geometric control condition, see~\cite{BLR:92,Leb:92}), , then, there exists $C>0$ so that $\Loc(\omega,\lambda_j)\geq C$ for all $j\in \N$. Under the sole assumtion  only $\omega\neq \emptyset$, there exists $C> 0$ so that $\Loc_{\omega,\lambda_j}\geq Ce^{-C\sqrt{\lambda_j}}$

\begin{lemma}
We have
\bna
 \co_{eig}(\omega) = \limsup_{\lambda \to +\infty, \lambda \in \Sp(-\Delta_g)} \frac{- \log \Loc_{eig}(\omega,\lambda)}{\sqrt{\lambda}}, 
 \qquad  \co_{\Sigma}(\omega) = \limsup_{\lambda \to +\infty} \frac{- \log \Loc_{\Sigma}(\omega,\lambda)}{\sqrt{\lambda}}, 
\ena
\bna
 \co_{heat}(\omega) = \limsup_{T\to 0^+} -T\log \Loc_{heat}(\omega,T), 
\ena
\end{lemma}
Note that we do not have a similar formulation for the constants $\co_\infty(\omega)$ and $\co_{wave}(\omega,T)$ since they do not correspond to an asymptotic r\'egime (like $T\to 0$ or $\lambda \to +\infty$).
\bnp
We only prove the second statement, the other proofs being similar. Setting $$\mathfrak{C}_{\Sigma}(\omega) = \limsup_{\lambda \to +\infty} \frac{- \log \Loc_{\Sigma}(\omega,\lambda)}{\sqrt{\lambda}},$$ we want to prove that $\mathfrak{C}_{\Sigma}(\omega)=\co_{\Sigma}(\omega)$. Assume $\co, C$ satisfy~\eqref{e:co-sum}, then we have 
$$
\Loc_{\Sigma}(\omega,\lambda) \geq \frac{1}{C}e^{-\co \sqrt{\lambda}}, 
$$
and hence 
$$
\frac{- \log \Loc_{\Sigma}(\omega,\lambda)}{\sqrt{\lambda}} \leq \frac{\co \sqrt{\lambda} +\log(C)}{ \sqrt{\lambda}} .
$$
Taking the $\limsup_{\lambda \to +\infty}$, this implies $\mathfrak{C}_{\Sigma}(\omega) \leq \co$. Taking the infimum over all such $\co$ and recalling Definition~\ref{def-coco}, we obtain $\mathfrak{C}_{\Sigma}(\omega) \leq \co_{\Sigma}(\omega)$. 

We now prove the converse inequality. The definition of $\mathfrak{C}_{\Sigma}(\omega)$ implies that for all $\eps$, there exists $\lambda_0(\eps)$ such that for all $\lambda \geq \lambda_0(\eps)$, 
$$
\frac{- \log \Loc_{\Sigma}(\omega,\lambda)}{\sqrt{\lambda}} \leq \mathfrak{C}_{\Sigma}(\omega) + \eps ,
$$
that is $\Loc_{\Sigma}(\omega,\lambda) \geq e^{-(\mathfrak{C}_{\Sigma}(\omega) + \eps)\sqrt{\lambda}}$. This, together with the fact that $\Loc_{\Sigma}(\omega,\lambda) >0$ does not vanish on $[0, \lambda_0(\eps)]$, implies the existence of a constant $C(\eps)>1$ such that $\Loc_{\Sigma}(\omega,\lambda) \geq \frac{1}{C(\eps)}e^{-(\mathfrak{C}_{\Sigma}(\omega) + \eps)\sqrt{\lambda}}$ for all $\lambda\geq 0$. This is precisely estimate~\eqref{e:co-sum} with $\co = \mathfrak{C}_{\Sigma}(\omega) + \eps$ and $C= C(\eps)$. Taking the infimum over all such $\co$ and recalling Definition~\ref{def-coco}, we obtain $\co_{\Sigma}(\omega) \leq \mathfrak{C}_{\Sigma}(\omega) + \eps$ for  all $\eps >0$, and hence $\co_{\Sigma}(\omega) \leq \mathfrak{C}_{\Sigma}(\omega)$, which concludes the proof.
\enp

\section{Construction of maximally vanishing eigenfunctions}
\label{s:contruction}
\subsection{The sphere}
\label{s:sphere}
In this section, we consider the simplest case of our results that is, the unit sphere in $\R^3$:
$$
\S^2 =  \{(x_1,x_2,x_3) \in \R^3 , x_1^2 + x_2^2 + x_3^2 = 1\} =  \{ x \in \R^3 ,|x|= 1\} .
$$
Eigenfunctions and eigenvalues of the Laplace-Beltrami operator on $\S^2$ are well-understood : eigenfunctions are restrictions to $\S^2$ of {\em harmonic homogeneous polynomials} of $\R^3$, associated to $k(k+1)$, where $k$ is the degree of the polynomial. We are particularly interested in so called equatorial spherical harmonics, given by
$$
u_k = P_k |_{\S^2}  \in C^\infty(\S^2) , \quad P_k(x_1,x_2,x_3) = (x_1 + i x_2)^k ,
$$
known to concentrate exponentially on the equator given by $x_3 = 0$. 

Since it can be written $P_k=z^k$ where $z=x_1 + i x_2\in\C$, it is easy to check that $P_k$ is holomorphic as a function of $z$ and indeed harmonic as a function of $(x_1,x_2,x_3)\in \R^3$. Moreover, $P_k$ is homogeneous of order $k$. Therefore, see e.g.~\cite[Proposition 22.2 p169]{Shubin:01}, we obtain that $u_k$ is an eigenfunction of the Laplace-Beltrami on $\S^2$:
\bna
-\Delta_{\S^2} u_k=\lambda_k u_k \textnormal{ with } \lambda_k=k(k+1).
\ena
Note that we have $$|u_k(\omega)|^2 = (x_1^2 +x_2^2)^k = (1-x_3^2)^k , \quad \omega = \frac{x}{|x|} .$$ 

We denote by $N = (0,0, 1)$ and $S = (0,0,-1)$, the north and south poles, and have coordinates :
$$
\begin{array}{ccc}
(0 , \pi) \times \S^1 & \to & \S^2 \setminus\{N,S\}  \\
(s ,\theta)& \mapsto & (\sin s \cos \theta , \sin s \sin \theta , \cos s)
\end{array}
$$
Remark that $s(x) = \dist_g(x, N)$, for $x\in \S^2$.
In these coordinates, the metric is given by $d s^2 +(\sin s)^2  d\theta^2$, the Riemannian volume element is  $d\omega=\sin s ds d\theta$, and the sequence $u_k$ is defined by 
$$
u_k(s, \theta) = \sin(s)^k e^{ik\theta} .
$$
\begin{remark}
The construction works equally well in the unit sphere $\S^n\subset \R^{n+1}$, $n\geq 2$. The coordinates has to be changed by  
$$
\begin{array}{ccc}
(0 , \pi) \times \S^1\times \S^{n-2} & \to & \S^{n} \setminus\{N,S\}  \\ 
(s ,\theta , t)& \mapsto & (\sin s \cos \theta , \sin s \sin \theta , t\cos s)
\end{array}
$$
and we can still consider the eigenfunction $u_k=(x_1 + i x_2)^k|_{\S^{n}}$ with $-\Delta_{\S^n} u_k=\lambda_k u_k$ and $\lambda_k=k(k+n-1)$.
\end{remark}

With the above choice of the eigenfunction $u_k$, we have
$$|u_k(x)|^2 = (1-x_3^2)^k  = (\sin s)^{2k} = |\sin \dist_g(x , N)|^{2 k} =e^{- 2k d_A(x)} , \quad d_A(x) = - \log \sin \dist_g(x , N) .$$ 
Note that $d_A$ is actually the Agmon distance to the equator ($s=\frac{\pi}{2}$) where $\S^2$ is seen as a surface of revolution, see Remark \ref{rksphereAgmon} below.

Also, given $f \in L^1(\S^2)$, we have 
\bna
\int_{\S^2} f(\omega) |u_k(\omega)|^2 d\omega 
& =&  \int_{(0,\pi)  \times \S^1}f(s , \theta) (\sin s)^{2k+1} d s d\theta  \\
& =&  2\pi \int_{(0,\pi) } F(s) (\sin s)^{2k+1} d s , \qquad F(s) = \frac{1}{2\pi} \int_{\S^1} f(s , \theta) d\theta .
\ena
In case $f=1$, this yields the asymptotics of the norm of $u_k$, given by the Laplace method (see e.g.~\cite{Erd:56,CopsonBook}):
\bna
\frac{1}{2\pi}\|u_k\|^2_{L^2(\S^2)}& = &\frac{1}{2\pi}\int_{\S^2} |u_k(\omega)|^2 d\omega 
 =  \int_{-1}^1 (1-x_3^2)^{k} d x_3  = \int_{-1}^1 e^{k \log(1-x_3^2)} d x_3 \\
 & = &  (1+O(\frac{1}{k}))\int_\R e^{- k x_3^2} d x_3 = \sqrt{\frac{\pi}{k}}(1+O(\frac{1}{k})) ,
 \ena
and hence $\|u_k\|_{L^2(\S^2)}\sim 2^{1/2}\pi^{3/4}k^{-1/4}$ as $k \to + \infty$.

We have the elementary estimate 
$$
\|u_k\|^2_{L^2(B(N,r))} = 2\pi \int_0^r (\sin s)^{2k+1} d s  \leq  \frac{\pi}{k+1} r^{2k+2} .
$$
This can be slightly refined, e.g. by writing
\begin{align*}
 \left| \|u_k\|^2_{L^2(B(N,r))} -  \frac{\pi}{k+1} (\sin r)^{2k+2} \right| & = \left| \|u_k\|^2_{L^2(B(N,r))} - 2\pi \int_0^r \cos s (\sin s)^{2k+1} d s \right| \\
 & = 2\pi \int_0^r (1-\cos s) (\sin s)^{2k+1} d s\\
 &  \leq  \frac{r^2}{2}2\pi \int_0^r (\sin s)^{2k+1} ds = \frac{r^2}{2} \|u_k\|^2_{L^2(B(N,r))} 
\end{align*}

To be a little more precise, let us now prove an equivalent for $\|u_k\|^2_{L^2(B(N,r))}$ as $k \to \infty$, which is uniform in $r$. 
\begin{lemma}
For all $k \in \N^*$ and all $r \in (0,\frac{\pi}{2})$, we have 
$$
\|u_k\|_{L^2(B(N,r))}^2 = \frac{\pi}{k+1}\frac{\sin(r)^{2k+2}}{\cos(r)}\left(1+R \right) , \quad \text{ with }\quad |R| \leq  \frac{\tan(r)^2}{2k+2} .
$$
\end{lemma}
This furnishes an optimal lower/upper bound for this quantity which is uniform with respect to $\eps$.
\bnp
We write $a= -\log \sin r >0$, change variable $y =-\log \sin s$, and want to have an asymptotic expansion of 
\bna
\frac{1}{2\pi} \|u_k\|^2_{L^2(B(N,r))} =  \int_0^r (\sin s)^{2k+1} d s  
 = \int_{a}^{+\infty} e^{-(2k+2)y} \frac{1}{\sqrt{1-e^{-2y}}} dy
\ena
This integral is of the form 
$$
\mathcal{I}(a,k) := \int_{a}^{+\infty} e^{-(2k+2) y} f(y)dy ,
$$ where $f(y)=\frac{1}{\sqrt{1-e^{-2y}}}$ is smooth on $[a, +\infty)$. Writing 
$$
|f(y) - f(a)|\leq (y-a)\sup_{[a,\infty)}|f'| \leq (y-a) \frac{e^{-2a}}{(1-e^{-2a})^{3/2}},
$$ 
since $f'(y) =- e^{-2y} (1-e^{-2y})^{-3/2}$
and integrating on $(a,+ \infty)$, we obtain
$$
\left| \mathcal{I}(a,k) - f(a)\frac{e^{-(2k+2)a}}{2k+2} \right| \leq   \frac{e^{-(2k+2)a}}{(2k+2)^2} \frac{e^{-2a}}{(1-e^{-2a})^{3/2}}.
$$
 Coming back to the original notation, this is precisely
 $$
\left|\frac{1}{2\pi}  \|u_k\|^2_{L^2(\O_\eps)} -\frac{\sin(r)^{2k+2}}{(2k+2) \cos(r)} \right| \leq \frac{\sin(r)^{2k+4}}{(2k+2)^2\cos(r)^3} =  \frac{\sin(r)^{2k+2}}{(2k+2)^2\cos(r)} \tan(r)^2,
$$
which  concludes the proof of the lemma.
\enp

Note that the eigenfunctions we have constructed are complex valued. Yet, since $u_k=(\sin(s))^k e^{ik\theta}$, we have for instance $\Re(u_k)=(\sin(s))^k \cos(k\theta)$ and the same estimates work exactly the same except that $\int_{\S^1}|e^{ik\theta}|^2 d\theta=2\pi$ should be replaced by $\int_{\S^1}\cos(k\theta)^2 d\theta=\pi$.

\subsection{General surfaces of revolution}
\label{s:revol}
In this section we consider a revolution surface $\mathcal{S}\subset \R^3$ being diffeomorphic to a sphere $\S^2$, generalizing the results of Section~\ref{s:sphere}.
We follow~\cite[Chapter4~B p95]{Besse} for the precise geometric description of such manifolds. 

Assume that $(\calS , g)$ is an embedded submanifold of $\R^3$ (endowed with the induced Euclidean structure), having $\S^1  = (\R/2\pi\Z) \sim SO(2)$ as an effective isometry group. 
The action of $\S^1$ on $\calS$, denoted by $\theta \mapsto \calR_\theta$ (such that $\calR_\theta \calS = \calS$) has exactly two fixed points denoted by $N, S \in \calS$ (the so-called North and South poles). 

We now describe a nice parametrization of $(\calS , g)$. Let $L= \dist_g(N, S)$ and $\gamma_0$ be a geodesic from $N$ to $S$ (thus with length $L$). For any $\theta \in \S^1$, the isometry $\calR_\theta$ transforms the geodesic $\gamma_0$ into $\calR_\theta (\gamma_0)$, which is another geodesic joining $N$ to $S$. Set
$U =\calS \setminus \{ N, S \}$. For every $m \in U$, there exists a unique $\theta \in \S^1$ such that $m$ belongs to $\calR_\theta (\gamma_0)$. The geodesic $\calR_\theta (\gamma_0)$ can be parametrized by arclength
$$
\rho : [0,L]\to  \calR_\theta (\gamma_0) ,  \quad \rho (0) = N , \quad \rho (L) = S ,  \quad s = \dist_g( \rho(s) , N) = L-  \dist_g( \rho(s) , S),
$$
and there exists a unique $s \in (0,L)$ such that $\rho(s) = m$. We use $(s,\theta)$ as a parametrization of $U \subset \calS$:
$$
\begin{array}{cccc}
\zeta : &  U = \mathcal{S} \setminus\{N,S\} & \to & ( 0,L) \times \S^1\\
& m& \mapsto &\zeta(m) =  (s, \theta).
\end{array}
$$
We define two other exponential charts $(U_N, \zeta_N)$ and $(U_S, \zeta_S)$ centered at the fixed points $N$ and $S$ by
$$
U_N = \{N\} \cup \zeta \left( \big( 0, \frac{L}{2} \big) \times \S^1 \right) = B_g\left(N, \frac{L}{2}\right)  \subset \calS, 
\quad U_S = \{S\} \cup \zeta \left( \big( \frac{L}{2}, L \big) \times \S^1 \right) = B_g\left(S, \frac{L}{2}\right) \subset \calS ,
$$
$$
\zeta_N : U_N \to B_{\R^2}\left(0 , \frac{L}{2}\right) , \quad \zeta_N (N) = 0 , \quad \zeta_S : U_S \to B_{\R^2}\left(0 , \frac{L}{2}\right) , \quad \zeta_S (S) = 0 ,
$$
with the transition maps 
$$
\begin{array}{cccc}
\zeta_N \circ \zeta^{-1} : & \zeta \big(U\cap U_N \big) = \left( 0,\frac{L}{2} \right) \times \S^1  & \to &  \zeta_N \big(U\cap U_N \big) =   B_{\R^2}\left(0 , \frac{L}{2}\right) \setminus \{ 0 \}\\
& (s ,\theta)& \mapsto &  \big(s \cos(\theta), s \sin (\theta) \big).
\end{array}
$$
and
$$
\begin{array}{cccc}
\zeta_S \circ \zeta^{-1} : & \zeta \big(U\cap U_S \big) = \left(\frac{L}{2} , L \right) \times \S^1  & \to &  \zeta_S \big(U\cap U_S \big) =   B_{\R^2}\left(0 , \frac{L}{2}\right) \setminus \{ 0 \}\\
& (s ,\theta)& \mapsto &  \big((L-s) \cos(\theta), (L-s) \sin (\theta) \big).
\end{array}
$$

On the cylinder $(0,L) \times \S^1$, the metric $g$ is given by 
$$
(\zeta^{-1})^*g = ds^2 + R(s)^2 d\theta^2
$$
for some smooth function $R : (0,L) \to \R^+_*$ (see below Remark \ref{rkautreparam} for the geometric interpretation of $R$). Since $g$ is a smooth metric on $\calS$, \cite[Proposition~4.6]{Besse} gives that $R$ extends to a $C^\infty$ function $[0,L] \to \R^+$ satisfying 
\bnan
\label{e:condR}
R(0) = R(L) = 0 , \quad R'(0) = 1 , \quad R'(L) = -1 , \quad R^{(2p)}(0) =  R^{(2p)}(L)=0 \quad \text{for any } p \in \N .
\enan
In these coordinates, the Riemannian volume form is hence $R(s) ds d\theta$, the Riemannian gradient of a function is
\bnan
\label{e:reim-gradient}
\nabla_g f = \d_s f \frac{\d}{\d s} + \frac{1}{R(s)^2} \d_\theta f \frac{\d}{\d \theta} , \quad \text{ with } \quad g(\nabla_g f , \nabla_g f ) = |\d_s f|^2 + \frac{1}{R(s)^2} |\d_\theta f|^2
\enan
and the Laplace-Beltrami operator is given by
\bnan
\label{deltacoord}
\Delta_{s,\theta} = \frac{1}{R(s)} \d_s R(s) \d_s + \frac{1}{R(s)^2} \d_\theta^2 .
\enan
Another important operator is the infinitesimal generator $X_\theta$ of the group $(\calR_\theta)_{\theta \in \S^1}$, defined, for $f \in C^\infty(\calS)$, by 
\bnan 
\label{e:Xtheta}
X_\theta f = \lim_{\theta \to 0} \theta^{-1} (f\circ \calR_\theta -f).
\enan In the chart $(U, \zeta)$, the action of $\calR_\theta$ is given by $(\zeta^{-1})^* \calR_\theta (u, \theta') = (u, \theta' +\theta)$, so that $(\zeta^{-1})^* X_\theta  = \d_\theta$. Let us now check that $X_\theta$ is a smooth vector field. Indeed, we have
$$
(\zeta_N^{-1})^*X_\theta = (\zeta_N^{-1})^*\zeta^* \d_\theta = d \left( \zeta_N \circ \zeta^{-1}\right) \cdot \d_\theta, 
$$
and hence 
$$
(\zeta_N^{-1})^*X_\theta \big(s \cos(\theta), s \sin (\theta) \big) =  \left(- s \sin (\theta) \d_{x_1} + s \cos(\theta) \d_{x_2} \right)  \big(s \cos(\theta), s \sin (\theta) \big) ,
$$
 that is 
 $$
 \big( (\zeta_N^{-1})^*X_\theta \big)(x_1,x_2) = - x_2  \d_{x_1} + x_1 \d_{x_2}.
$$ 
Since $\big( (\zeta_N^{-1})^*X_\theta \big)(0) = 0$ (and since the computation is similar in $U_S$), we have obtained that $X_\theta$ is smooth. Note also that $X_\theta(N) = X_\theta(S)=0$ and that its norm is given by $\sqrt{g(X_\theta, X_\theta) (s, \theta)}= R(s)$ (in the coordinates of $U$). 

We define by $L^2(\calS) := L^2(\calS, d\Vol_g)$ the space of square integrable functions, which is also invariant by the action of $(\calR_\theta)_{\theta \in \S^1}$.

Now, remark that $(\calR_\theta)_{\theta \in \S^1}$ acts as a (periodic) one-parameter unitary group on $L^2(\calS)$ by $f \mapsto f\circ \calR_\theta$. The Stone Theorem (see e.g.~\cite[Theorem~VIII-8~p266]{RS:I}) hence implies that its infinitesimal generator is $i A$, where $A$ is a selfadjoint operator on $L^2(\calS)$ with domain $D(A) \subset L^2(\calS)$. Since $i Af = X_\theta f$ for $f \in C^\infty(\calS)$ (which is dense in $D(A)$) according to~\eqref{e:Xtheta}, we have that $A$ is the selfadjoint extension of $\frac{X_\theta}{i}$. From now on, we slightly abuse the notation and still denote $\frac{X_\theta}{i}$ for its selfadjoint extension $A$. 

Since $g$ is invariant by the action of $\calR_\theta$, we have 
$$
[X_{\theta} , \Delta_g ] =0 .
$$ 
Moreover, $\Delta_g$ has compact resolvent, so that the operators $\Delta_g$ and $X_\theta$ share a common basis of eigenfunctions: indeed, $X_\theta$ preserves each (finite dimensional) eigenspace of $\Delta_g$, and it can be diagonalized on these spaces. In $U$ it can be written as $e^{ik \theta}f(s)$ with $k\in \Z$, $f \in C^{\infty}(0,L)\cap L^2\left((0,L), R(s)ds\right)$ solution of 
\bnan
\label{equationk1D}
- \frac{1}{R(s)} \d_s\left( R(s) \d_s f\right) + \frac{k^2}{R(s)^2} f=\lambda f .
\enan
for some $\lambda\geq 0$, eigenvalue for $-\Delta$. Let us detail this assertion. Take $u$ a necessarily smooth common eigenvalue of $\Delta_g$ and $X_{\theta}$. In $U$ (with the coordinates $(s,\theta)$), denote $f(s)=u(s,0)$. $u$ is smooth and satisfies $X_{\theta}u=i\lambda_{\theta}u$ in the classical sense. Then, for any fixed $s_0\in (0,L)$, the function $g_{s_0}\in C^{\infty}(\S^1)$ defined by $\theta \mapsto g_{s_0}(\theta)=u(s_0,\theta)$, is solution of $\d_{\theta} g_{s_0}(\theta)=i\lambda_{\theta}g_{s_0}(\theta)$ and can be written $g_{s_0}(\theta)=e^{i\lambda_{\theta}\theta}f(s_0)$. By periodicity, $\lambda_{\theta}=k\in \Z$ and is thus independent of $\theta$. So, $u(s_0,\theta)=e^{ik\theta}f(s_0)$, and it is clear from \eqref{deltacoord} that $f$ must satisfy \eqref{equationk1D}.

We will call these normalized eigenfunctions $\varphi_{k,n}=e^{ik \theta}f_{k,n}(s)$ with eigenvalues $\lambda_{k,n}$ for $-\Delta_g$, where $n\in \N$.  In particular, we can write $L^2(\calS)=  \oplus^{\perp }_{(k,n)\in \Z\times \N}\vect (\varphi_{k,n}) $.

We will denote $L^2_k=\ker (X_{\theta}-ik))=\left\{\varphi\in L^2(\calS);\varphi_{|U} =e^{ik \theta}f(s),f \in L^2\left((0,L),R(s)ds\right)\right\}$ and $H^2_k=H^2(\calS)\cap L^2_k$. The commutation property implies that $\Delta H^2_k\subset L^2_k$, so we can define the operator $\Delta_k=\Delta_{\left|L^2_k\right.}$ which is self-adjoint with domain $H^2_k$. This can be seen for instance directly on the simultaneous diagonalization which implies an isometry $L^2(\calS) \approx \ell^2(\Z\times \N)$ where $L^2_k\approx \left\{(k,n)\left|n\in \N\right.\right\}$ as a closed subspace of $\ell^2(\Z\times \N)$. The fact that $\Delta_g$ has compact resolvent implies the same things for $\Delta_k$.

\begin{remark}
\label{r:def-intrinsic-dA}
Note that the introduction of $X_\theta$ allows to give a more intrinsic definition of $d_A$ introduced in \eqref{e:defdA}: given a point $m_0$ on the ``strict global non-degenerate equator'' of $\calS$, the Agmon distance $d_A$ is the unique continuous function such that
$$
X_\theta d_A = 0 , \quad d_A(m_0) = 0 , \quad |\nabla_g d_A|_g^2(m) - \left( \frac{1}{g(X_\theta ,X_\theta)(m)}- \frac{1}{g(X_\theta ,X_\theta)(m_0)} \right) = 0.
$$
All properties of Lemma~\ref{lemma-prop-dA} can be formulated intrinsically since $s$ measures the geodesic distance to the north pole, and hence $s(m) = \dist_g(m, N)$, $L- s(m) = \dist_g(m, S)$, and $s(m)-s_0 =\dist_g(m, equator) $.
\end{remark}
\begin{remark}(Another possible parametrization)
\label{rkautreparam}
Some such surfaces of revolution admit the following ``cylindrical'' parametrization on the set $U$:
with $ \pm z_\pm>0$ and the two poles $P_\pm = (0,0, z_\pm)$, we have 
$$
\begin{array}{ccc}
( z_- , z_+) \times \S^1 & \to & U =  \mathcal{S} \setminus\{N,S\}  \subset \R^3\\
(z ,\theta)& \mapsto & (\mathsf{R}(z) \cos \theta , \mathsf{R}(z)\sin \theta , z)
\end{array}
$$
where $\mathsf{R}: [z_- , z_+] \to (0,\infty)$ is the profile of the surface, that is, a smooth positive function on $(z_-, z_+)$ satisfying $\mathsf{R}(z_\pm)=0$ and $\lim_{z \to z_\pm} \mathsf{R}'(z) = \mp\infty$.

We have 
$$
\left\{
\begin{array}{l}
dx_1 = \mathsf{R}'(z) \cos \theta dz - \mathsf{R}(z) \sin \theta d\theta \\
dx_2 = \mathsf{R}'(z) \sin \theta dz + \mathsf{R}(z) \cos \theta d\theta  \\
dx_3 = dz
\end{array}
\right.
$$
so that the metric on $\mathcal{S}$ induced by the Euclidean structure is given by 
$$g = dx_1^2 +dx_2^2+dx_3^2 = (1+\mathsf{R}'(z)^2)dz^2 + \mathsf{R}(z)^2 d\theta^2 . $$
As a consequence, the Riemannian volume element is $V(z) dz d\theta$ with $V(z)= \mathsf{R}(z)\sqrt{1+\mathsf{R}'(z)^2}$ and the Laplace-Beltrami operator is given in this coordinates, by
$$
\Delta_{z,\theta} = \frac{1}{V(z)} \d_z \left(\frac{V(z)}{1+\mathsf{R}'(z)^2} \d_z  \right)+ \frac{1}{\mathsf{R}(z)^2} \d_\theta^2 ,
$$
with a suitable selfadjoint extension on $L^2\big((z_- , z_+) \times \S^1, V(z) dz d\theta \big)$. 
The link between $s$ and $z$ is the following diffeomorphism
$$
s(z) = \int_{z_-}^z \sqrt{1+\mathsf{R}'(t)^2} dt ,
$$
and we have $L= \int_{z_-}^{z_+} \sqrt{1+\mathsf{R}'(t)^2} dt$, together with $R(s(z)) =\mathsf{R}(z)(=\sqrt{g(X_\theta, X_\theta)})$, so that $R(s)$ indeed measures the distance to the axis of revolution.
\end{remark}

\begin{remark}[The sphere]
Note that, in the $z$-parametrization, the sphere is given by $z_\pm = \pm1$ and $r(z)= \sqrt{1-z^2}$ and hence $r'(z)= \frac{-z}{\sqrt{1-z^2}}$ and $V(z) = 1$ is smooth (which is not the general case if the surface is flat near the poles).
\end{remark}

In the proofs below, we shall often consider  $h=k^{-1}$ as a semiclassical parameter.

\begin{lemma}
\label{l:exist-mu-psi}
Assume that $s \mapsto R(s)$ admits a non-degenerate local maximum at $s_0 \in (0,L)$. 
Then, for all $k\in\N$, there exists $\psi_k \in C^{\infty}(\calS)\cap L^2_k $,  and $\mu_k \in \R$ such that $\mu_k=  \frac{1}{R(s_0)^2}+\frac{1}{k}\sqrt{\frac{|R''(s_0)|}{R^3(s_0)}} + O(\frac{1}{k^{\frac32}})$, $\|\psi_k\|_{L^2(\calS)}=1$, and we have $- \Delta_g \psi_k = k^2 \mu_k \psi_k$.
\end{lemma}
Note that the assumption  of the lemma is $R'(s_0) = 0$ and $R''(s_0)<0$.

\bnp
We first construct a family of exponentially accurate quasimodes (i.e. approximate eigenfunctions) compactly supported in $U$ and of the form (in the coordinates $(s, \theta)$ of $U$) $e^{ik \theta} u_k(s)$.
The function $u_k(s)$ should thus solve~\eqref{equationk1D} approximately. Setting $h=k^{-1}$ and $\mu = \lambda h^2$, we are left with the following semiclassical eigenvalue (or approximate eigenvalue) problem in the limit $h\to 0^+$
$$
(P_h - \mu)f = - \frac{h^2}{R(s)} \d_s \left(R(s) \d_s f \right)+\left( \frac{1}{R(s)^2} -  \mu \right) f = 0 .
$$
According to the assumption, the potential $\frac{1}{R(s)^2}$ is positive, tends to plus infinity near $0$ and $L$, and admits $\frac{1}{R(s_0)^2}$ as a nondegenerate local minimum. Namely, this is $R'(s_0)=0$ and $R'' (s_0) <0$.
The construction is classical (harmonic approximation) and follows e.g. that of~\cite[Theorem~4.23]{DS:book} in a simpler setting. The idea is to approximate the operator $P_h$ by its harmonic approximation at $s_0$, namely
\bnan
\label{e:model-harmonic-oscillator}
\tilde{P}_h := - \frac{h^2}{R(s_0)} \d_s R(s_0) \d_s + \frac{1}{R(s_0)^2} + \left(\frac{1}{R^2} \right)''(s_0) \frac{(s-s_0)^2}{2} = -h^2 \d_s^2 + \frac{1}{R(s_0)^2} - \frac{2R''(s_0)}{R^3(s_0)}\frac{(s-s_0)^2}{2} 
\enan
Recall that the spectrum of the operator $-h^2 \d_y^2+ c_0 y^2$ on $\R$ ($c_0>0$) is given by $E_n(h) = hE_n(1) = h (2n+1)\sqrt{c_0}$, associated with the eigenfunctions $u_n^h(y) = h^{-\frac14}u_n^1(y/\sqrt{h})$ where $u_n^1(y) = p_n(y)e^{-\sqrt{c_0}\frac{y^2}{2}}$ ($p_n$ being a Hermite polynomial).
Here, this applies with $c_0 =  \frac{|R''(s_0)|}{R^3(s_0)}$.

We consider a cutoff function $\chi \in C^\infty_c(0,L)$ such that $\chi=1$ in a neighborhood of $s_0$. We set 
$u^h(s) = \chi(s) u_0^h(s)$, with $u_0^h(s) = C h^{-\frac14}e^{-\sqrt{c_0}\frac{(s-s_0)^2}{2h}}$ where 
$C$ is a normalizing constant, 
and prove this is an approximate eigenfunction (quasimode).  
First notice that we have, with $\tilde{P}_h$ defined in \eqref{e:model-harmonic-oscillator}, that
$$
\tilde{P}_h u^h = \chi \tilde{P}_h u_0^h + [ \tilde{P}_h , \chi] u_0^h 
= \Big(\frac{1}{R(s_0)^2}+h\sqrt{c_0}\Big)\chi u_0^h + [ -h^2 \d_s^2 , \chi] u_0^h . 
$$ 
In this expression, $[ -h^2 \d_s^2 , \chi]$ is a first order differential operator supported away from zero, where $u_0^h$ and its derivatives are exponentially small. This yields
\bna
\| \tilde{P}_h u^h - \Big(\frac{1}{R(s_0)^2}+h\sqrt{c_0}\Big) u^h \|_{L^2\big((0 ,L) , R(s) ds\big)} = O(e^{-c/h}).
\ena
Now we consider, with norms $L^2\big((0 ,L) , R(s) ds\big)$
\bna
\nor{\left(P_h - \Big(\frac{1}{R(s_0)^2}+ h\sqrt{c_0} \Big) \right)u^h}{L^2} & \leq &
\nor{\left(P_h - \tilde{P}_h \right)u^h}{L^2} 
+\nor{\left( \tilde{P}_h u^h - \Big(\frac{1}{R(s_0)^2}+h\sqrt{c_0}\Big)\right)u^h}{L^2}  \\
& \leq &
\nor{\left(\frac{h^2}{R(s)} \d_s R(s) \d_s - h^2 \d_s^2\right)u^h}{L^2} \\
&& + \nor{\left(\frac{1}{R(s)^2} -\frac{1}{R(s_0)^2} - c_0 (s-s_0)^2\right)u^h}{L^2} 
+ Ce^{-c/h}  .
\ena
According to the Taylor formula and the definition of $c_0$, we have $\frac{1}{R(s)^2} -\frac{1}{R(s_0)^2} - c_0 (s-s_0)^2 = O((s-s_0)^3)$ on the support of $\chi$, so that
$$
\nor{\left(\frac{1}{R(s)^2} -\frac{1}{R(s_0)^2} - c_0 (s-s_0)^2\right)u^h}{L^2}^2 \leq C \int_{\R} |(s-s_0)^3 h^{-\frac14}e^{-\sqrt{c_0}\frac{(s-s_0)^2}{2h}}|^2 dz \leq C h^3 .
$$
We now estimate the term 
\bna
\nor{\left(\frac{h^2}{R(s)} \d_s R(s) \d_s - h^2 \d_s^2\right)u^h}{L^2} 
& = & \nor{\frac{h  R'(s)}{R(s)}  h\d_s u^h}{L^2}
\ena
Notice that $h \d_s u^h = h \chi' u_0^h + h \chi \d_s u_0^h = O_{L^2}(e^{-c/h}) - \sqrt{c_0} (s-s_0) \chi u_0^h$.
Moreover,  since $R'(s_0)=0$, the Taylor formula yields
$$
\nor{\frac{h  R'(s)}{R(s)}  h\d_s u^h}{L^2}  \leq  C e^{-c/h}+ C\nor{h (s-s_0)^2 \chi u_0^h}{L^2} \leq Ch^2 .
$$

Now, combining the above estimates finally yields the existence of constants $D, h_0>0$ such that for all $h<h_0$, we have, with $\nu_h = \frac{1}{R(s_0)^2}+ h\sqrt{c_0}$, 

$$
\nor{(P_h - \nu_h) u^h}{L^2\big((0 ,L) , R(s) ds\big)}  \leq Dh^{3/2} \approx Dh^{3/2} \nor{u^h}{L^2\big((0 ,L) , R(s) ds\big)}  .
$$
Now, we define in coordinates in $U\subset \calS$, $f_k(s,\theta)= e^{ik\theta}u^h(s)$, $h=k^{-1}$. This function is smooth and compactly supported in $U$ thanks to the cutoff $\chi$, and can therefore be extended as a function in $C^{\infty}(\calS)\cap L^2_k$, still denoted $f_k$, which satisfies
\bna
\nor{(h^2\Delta_k - \nu_h) f_k}{L^2_k}\leq Dh^{3/2}  \approx Dh^{3/2} \nor{f_k}{L^2_k} .
\ena

Hence, if $\nu_h \notin \Sp(-h^2\Delta_k )$, this implies $\nor{(-h^2\Delta_k- \nu_h)^{-1}}{L^2_k\to L^2_k}\geq \frac{1}{Dh^{3/2}}$.
Finally, the operator $h^2\Delta_k$ being selfadjoint on $L^2_k$, we have, for $z \in \C \setminus \Sp(-h^2\Delta_k)$, $\|(-h^2\Delta_k-z)^{-1}\| = \frac{1}{d(z, \Sp(-h^2\Delta_k))}$, so that, if $\nu_h \notin \Sp(-h^2\Delta_k)$, 
$$
\frac{1}{d(\nu_h, \Sp(-h^2\Delta_k))} \geq \frac{1}{Dh^{3/2}} .
$$
In any case, this implies $d(\nu_h, \Sp(-h^2\Delta_k))\leq Dh^{3/2}$, and using that the spectrum of $h^2\Delta_k$ is purely pointwise, this proves the sought result.
\enp
The next step is to study the behavior of the eigenfunction $\psi_k$ constructed in the previous lemma (and under a stronger assumption on the point $s_0$). This is the goal of the so-called Agmon estimates.
We first need the following integration-by-parts lemma.

\begin{lemma}
\label{l:IPPphi}
For all $\Psi \in W^{1,\infty}(\mathcal{S})$ real valued and all $w \in H^2(\mathcal{S})$, we have
\bna
\int_\mathcal{S} |\nabla_g(\Psi w)|_g^2 d \Vol_g  - \int_\mathcal{S}  |\nabla_g\Psi |_g^2 |w|^2 d \Vol_g
=  \Re \Big( \int_\mathcal{S} |\Psi|^2(-\Delta_g w)\overline{w}  \ d\Vol_g \Big) .
\ena
\end{lemma}
\bnp
For $\Psi \in C^2(\mathcal{S})$, this is a direct consequence of the integration by parts formula (also valid when $\mathcal{S}$ has a boundary $\d \mathcal{S}$ and $w|_{\d \mathcal{S}}=0$)
\bna
\int_\mathcal{S} |\nabla_g(\Psi w)|_g^2 d \Vol_g 
& = & -   \int_\mathcal{S} \Delta_g (\Psi w)\Psi \overline{w}  d \Vol_g   \\
& = & \Re\left(  \int_\mathcal{S} \big(- \Psi (\Delta_g w) - (\Delta_g \Psi)  w  - 2  \nabla_g \Psi \cdot \nabla_g w\big)  \Psi \overline{w}d \Vol_g  \right) \\
& = & \Re \Big( \int_\mathcal{S} |\Psi|^2(-\Delta_g w)\overline{w}  \ d\Vol_g \Big) + A 
\ena
with 
\bna
A & = & \Re\left(  \int_\mathcal{S} \big( - (\Delta_g \Psi)  \Psi |w|^2  - 2  \nabla_g \Psi \cdot \nabla_g w  \Psi \overline{w}\big) d \Vol_g  \right) \\
& = & \Re\left(  \int_\mathcal{S} \big(  |\nabla_g \Psi|^2 |w|^2 +  \nabla_g \Psi \cdot \nabla_g ( |w|^2) \Psi  - 2  \nabla_g \Psi \cdot \nabla_g w  \Psi \overline{w}\big) d \Vol_g  \right) \\
& = &  \int_\mathcal{S} \big(  |\nabla_g \Psi|^2 |w|^2 d \Vol_g ,
\ena
where we integrated by parts in the second line. This is the sought estimate in case $\Psi \in C^2(\mathcal{S})$.
The result of the lemma follows by a classical approximation argument, see e.g.~\cite[Proof of Proposition~6.1]{DS:book}.
\enp

\medskip
We shall now assume that $R$ reaches at $s_0$ a {\em strict global non-degenerate} maximum, and introduce the relevant Agmon distance to the ``equator'' $s=s_0$. The latter is defined in the coordinates of $U$ by the eikonal equation~\eqref{e:defdA}, or, more explicitely, for $s\in (0,L)$, by~\eqref{e:defbisdA}.

\begin{lemma}[Properties of $d_A$]
\label{lemma-prop-dA}
Assume that $R$ reaches at $s_0$ a {\em strict global non-degenerate} maximum. Then, $d_A \in C^2(0,L)$, and  we have 
\bnan
\label{equivdAlog}
d_A(s) = -\log(s) + O(1) , \quad \text{as } s \to 0^+ , \qquad d_A(s)= - \log (L-s) + O(1), \quad \text{as } s \to L^- ,
\enan
\bnan
\label{equivdAs0}
d_A(s) = \frac{1}{2}\sqrt{\frac{-R''(s_0)}{R^3(s_0)}} (s-s_0)^2 + O((s-s_0)^3), \quad \text{as }s \to s_0 .
\enan
\end{lemma}
\bnp
Remark that according to \eqref{e:condR}, we have $\frac{1}{R(y)}\to + \infty$ as $y\to 0^+$ or $y \to L^-$, with
$$
R(s) = s + O(s^3) , \text{ when } s \to 0^+ , \quad \text{ and } \quad R(s) = L- s + O((L-s)^3) , \text{ when }s \to L^- .
$$
As a consequence, with~\eqref{e:defbisdA}, we obtain $d_A(s) = \left| \int^s_{s_0} \frac{1}{y}(1+O(y^2)) dy \right| = -\log(s) + \rouge{O(1)},$   as $s \to 0^+$ (and similarly when $s \to L^-$), that is~\eqref{equivdAlog}.

Let us also study the behavior of $d_A$ near $s_0$. Denoting $V(s)=\frac{1}{R(s)^2} - \frac{1}{R(s_0)^2} $, we have $V(s_0)=V'(s_0)=0$ and $V''(s_0)=\frac{-2R''(s_0)}{R^3(s_0)} >0 $. This implies ~\eqref{equivdAs0} and that $d_A$ is of class $C^2$ near $s_0$, by Taylor expansion of $d_A$ and its derivatives.
\enp
We can now state the following very precise result. All results concerning surfaces of revolution are corollaries of this one.
\begin{theorem}[Agmon estimate]
\label{lmagmon}
Assume that $R$ reaches at $s_0$ a {\em strict global non-degenerate} maximum, and consider the associated numbers $\mu_k$ and functions $\psi_k$ given by Lemma~\ref{l:exist-mu-psi}.
There exist $C,C_0,k_0>0$ such that, for all $k \in \N$, $k \geq k_0$, the following integral is well defined with the estimate
$$
\int_\calS e^{2 k d_A(m)} |\psi_k|^2(m) d\Vol_g(m)  
\leq Ck^{2C_0} .
$$
\end{theorem}
Using first that $d_A$ is decreasing on $(0,s_0]$, we obtain the following direct Corollary. 
\begin{corollary}
\label{coragmonfaible}
Under the assumptions of Theorem~\ref{lmagmon}, there exist $C,C_0,k_0>0$ such that, for all $k \in \N$, $k \geq k_0$ and all  $s_1 \leq s_0$, we have 
$$
\int_{B(N,s_1)} |\psi_k|^2d \Vol_g
\leq Ck^{2C_0} e^{- 2d_A(s_1)k}.
$$
\end{corollary}

From this result, we may now derive a proof of Theorem~\ref{t:agmon-intro} and Corollary~\ref{c:rate-vanish}.

\bnp[Proof of Theorem~\ref{t:agmon-intro} and Corollary~\ref{c:rate-vanish}]
The eigenfunctions constructed in Lemma \ref{l:exist-mu-psi} satisfy $\lambda_k=k^2\left(\frac{1}{R(s_0)^2}+O(k^{-1}\right)$. In particular, $k\geq \sqrt{\lambda_k}R(s_0)-C$ for an appropriate constante $C$ and $k$ large enough. This gives $e^{- 2kd_A(s_1)}\leq e^{C d_A(s_1)} e^{- 2d_A(s_1)R(s_0)\sqrt{\lambda_k}}$. Then, Theorem~\ref{t:agmon-intro} follows directly from Corollary~\ref{coragmonfaible} up to changing the constants involved.  The second pat of Theorem~\ref{t:agmon-intro} follows directly from Proposition \ref{p:link-eigenfct-heat-etc}.

Corollary~\ref{c:rate-vanish} follows from the asymptotic \eqref{equivdAlog} of $d_A$ and the fact than Theorem~\ref{t:agmon-intro} is uniform for $r$ small. Indeed, for an appropriate constant $C$, we have $ d_A(s)\leq -\log(s)+C$ for all $0<s_1\leq s_0$,  . 

For fixed $\lambda_k$ and using the uniformity for $r$ small, we get the order of vanishing using the general Lemma \ref{lmordrevanishinL2} of the Appendix.
\enp

\bigskip
We will need a very simple Lemma 
\begin{lemma}
\label{lminfgrad}
Let $\varphi \in W^{1,\infty}(\calS)\cap L^2_k$, then, we have the pointwise estimate on $U$
\bna
 |\nabla_g(\varphi)|_g^2\geq \frac{k^2}{g(X_\theta, X_\theta)} \left|\varphi\right|^2.
\ena
\end{lemma}
\bnp
We have, in the coordinates of $U$, that $\varphi$ writes $\varphi(s, \theta) = e^{ik\theta}f(s)$, with, according to~\eqref{e:reim-gradient},
\bna
 |\nabla_g(\varphi)|_g^2 & = & |\d_sf|^2 + \frac{1}{R(s)^2}|\d_\theta(e^{ik\theta}f(s))|^2 =  |\d_sf|^2 + \frac{k^2}{R(s)^2}|e^{ik\theta}f(s)|^2 \\
 & \geq & \frac{k^2}{R(s)^2}|e^{ik\theta}f(s)|^2 = \frac{k^2}{g(X_\theta, X_\theta)} \left|\varphi\right|^2,
\ena
which is the sought result.
\enp
The proof follows that of \cite[Proposition 3.3.5]{Helffer:booksemiclassic}.

\bnp[Proof of Theorem~\ref{lmagmon}]
As in the above proof, we use the notation $h=k^{-1}$, considered as a semiclassical parameter.
We define, for some constant $C_0>1$, $h_0>0$ and $h \in (0,h_0)$ the sets 
\bna
\Omega_- = \{s \in (0,L), d_A(s) \leq C_0 h\}, \quad \Omega_+ = \{s \in (0,L),   d_A(s)> C_0 h\},
\ena
We set 
\bna
\phi(s)
& = & d_A(s) - C_0 h \log(C_0) , \quad \text{for } s\in  \Omega_- ,\\
 & = & d_A(s) - C_0 h \log(d_A(s)/h) , \quad \text{for } s\in  \Omega_+ .
\ena
For $M>1$, set $\phi_M = \min(\phi , M)$ and $\Omega_M = \phi_M^{-1}(\{M\})$.
Moreover, on $\Omega_-$, we have $\phi=d_A-C_0h\log(C_0)\leq d_A\leq C_0 h<C_0 h_0$, so for $M\geq C_0 h_0$, we have $\Omega_-\cap \Omega_M=\emptyset$. Indeed, we have a partition $\Omega_- \sqcup (\Omega_+ \setminus \Omega_M) \sqcup (\Omega_+ \cap  \Omega_M)$.

Note that it will be very important in what follows that all the estimates are independent on $M$ while $C_0$ will be defined later on.
The function $\phi_M$ is Lipschitz on $(0,L)$, and can be pulled back to a $(R_\theta)$ invariant Lipschitz function defined on $U$, and extended to $\calS$ by $\phi_M(N)=\phi_M(S)=M$. We will call $\calS_+$, $\calS_-$ and $\calS_M$, the naturally defined zones so that 
$$
\calS=\calS_-\sqcup ( \calS_+ \setminus \calS_M )\sqcup\calS_M .
$$

We now apply the formula of Lemma~\ref{l:IPPphi} with $\Psi=e^{\frac{\phi_M}{h}}$ with $\phi_M$ given above and $M$ large, and $w=\psi_h$ (note that $\psi_h\in C^{\infty}(\calS)$ since it is an eigenfunction of $\Delta_g$, so the Lemma applies). 
\bna
\int_\mathcal{S} |\nabla_g(\Psi \psi_h )|_g^2 d \Vol_g -\int_\mathcal{S} |\nabla_g\Psi |_g^2  |\psi_h|^2 d \Vol_g
=  k^2\mu_h \int_\mathcal{S} |\Psi|^2 |\psi_h|^2   d \Vol_g  .
\ena
Applying now Lemma \ref{lminfgrad} since $\Psi \psi_h \in W^{1,\infty}(\calS)\cap L^2_k$ and using $|\nabla_g\Psi |_g^2=k^2 |\phi_M'(s)|^2 e^{2\phi_M/h} $ in $U$ and so almost everywhere in $\calS$, we get
\bna
\int_\mathcal{S}\left( \frac{1}{R(s)^2}-  |\phi_M'(s)|^2 -\mu_h\right)  e^{2\phi_M/h} |\psi_h|^2 d \Vol_g\leq  0.
\ena

 Using the expression of $\phi_M$ on $\Omega_-$ and of $\mu_h = \frac{1}{R(s_0)^2} + O(h)$, this yields, for some $C>0$ (independent of $h$ and $M$),
\bna
\int_{\mathcal{S}_+}\left( \frac{1}{R(s)^2}-  |\phi_M'(s)|^2 -\mu_h\right)  e^{2\phi/h} |\psi_h|^2 d \Vol_g 
& \leq & C h \int_{\mathcal{S}_-}  e^{2d_A(s)/h}|\psi_h|^2 d \Vol_g \\
& \leq & C h e^{2C_0}\int _{\mathcal{S}_-}|\psi_h|^2 d \Vol_g  \leq  Ch  e^{2C_0},
\ena
since $\psi_h$ is normalized. 

Note also that on $\Omega_M \cap \Omega_+$, we have $d_A\geq C_0h$ and so $d_A\geq d_A - C_0 h \log(C_0) \geq  \phi \geq M\geq 1$.
Hence, since $d_A$ is continuous, there is a constant $\e>0$ so that $s\in \Omega_M \cap \Omega_+$ implies $|s-s_0|\geq \e$. In particular, since $s_0$ is a nondegenerate maximum for $R$, there is $\eta>0$ so that it also implies $ \frac{1}{R(s)^2} -\frac{1}{R(s_0)^2}\geq \eta$. In particular, on $\calS_M \cap \calS_+$, we have 
 $$
 \frac{1}{R(s)^2}-  |\phi_M'(s)|^2 -\mu_h =  \frac{1}{R(s)^2} -\frac{1}{R(s_0)^2} + O(h) \geq 0
 $$ for $h<h_0$ for $h_0$ only depending on the geometry, and not on $M$. Therefore, we have obtained
\bnan
\label{recap1}
\int_{\mathcal{S}_+ \setminus \calS_M}\left( \frac{1}{R(s)^2}-  |\phi'(s)|^2 -\mu_h\right)  e^{2\phi/h} |\psi_h|^2 d \Vol_g
& \leq & Ch  e^{2C_0}.
\enan
 Next, on $\Omega_+ \setminus \Omega_M$, we have $\phi' = d_A' - C_0 h \frac{d_A'}{d_A}$ and hence
\bna
\frac{1}{R(s)^2}- |\phi'|^2- \mu_h  
& = & - h\sqrt{\frac{|R''(s_0)|}{R^3(s_0)}} + O(h^{\frac32}) + 2C_0 h \frac{(d_A')^2}{d_A}  - C_0^2 h^2\frac{(d_A')^2}{d_A^2}  \\
& \geq & - h\sqrt{\frac{|R''(s_0)|}{R^3(s_0)}} + O(h^{\frac32}) + C_0 h \frac{(d_A')^2}{d_A} 
\ena
where we used that $d_A\geq C_0 h$.
According to~\eqref{equivdAs0}, $ \frac{(d_A')^2}{d_A} \to 2 \sqrt{\frac{-R''(s_0)}{R(s_0)^3}}>0$ and $  \frac{(d_A')^2}{d_A}$ can thus be extended by continuity at $s_0$. Since $d_A'(s) = 0$ iff $s=s_0$ ($R$ reaches at $s_0$ its {\em unique} global maximum), the extended function is uniformly bounded from below on any compact of $(0, L)$. Moreover, according to~\eqref{equivdAlog}, we have
$$
\frac{(d_A')^2}{d_A}(s) \sim_{s \to 0^+} \frac{1}{s^2\log(s^{-1})} , \quad \text{and} \quad \frac{(d_A')^2}{d_A}(s) \sim_{s \to L^-} \frac{1}{(L-s)^2\log((L-s)^{-1})} .
$$
Hence, there is a constant $C_1>0$ such that $\frac{(d_A')^2}{d_A}\geq C_1$ on $(0, L)$, and we have
\bna
\frac{1}{R(s)^2}- |\phi'|^2- \mu_h 
& \geq & h \left(C_0 \frac{(d_A')^2}{d_A} - \sqrt{\frac{|R''(s_0)|}{R^3(s_0)}} + O(h^{\frac12}) \right) \geq  C_0 h \frac{(d_A')^2}{2d_A},
\ena
 when taking $C_0$ large w.r.t. $C_1^{-1}$ and $h\leq h_0$ with $h_0$ depending on $C_0,C_1$. We can now fix $C_0, h_0$. After \eqref{recap1}, we have thus obtained
\bna
C h \int _{\mathcal{S}_+ \setminus \calS_M} \frac{(d_A')^2}{d_A} e^{2\phi/h}|\psi_h|^2 d \Vol_g   
\leq Ch  e^{2C_0} .
\ena
Our next task is to replace $\phi$ by $d_A$ in this expression. Note that $e^{2\phi(s)/h} = e^{2d_A(s)/h}\left(\frac{h}{d_A(s)} \right)^{2C_0}$. 
In particular, this yields
$$
C h  \int _{\mathcal{S}_+ \setminus \calS_M} \frac{(d_A')^2}{d_A}e^{2d_A(z)/h}\left(\frac{h}{d_A(s)} \right)^{2C_0} |\psi_h|^2 d \Vol_g  
\leq Ch .
$$
Now, the function $\frac{(d_A')^2}{d_A^{1+2C_0}}$ is positive on $(0,s_0) \cup (s_0,L)$, tends to $+\infty$ at $s_0$, and satisfies, as above
$$
\frac{(d_A')^2}{d_A^{1+2C_0}} \sim \frac{1}{s^2(\log(s^{-1}))^{1+2C_0}} \to + \infty , \text{ as } s\to 0^+, 
$$
and similarly $\frac{(d_A')^2}{d_A^{1+2C_0}} \sim \frac{1}{(L-s)^2(\log((L-s)^{-1}))^{1+2C_0}} \to + \infty,\text{ as } s\to L^-$.
Hence, it is bounded from below on $(0, L)$ by a constant, and we obtain
$$
\int_{\mathcal{S}_+ \setminus \calS_M} e^{2d_A(z)/h} |\psi_h|^2d \Vol_g 
\leq Ch^{-2C_0} ,
$$
which, combined with the already remarked fact that $\int_{\mathcal{S}_-} e^{2d_A(z)/h} |\psi_h|^2 \Vol_g  \leq Cte$, gives
$$
\int_{\calS \setminus \mathcal{S}_M} e^{2d_A(z)/h} |\psi_h|^2d \Vol_g
\leq Ch^{-2C_0}.
$$
Since all the constants are independent on $M$, it gives the sought result by dominated convergence making $M$ tends to infinity.
\enp

\subsection{The disk}
\label{sectDisk}
Denote $\ID=\left\{(x,y)\in \R^2\left|x^2+y^2\leq 1\right.\right\}\subset \R^2$ the unit disk. We denote by $\Delta$ the (negative) flat Laplace operator in $\R^2$.
In polar coordinates, $x=r \cos\theta$, $y=r\sin\theta$, we have 
$$
\Delta = \partial_x^2+\partial_y^2=\frac{\d^2}{\d r^2}+\frac{1}{r}\frac{\d}{\d r}+\frac{1}{r^2}\frac{\d^2}{\d \theta^2} .
$$
Then, it can be seen that 
\bnan
\label{e:disk-eig}
\psi_{n,k}(r,\theta)=J_n(z_{n,k}r)e^{in\theta}
\enan
is an orthogonal basis of $L^2(\ID)$, where
\begin{itemize}
\item $J_n$ is the Bessel function of order $n$, namely:
\bnan
\label{e:bessel}
J_n(z) = \frac{1}{2\pi} \int_{-\pi}^{\pi} e^{i z \sin \theta} e^{- i n \theta} d \theta , \quad n \in \Z ,  z \in \C \setminus \R_- ,
\enan
\item $0<z_{n,1}<z_{n,2}<z_{n,3}<\cdots$ is the sequence of the positive zeros of $J_n$.
\end{itemize}
We refer for instance to~\cite[Chapters~14.4 and~15]{VasyBook} for an elementary introduction.
In particular, the functions defined in~\eqref{e:disk-eig} satisfy
$$
-\Delta \psi_{n,k}=\lambda_{n,k}\psi_{n,k} \ \text{ in } \Int(\ID), \quad \text{ with } \lambda_{n,k}=z_{n,k}^2 \quad \text{ and }\psi_{n,k} |_{\d\ID} = 0  .
$$

Roughly speaking, the index $n$ encodes the oscillation in the $\theta$ variable while the index $k$ will contain an oscillation in the radial variable. We refer to~\cite{ALM:cras} for a description of concentration/delocalization properties of general eigenfunctions (or, more generally, quasimodes) on the disk. Here, we want to analyse some eigenfunctions corresponding to the so-called whispering gallery modes that are concentrated close to the boundary of $\ID$. They ``rotate'' very fast and concentrate towards one of the two trajectories of the billiard contained in $S^*\d \ID$. This phenomenon corresponds to $n \to + \infty$ and $k$ small, typically $k=1$. 
In the following, we thus focus on: 
\bna
\psi_{n,1}(r,\theta)=J_n(z_{n,1}r)e^{in\theta} ,
\ena
and hence on the function  $J_n(z_{n,1}r)$. This requires information on $z_{n,1}$.

A huge amount of information is known on the Bessel functions ant its zeros. But we will need very few of them. First, we need to normalize them. This is for instance done in Lemma 5.1 of Burq-G\'erard-Tzvetkov \cite{BGT:03} in the case $k=1$ which is of interest for us.
\bna
\nor{\psi_{n,1}}{L^2(D)}\approx n^{-\frac{2}{3}}.
\ena
We also need a rough estimate on the asympotic of the $z_{n,1}$, see \cite{BGT:03} Lemma 4.3 for instance, namely, 
\bna
z_{n,1}=n+\O(n^{1/3}) , \qquad z_{n,1} > n .
\ena
To estimate the norm of $\psi_{n,1}$ on $B(0,\eps)$, $\eps<1$, we first prove the following lemma.
\begin{lemma}
\label{l:bessel-asympt}
For all $\alpha \geq 0$ and $n\in \N$, we have
\bna
\left|J_n\left(\frac{n}{\cosh(\alpha)}\right) \right| \leq e^{n(\tanh(\alpha)-\alpha)} .
\ena
\end{lemma}
Note that in \cite[Section 32 p79]{CopsonBook}, for fixed $\alpha$, a full asymptotics in terms of $n$ is proved, with principal term:
\bnan
\label{e:copson}
J_n\left(\frac{n}{\cosh(\alpha)}\right) \approx \frac{e^{n(\tanh(\alpha)-\alpha)}}{\sqrt{2\pi n \tanh(\alpha)}} .
\enan
Here, we need only the principal term but also a uniform bound in terms of $\alpha$. Note that the short proof below is not very informative, and the reader is referred to~\cite[Section 32]{CopsonBook} for a complete steepest descent approach to this asymptotic expansion.
\bnp[Proof of Lemma~\ref{l:bessel-asympt}]
We start from formula~\eqref{e:bessel}, in which we write $\nu = \frac{n}{\cosh(\alpha)}$, and use the holomorphy of the integrand, together with the fact that $e^{i \nu\left( \sin z -  z \cosh\alpha \right)}$ is a periodic function of $\Re(z)$ to change the contour. This yields:
\begin{align*}
J_n\left(\nu\right) &= \frac{1}{2\pi} \int_{-\pi}^{\pi} e^{i \left(\frac{n}{\cosh(\alpha)}\right) \sin \theta} e^{- i n \theta} d \theta 
= \frac{1}{2\pi} \int_{-\pi}^{\pi} e^{i \nu\left( \sin \theta - \theta \cosh\alpha  \right)} d \theta \\
&= \frac{1}{2\pi} \int_{-\pi - i \alpha}^{\pi - i \alpha} e^{i \nu\left( \sin z -  z \cosh\alpha \right)} d z  
= \frac{1}{2\pi} \int_{-\pi}^{\pi} e^{i \nu\left( \sin x \cosh\alpha - i \cos x \sinh\alpha - x\cosh\alpha + i \alpha \cosh\alpha \right)} d x .
\end{align*}
This implies 
\begin{align*}
| J_n\left(\nu\right)| &\leq \frac{1}{2\pi} \int_{-\pi}^{\pi} e^{\nu \left( \cos x \sinh\alpha -  \alpha \cosh\alpha \right)} d x \leq  e^{\nu \left(  \sinh\alpha -  \alpha \cosh\alpha \right)} = e^{n \left(  \tanh\alpha -  \alpha \right)} ,
\end{align*}
and concludes the lemma.
\enp

\begin{lemma}
There exist $C, \beta , n_0 >0$  such that for all $n \geq n_0$ and $0< r \leq 1-\beta n^{-2/3}$, we have
\bna
\|\psi_{n,1}\|_{L^\infty(B(0,r))} \leq \exp \left(-n d_A(r)  + C n^{1/3} \right) .
\ena
\end{lemma}
Note that for $r \in (0,1)$ fixed, the asymptotic formula~\eqref{e:copson} implies that such eigenfunctions have indeed the decay prescribed by this formula. 
\bnp
We have $\frac{z_{n,1}}{n} = 1 + O(n^{-2/3})$ and $\frac{z_{n,1}}{n} >1$. 
Hence recalling that $|d_A'|$ is decreasing on $(0,1]$, we have, as long as $\frac{r z_{n,1}}{n} \leq 1$, 
$$
\left|d_A(\frac{r z_{n,1}}{n}) - d_A(r) \right| \leq C n^{-2/3} r |d_A'(r)|= C n^{-2/3}r \sqrt{\frac{1}{r^2}-1} =C n^{-2/3} \sqrt{1-r^2} . 
$$
Thus we obtain from Lemma~\ref{l:bessel-asympt} 
$$
|J_n(z_{n,1}r)|= |J_n(n\frac{z_{n,1}}{n}r)| \leq \exp \left(-n d_A(\frac{z_{n,1}}{n}r)\right) \leq \exp \left(-n d_A(r)  + C n^{1/3} \right)
$$
for all $n \in \N$ and $0< r \leq \frac{n}{z_{n,1}}$.
\enp
The combination of the previous estimates give Theorem \ref{t:agmon-diskintro}.
\section{Maximal vanishing rate of sums of eigenfunctions, and observability on small balls}
\label{s:unif-LR-ineq}
In this section, we prove Theorem~\ref{t:unif-LR-ineq}, i.e. the Lebeau-Robbiano spectral inequality with observation in balls of (small) radius $r$ and constants uniform in $r$.

We follow the proof proposed by Jerison and Lebeau in~\cite[middle of p231]{JL:99}. There are three main steps, that we summarize in three lemmata. We then prove Theorem~\ref{t:unif-LR-ineq} from these lemmata, and prove the lemmata afterwards.

In the following, for $\beta>0$, we set $X_\beta= (-\beta , \beta) \times \M$, and denote $P = -\d_{s}^2 - \Delta_g$. 
In the set $X_{2S} =(-2S , 2S) \times \M$, we denote by $(s, x)$ the running point and by $B_{r}$ a geodesic ball (for the metric $\id \otimes g$) of radius $r$ (its center being implicit in the notation).
We also use the rescaled $H^1$ norm on an open set $U$, denoted $H_r^1(U)$ and defined by 
\bnan
\label{e:def-norm-r}
\|F\|_{H_r^1(U)}^2  = \|F\|_{L^2(U)}^2 + r^2  \|\nabla_g F\|_{L^2(U)}^2 .
\enan
 This will only be used on small geodesic balls or annuli, namely $U = B_{\alpha r}$ or $U = B_{\alpha r}\setminus B_{\beta r}$.

\subsection{The three key lemmata}
In this section, we state the three key lemmata needed for the proof of Theorem~\ref{t:unif-LR-ineq}.

The first lemma is a classical global Lebeau-Robbiano interpolation inequality, \cite[Section~3, Estimate~(1)]{LR:95}.

\begin{lemma}[Global interpolation inequality from unit balls to the whole space]
\label{t:global-interp-LR}
Let $S>0$ and let $U \subset X_{2S}$ be any nonempty open set, then there is $C>0$ and $\alpha_0 \in (0,1)$ such that we have
\begin{equation*}
\|F\|_{H^1(X_{S})} \leq C \left(\| P F\|_{L^2(X_{2S})} +  \| F \|_{H^1(U)}  \right)^{\alpha_0} \|F\|_{H^1(X_{2S})}^{1-\alpha_0} .
\end{equation*}
for all $F \in H^2(X_{2S})$ such that $F|_{(-2S,2S)\times \d M}=0$.
\end{lemma}

The next lemma states a local interpolation inequality. Its specificity is that the observation  term is on a small ball $B_r$ and the constants are uniform in $r$ small. For this, the exponent has to depend on $r$ as $|\log(r)|^{-1}$.
\begin{lemma}[Local interpolation inequality from small balls to unit balls]
\label{l:aronsajn-arleman-interp}
Let $P= -\d_s^2 - \Delta_g$ and let $B_r$ denote balls centered at $(s_0, x_0) \in X_T$, away from the boundary.
Then, there exists constants $r_1 >0$ such that for all $0< r_0 \leq r_1$, there is $C>0$ such that for all $r \in (0, \frac{r_0}{10})$, and $F\in H^2(B_{r_0})$, we have
\begin{align*}
\|F\|_{H^1(B_{\frac{r_0}{4}})} \leq  C\left( \| P F \|_{L^2(B_{r_0})}+ \|F\|_{H^1_r(B_r)}  \right)^{\alpha_r} \|F\|_{H^1(B_{r_0})}^{1-\alpha_r } ,
 \quad \alpha_r = \frac{\log 2}{\log \left(\frac{2r_0}{r}\right) + \log 2}.
\end{align*}
\end{lemma}
A proof of this Lemma is given in Section~\ref{s:proof-from-aronszajn}, starting from a Carleman estimate (with singular weight) due to Aronszajn~\cite{Aronsjajn:57} (see also \cite{AKS:62,DF:88,DF:90}).

The last lemma is an interpolation inequality with boundary observation term. All terms are taken on sets of size $r$, and the important feature of this estimate is that the constants are uniform in $r$.
\begin{lemma}[Uniform local interpolation at the boundary on small balls]
\label{t:unif-LR}
Let $(0 , x_0) \in \{0\} \times \M$, $ \dist_g(x_0, \d \M)>0$ (all balls are centered in $x_0$). Then, there exists $C>0$, $r_0>0$ and $\alpha_0 \in (0,1)$ such that we have for all $0<r<r_0$
\begin{equation*}
\|F\|_{H^1_r(B_r)} \leq C \left(r^2 \|P F\|_{L^2(B_{2r})} + r^{3/2} \| \d_s F|_{s=0} \|_{L^2(B_{2r}\cap \{0\}\times \M)}  \right)^{\alpha_0} \|F\|_{H^1_r(B_{2r})}^{1-\alpha_0}
\end{equation*}
for all $F \in H^2(X_{2S})$ such that $F|_{(-2S,2S)\times \d M}=0$. 
\end{lemma}
This lemma is proved in Section \ref{subsectLmscaling}, consequence of a uniform Carleman estimate proved in Appendix~\ref{s:carleman}.

\subsection{Concluding the proof of Theorem~\ref{t:unif-LR-ineq} from the three lemmata}

From these three lemmata, we may now give a proof of Theorem~\ref{t:unif-LR-ineq}. We first formulate a straightforward corollary of the three lemmata to prepare the proof.
\begin{corollary}
Let $P= -\d_s^2 - \Delta_g$ 
and $(0 , x_0) \in \{0\} \times \Int( \M)$ and consider balls centered at $(0,x_0)$. Then, there exists $r_0 >0$, $C>0$ and $\alpha_0 \in (0,1)$ such that, for all $r \in (0, \frac{r_0}{10})$ and $F\in H^2(X_{2S})$ with $PF=0$ and $F|_{(-2S,2S)\times \d M}=0$, we have 
\begin{align*}
\|F\|_{H^1(X_{S})} & \leq C \| F \|_{H^1(B_{\frac{r_0}{4}})}^{\alpha_0} \|F\|_{H^1(X_{2S})}^{1-\alpha_0} , \\
 \|F\|_{H^1(B_{\frac{r_0}{4}})} &\leq  C \|F\|_{H^1(B_r)}^{\alpha_r} \|F\|_{H^1(X_{2S})}^{1-\alpha_r } ,
 \quad \alpha_r = \frac{\log 2}{\log \left(\frac{2r_0}{r}\right) + \log 2}  ,\\
 \|F\|_{H^1(B_r)} &\leq C  \|\d_s F|_{s=0} \|_{L^2(B_{2r}\cap \{0\}\times \M)}^{\alpha_0} \|F\|_{H^1(X_{2S})}^{1-\alpha_0} .
\end{align*}

\end{corollary}

\begin{proof}[Proof of Theorem~\ref{t:unif-LR-ineq}]
Let us first treat the case where $\d \M = \emptyset$, or $\d \M \neq \emptyset$ but the center of the balls, $x_0$ is in $\Int(\M)$. The case $x_0$ near $\d \M$ will be treated afterwards.
 
We reformulate (again) these three results as (in a form close to that of~\cite{DF:88})
\begin{align*}
\frac{ \|F\|_{H^1(X_{2S})}}{\|F\|_{H^1(B_{\frac{r_0}{4}})}} 
& \leq  \left( C \frac{ \|F\|_{H^1(X_{2S})}}{\|F\|_{H^1(X_S)}}  \right)^{\frac{1}{\alpha_0}} , \\
\frac{ \|F\|_{H^1(X_{2S})}}{\|F\|_{H^1(B_r)}} 
& \leq  \left( C \frac{ \|F\|_{H^1(X_{2S})}}{\|F\|_{H^1(B_{\frac{r_0}{4}})}}  \right)^{\frac{1}{\alpha_r}} , \\
\frac{ \|F\|_{H^1(X_{2S})}}{\|\d_s F|_{s=0} \|_{L^2(B_{2r}\cap \{0\}\times \M)}} 
& \leq  \left( C \frac{ \|F\|_{H^1(X_{2S})}}{\|F\|_{H^1(B_r)}}  \right)^{\frac{1}{\alpha_0}} , 
\end{align*}
and combine them to obtain 
\begin{align}
\label{e:interp-interp}
\frac{ \|F\|_{H^1(X_{2S})}}{\|\d_s F|_{s=0} \|_{L^2(B_{2r}\cap \{0\}\times \M)}} 
& \leq  C^{\frac{1}{\alpha_0}} C^{\frac{1}{\alpha_0\alpha_r}}  C^{\frac{1}{\alpha_0^2\alpha_r}} \left(  \frac{ \|F\|_{H^1(X_{2S})}}{\|F\|_{H^1(X_S)}}  \right)^{\frac{1}{\alpha_0^2\alpha_r}} .
\end{align}
We then follow~\cite{LR:95,JL:99,LZ:98,LeLe:09}, and, given $\psi \in E_{\leq \lambda}$ take the function
$$
F(s) = \frac{\sinh(s \sqrt{-\Delta_g})}{\sqrt{-\Delta_g}} \Pi_+ \psi + s \Pi_0 \psi ,
$$
where $\Delta_g$ is the Dirichlet Laplacian, $\Pi_0$ the orthogonal projector on $\ker(\Delta_g)$ and $\Pi_+ = \id-\Pi_0$, that is $F$ is the unique solution to 
$$
(-\d_s^2-\Delta_g)F = 0 , \quad F|_{(-2S,2S)\times \d M}=0,  \quad (F, \d_s F)|_{s=0} = (0 , \psi) .
$$
Classical computations (see e.g.~\cite[Proof of Theorem~5.4]{LeLe:09}) show that there is $C>1$ such that for all $\lambda\geq 0$ and $\psi \in E_{\leq\lambda}$, we have
$$
\frac{1}{C}\|\psi\|_{L^2(\M)} \leq \|F\|_{H^1(X_{S})} \leq \|F\|_{H^1(X_{2S})} \leq C e^{3S \sqrt{\lambda}} \|\psi\|_{L^2(\M)} .
$$
As a consequence, \eqref{e:interp-interp} yields for some $C, \kappa>0$, for all $\lambda\geq 0$, $\psi \in E_{\leq\lambda}$, and $r \in (0,\frac{r_0}{4})$ 
\begin{equation}
\label{e:JL-almost-finished}
\frac{\|\psi\|_{L^2(\M)}}{\|\psi\|_{L^2(B_\M(x_0, 2r))}} \leq C^{\kappa + \frac{1}{\alpha_r}}  e^{(\kappa + \frac{1}{\alpha_r})\sqrt{\lambda}} .
\end{equation}
Recalling the definition of $\alpha_r$, this is the sought result of Theorem~\ref{t:unif-LR-ineq} (up to changing $2r$ into $r$, and the names of the constants accordingly) with the restriction $r \in (0,\frac{r_0}{4})$. To conclude for all $r>0$, it suffices to notice that \eqref{e:JL-almost-finished} remains true with $\alpha_{\frac{r_0}{16}}$ on the r.h.s. uniformly for observation terms $\|\psi\|_{L^2(B_\M(x_0, 2r))}$ with $r \geq \frac{r_0}{8}$ (the constants are non-increasing functions of the observation set).

\bigskip
To conclude the proof in the general case, we need to consider the situation $\d \M \neq \emptyset$ in full generality. We again follow~\cite{DF:88,JL:99}. 
In this case, we define the double manifold $\widetilde{\M} = \M \sqcup \M$, consisting in gluing two copies of $\M$, endowed with a smooth structure of compact manifold, as in~\cite[Theorem~9.29-Example~9.32]{Lee:book}. Then, the procedure very well explained in~\cite[Section~3]{Anton:08} and we only sketch the proof.
We extend the metric $g$ on $\M$ by symmetry/parity with respect to the boundary $\d\M$ as a metric $\tilde{g}$ on $\widetilde{\M}$. Note that even if $g$ is smooth, the extended metric $\tilde{g}$ is only Lipschitz on $\widetilde{\M}$. This is not an issue since the three lemmata~\ref{t:global-interp-LR},~\ref{l:aronsajn-arleman-interp} and~\ref{t:unif-LR} remain valid for Lipschitz metrics (as a consequence of Appendix~\ref{s:carleman}, \cite{AKS:62,DF:90}, and Appendix~\ref{s:carleman}, respectively).
In the case of Dirichlet boundary condition on $\d \M$, and given $\psi \in E_{\leq \lambda}$ we take its anti-symmetric/odd extension on $\widetilde{\M}$, yielding a function $\tilde{\psi} \in \tilde{E}_{\leq \lambda}$. Here, $ \tilde{E}_{\leq \lambda}$ is the counterpart of $E_{\leq \lambda}$ defined for the Laplace-Beltrami operator $\Delta_{\tilde{g}}$ on $\widetilde{\M}$.
The above computations are then made for $\Delta_{\tilde{g}}$ on $\widetilde{\M}$ and the estimate~\eqref{e:JL-almost-finished} is proved for $\tilde{\psi}$. The same estimate for $\psi$ follows.
Similarly, in the case of Neumann boundary condition, we take the symmetric/even extension of functions, yielding the sought result.
\end{proof}

\subsection{A proof of Lemma~\ref{l:aronsajn-arleman-interp} from Aronszajn estimates}
\label{s:proof-from-aronszajn}
In section, we give a proof of Lemma~\ref{l:aronsajn-arleman-interp} starting from Carleman-Aronszajn estimates as stated in~\cite[Proposition~2.10]{DF:88} and~\cite[Proposition~2.10]{DF:90} (and slightly modified according to the remarks in~\cite[Beginning of Section~14.3]{JL:99}), which we now state. An alternative proof of a closely related estimate is given by H\"ormander in~\cite[Inequality~(17.2.11), Chapter~XVII.2]{Hoermander:V3}.
\begin{proposition}
\label{p:carleman-aronsajn}
Let $P= -\d_s^2 - \Delta_g$ and let $(\rho, t) \in (0,r_1) \times \S^{n}$ be geodesic polar coordinates around a point $(s_0, x_0) \in X_S$ away from the boundary.
Then, there exists a function $\bar{\rho}(\rho)$ with 
\bnan
\label{e:rho=rhobar}
\bar{\rho}=\rho + O(\rho^2) , \quad \text{ as } \rho \to 0^+ ,
\enan 
and constants $\tau_0, C, r_0 >0$, such that we have 
\bna
C \int |\bar{\rho}^{-\tau}Pu|^2 \rho^{-1} d\rho dt \geq  \int \left(|\bar{\rho}^{-\tau}\nabla u|^2 + |\bar{\rho}^{-\tau}u|^2 \right)\rho^{-1}  d\rho dt , \quad \text{for all } \tau \geq \tau_0 , \quad  u \in C^\infty_0(B_{r_0} \setminus \{0\} ).
\ena
\end{proposition}
With this Carleman-Aronszajn estimate in hand, we now give a proof of Lemma~\ref{l:aronsajn-arleman-interp}.
\bnp[Proof of Lemma~\ref{l:aronsajn-arleman-interp}]
We use the estimate of Proposition~\ref{p:carleman-aronsajn} as in~\cite{LR:95} (see also~\cite[Section~5]{LeLe:09}) to deduce an interpolation inequality. We introduce for this (as in~\cite[Beginning of Section~3]{DF:88}) a cutoff function $\chi_\rr = \chi_\rr(\rho)$ such that, with $0<\rr < \frac{r_0}{2}$ a small parameter (appearing in the statement of the lemma)
\bna
\supp(\chi_\rr ) \subset \left\{ \frac{\rr}{2} < \bar{\rho} < r_0  \right\}, \quad \chi_\rr = 1 \text{ on } \left\{ \rr < \bar{\rho} < \frac{r_0}{2} \right\} , \\
 |\d^{\alpha} \chi_\rr| \leq C_\alpha \rr^{-|\alpha|} \text{ on } \left\{ \frac{\rr}{2} < \bar{\rho} < \rr \right\},
  \quad  |\d^{\alpha} \chi_\rr| \leq C_\alpha  \text{ on } \left\{ \frac{r_0}{2} < \bar{\rho} < r_0 \right\}.
\ena
We apply Proposition~\ref{p:carleman-aronsajn} to $u = \chi_\rr F$. The operator $[P, \chi_\rr]$ is a first order differential operator with $\supp [P, \chi_\rr] \subset  \left\{ \frac{\rr}{2} < \bar{\rho} < \rr  \right\} \cup \left\{  \frac{r_0}{2}  < \bar{\rho} <r_0 \right\}$, being moreover of the form $O(r^{-1})D + O(r^{-2})$ on the set $\left\{ \frac{\rr}{2} < \bar{\rho} < \rr  \right\}$. Therefore, we obtain using~\eqref{e:rho=rhobar}, for all $\tau \geq \tau_0$
\begin{align*}
 \int \left(|\bar{\rho}^{-\tau}\nabla( \chi_\rr F)|^2 + |\bar{\rho}^{-\tau}\chi_\rr F|^2 \right)\rho^{-1}  d\rho dt &  \leq C \int |\bar{\rho}^{-\tau}\chi_\rr P F|^2 \rho^{-1} d\rho dt + C \int |\bar{\rho}^{-\tau}[P ,\chi_\rr] F|^2 \rho^{-1} d\rho dt  \\
&  \leq C\left(\frac{\rr}{2}\right)^{-2\tau-1} \| P F \|_{L^2(\bar{B}_{r_0})}^2 + C\left(\frac{\rr}{2}\right)^{-2\tau-2} \|F\|_{H^1_r(\frac{\rr}{2}\leq \bar{\rho} \leq \rr)}^2  \\
& \quad + C \left(\frac{r_0}{2}\right)^{-2 \tau}\|F\|_{H^1(\frac{r_0}{2}\leq \bar{\rho} \leq r_0)}^2 ,
\end{align*}
where $\bar{B}_{r_0}$ denotes the set $\{\bar{\rho}\leq r_0\}$. Recall that the norm $H^1_r$ is defined in~\eqref{e:def-norm-r}.
Concerning the left hand-side, we bound it from below by 
\begin{align*}
 \int \left(|\bar{\rho}^{-\tau}\nabla( \chi_\rr F)|^2 + |\bar{\rho}^{-\tau}\chi_\rr F|^2 \right)\rho^{-1}  d\rho dt & \geq \int_{2\rr \leq \bar{\rho} \leq \frac{r_0}{4}} \left(|\bar{\rho}^{-\tau}\nabla( \chi_\rr F)|^2 + |\bar{\rho}^{-\tau}\chi_\rr F|^2 \right)\rho^{-1}  d\rho dt \\
& \geq \left(\frac{r_0}{4}\right)^{-2 \tau} \|F\|_{H^1(\rr \leq \bar{\rho} \leq \frac{r_0}{4})}^2  .
\end{align*}
Combining the last two estimates together with the fact that $\left(\frac{r_0}{4}\right)^{-\tau} \|F\|_{H^1(\bar{B}_\rr)} \leq \left(\frac{\rr}{2}\right)^{- \tau} \|F\|_{H^1(\bar{B}_\rr)}$ yields, for some $\tau_0>0$ and all $\tau\geq \tau_0$ and $\rr \in (0, \frac{r_0}{10})$,
\begin{align*}
\left(\frac{r_0}{4}\right)^{-\tau} \|F\|_{H^1(\bar{B}_{\frac{r_0}{4}})} \leq  C\left(\frac{\rr}{2}\right)^{-\tau} \left( \| P F \|_{L^2(\bar{B}_{r_0})}+ \|F\|_{H^1_r(\bar{B}_\rr)}  \right)
+ C \left(\frac{r_0}{2}\right)^{- \tau}\|F\|_{H^1(\bar{B}_{r_0})} .
\end{align*}
Multiplying by $r_0^\tau$ and recalling~\eqref{e:rho=rhobar} to replace balls in $\bar{\rho}$ by real balls, we obtain, up to changing the names of the parameters $\rr, r_0$, that
\begin{align*}
\|F\|_{H^1(B_{\frac{r_0}{4}})} \leq  C\left(\frac{2r_0}{\rr}\right)^{\tau} \left( \| P F \|_{L^2(B_{r_0})}+ \|F\|_{H^1_r(B_\rr)}  \right)
+ \frac{C}{2^\tau}\|F\|_{H^1(B_{r_0})} .
\end{align*}
An optimization in $\tau \geq\tau_0$~\cite{Robbiano:95} (see also~\cite[Lemma~5.2]{LeLe:09}), then implies the following interpolation inequality
\begin{align*}
\|F\|_{H^1(B_{\frac{r_0}{4}})} \leq  C\left( \| P F \|_{L^2(B_{r_0})}+ \|F\|_{H^1_r(B_\rr)}  \right)^{\alpha_\rr} \|F\|_{H^1(B_{r_0})}^{1-\alpha_\rr } ,
 \quad \alpha_\rr = \frac{\log 2}{\log \left(\frac{2r_0}{\rr}\right) + \log 2},
\end{align*}
 and concludes the proof of the lemma.
\enp

\subsection{A proof of Lemma~\ref{t:unif-LR} from Proposition~\ref{p:unif-interp-boundary-metric}}
\label{subsectLmscaling}
In this section, we give a proof of Lemma~\ref{t:unif-LR}. The latter consists in performing a scaling argument to reduce the problem to fixed-size balls. However, the scaling argument yields in these fixed balls a family of metrics (converging to a fixed metric as $r\to 0$), and we need to use uniform interpolation/Carleman estimates for such families of metrics. These uniform estimates are proved in Appendix~\ref{s:carleman} (Proposition~\ref{p:unif-interp-boundary-metric}).

\bnp[Proof of Lemma~\ref{t:unif-LR}]
We first choose $r_0$ small enough so that $\overline{B_{2r_0}} \subset X_S$ and there exist local coordinate patch on $\M$ : $\Phi: \{x \in \M, \dist(x,x_0)<2r_0\}\to U$  where $U$ is a neighborhood of $0$ in $\R^{n}$, with $\Phi(x_0) = 0$. Up to a multiplication by an invertible constant matrix, we may assume that $\left((\Phi^{-1})^*g\right) (0) = \id$.
As a consequence, $ds^2 \otimes \left((\Phi^{-1})^*g\right) (ry)$, defined on the ball of radius $2$, converges uniformly in this ball towards the flat metric on the flat ball of $\R^{n+1}$ in the limit $r \to 0^+$.
 We will thus only use the flat metric in the present proof which behaves well with respect to scaling. The distance (hence the balls, still denoted $B_r$ or $B_1$ below, all centered at $0$) will be defined with respect to the flat metric, as well as the Sobolev norms (still denoted $H^1_r(B_r)$, $H^1(B_1)$ below).
 The final result we obtain will be formulated in terms of the flat metric,  and associated balls and Sobolev spaces. Coming back to a formulation on the manifold $\R \times \M$ with the metric $ds^2 \otimes g$ only uses the uniform equivalence of norms in $T^*(\R \times \M)$ and in $L^2(\R \times \M)$ for $r$ sufficiently small.

With this in mind, let us now proceed with the scaling argument in the coordinate chart. Denote by $F_r(x)=F(rx)$ and $P_r$ the Laplace-Beltrami operator with respect the metric $ds^2 \otimes \left((\Phi^{-1})^*g\right)(ry)$ defined on the ball of radius $2$, we have 
\begin{align*}
\|F\|_{H^1_r(B_r)}&=r^{(n+1)/2}\|F_r\|_{H^1(B_1)}, \\
r^2 \|P F\|_{L^2(B_{2r})}&= r^{(n+1)/2}\|P_r F_r\|_{L^2(B_{2})}, \\
 r^{3/2} \| \d_s F|_{s=0} \|_{L^2(B_{2r}\cap \{0\}\times \M)} &=r^{1/2}r^{n/2} \| \d_s F_r|_{s=0} \|_{L^2(B_{2}\cap \{0\}\times \M)}.
\end{align*}
Note that the metric $ds^2 \otimes g(r\cdot)$ defined on $B_2$ converges uniformly for $r$ converges to zero to the flat metric $ds^2 \otimes g(0) = ds^2\otimes dy_1^2\otimes \cdots \otimes dy_n^2$ for the Lipschitz topology on metrics. 
So, the result follows if we are able to prove the following estimate: there exist $\epsilon, \alpha_0 , C$ such that for all  Lipschitz metric $\mathfrak{g}$ with $\|\mathfrak{g}- \id\|_{W^{1,\infty}}<\epsilon$ and all $u\in H^2(B_2)$ such that $u|_{s=0}=0$, we have 
\begin{equation*}
\|u\|_{H^1(B_1)} \leq C \left(\|(-\d_s^2 - \Delta_{\mathfrak{g}}) u\|_{L^2(B_2)} +  \| \d_s u|_{s=0} \|_{L^2(B_2\cap \{0\}\times \R^{n})}  \right)^{\alpha_0} \|F\|_{H^1(B_2)}^{1-\alpha_0}.
\end{equation*}
This is the object of Proposition~\ref{p:unif-interp-boundary-metric} proved in the Appendix. Note that the result of Proposition~\ref{p:unif-interp-boundary-metric} is stated with half-balls $B_k^+$ but is also true with real balls $B_k$ instead by a symmetry argument.
\enp

\section{The observability constant for positive solutions}
\label{s:positive}
The aim of this Section is to prove the positive result of Theorem \ref{thmpositive} about the observability of positive solutions. The main tool will be the following Li-Yau estimates.
\begin{theorem}[Theorem 2.3 of Li-Yau \cite{LY:86}]
\label{t:Li-Yau}
Let $\M$ be a compact manifold. Let $$- K = \min(0, \min_{x \in \M} Ricc(x) ) \leq 0 ,$$ where $Ricc(x)$ is the Ricci curvature at $x$.  
We assume that the boundary of $\M$ is convex, i.e. $II>0$. Let $u(t,x)$ be a positive solution on $(0, +\infty)$ of the heat equation
with Neumann boundary condition. Then for any $\alpha>1$, $x, y\in \M$, and $0<t_1<t_2$, we have
\bna
u(t_1,x)\leq \left(\frac{t_2}{t_1}\right)^{n\alpha/2}e^{\frac{n\alpha K(t_2-t_1)}{\sqrt{2}(\alpha-1)}}e^{\alpha\frac{ d(x,y)^2}{4(t_2-t_1)}}u(t_2,y) .
\ena
\end{theorem}

\begin{remark}
\label{r:non-conv-Li-Yau}
The convexity assumption is not necessary to obtain a Li-Yau type estimate (if the boundary is smooth), up to a loss in the exponent. Indeed, setting $- H= \min(0, \min_{x \in \d \M} II(x) ) \leq 0$, where $II(x)$ is the second fundamental form of $\d M$ with respect to outward pointing normal, Wang proves in \cite[Theorem 3.1]{W:97} the estimate
\bna
u(t_1,x)\leq \left(\frac{t_2}{t_1}\right)^{C_\alpha}e^{C_\alpha' (t_2-t_1)} e^{\alpha\frac{ d(x,y)^2}{4(t_2-t_1)}}u(t_2,y) , \quad \text{for all } \alpha> (1+H)^2 .
\ena
The proof of Theorem~\ref{thmpositive} shows that the result still holds without the convexity argument, but yields  
\bna
\nor{u(T)}{L^2(\M)}^2\leq  \frac{C_\e}{T} e^{(1+H+\e)^2\frac{\mathcal{L}(M,\omega)^2}{2T}}\int_0^{T}\nor{u(t, \cdot)}{L^2(\omega)}^2~dt ,\\
\nor{u(T)}{L^2(\M)}^2\leq  \frac{C_\e}{T} e^{(1+H +\e)^2\frac{\mathcal{L}(M,z_0)^2}{2T}}\int_0^{T}u(t,z_0)^2~dt ,
\ena
instead of~\eqref{estimpos1}-\eqref{estimpos2} (hence with a loss $(1+H)^2$ in the exponent). We do not know wether this is optimal.
Finally, we did not find any analogue estimate in the case of Dirichlet boundary conditions.
\end{remark}

\bnp[Proof of Theorem \ref{thmpositive}]
Actually, this will appear along the proof that we will need the following asymptotic constants, all depending on the chosen $\e>0$.
Namely, we shall use $\eta_0>0$ arbitrarily small, $r>1$ arbitrarily large, $\lambda \in (0,1)$ arbitrarily close to $1$, and $\alpha>1$ arbitrary close to $1$. Given $\eps>0$, they will all be fixed at the end so that $$ \frac{r\alpha }{(r-1)\lambda}(d_\omega + 3\eta_0 )^2 \leq (1+\e) d_\omega^2 .$$

\medskip
For any $x_0\in\M$ and for any $\eta_0>0$, there exist  $\eta = \eta (x_0, \eta_0) \in (0, \eta_0)$ and $y_0\in\omega$ such that 
\bna
d(x_0,y_0)\leq d_\omega +\eta, \quad \text{ and } \quad B(y_0,\eta)\subset \omega .
\ena
In particular, we have $\M \subset \bigcup_{x_0 \in \M}B(x_0, \eta)$ so that, the compactness of $\M$ yields the following statement: given $\eta_0>0$, there exist a finite set $J$ and families $(x_j)_{j\in J} \in \M^J$,  $(y_j)_{j\in J} \in \omega^J$ and $(\eta_j)_{j \in J} \in (0,\eta_0)^J$ such that
 \bna
\M \subset \bigcup_{j \in J}B(x_j, \eta_j), \quad 
d(x_j,y_j)\leq d_\omega +\eta_j , \quad \text{ and } \quad B(y_j,\eta_j )\subset \omega , \quad \text{for all }j \in J.
\ena
Now,  fix $j \in J$, and take $x\in B(x_j,\eta_j)$ and $y\in B(y_j,\eta_j)\subset \omega$, and we have 
$$
d(x,y)\leq\eta_j +  d_\omega + \eta_j + \eta_j \leq d_\omega + 3 \eta_0 = :d_m .
$$
For $t\in [0,T/r]$, Theorem~\ref{t:Li-Yau} with $t_1=t$ and $t_2=rt_1=rt$ then yields 
\bna
u(t,x)^2\leq r^{n\alpha}e^{\frac{2n\alpha Kt(r-1)}{\sqrt{2}(\alpha-1)}}e^{\frac{\alpha d_m^2}{2(r-1)t}}u(rt,y)^2.
\ena
Denoting
$$
\gamma := \frac{2n\alpha K(r-1)}{\sqrt{2}(\alpha-1)} ,
$$
this may be rewritten as 
\bnan
\label{e:before-int-y}
u(t,x)^2 e^{-\frac{\alpha d_m^2}{2(r-1)t}}\leq r^{n\alpha}e^{\gamma t}u(rt,y)^2.
\enan
We may now integrate this estimate for $x\in B(x_j,\eta_j)$ and $y\in B(y_j,\eta_j)\subset \omega$, 
\bna
e^{-\frac{\alpha d_m^2}{2(r-1)t}}\nor{u(t)}{L^2(B(x_j,\eta_j))}^2
\leq \frac{|B(x_j, \eta_j)|}{|B(y_j, \eta_j)|}r^{n\alpha} e^{\gamma t} \nor{u(rt)}{L^2(B(x_j,\eta_j))}^2
 \leq \frac{|B(x_j, \eta_j)|}{|B(y_j, \eta_j)|}r^{n\alpha} e^{\gamma t} \nor{u(rt)}{L^2(\omega)}^2 .
\ena
Summing all these estimates for $j \in J$ yields, for a constant $C(\eta_0)$ depending only on the geometry of $(\M,g)$, of $\omega$, and the constant $\eta_0$, the inequality 
\bna
e^{-\frac{\alpha d_m^2}{2(r-1)t}}\nor{u(t)}{L^2(\M)}^2
 \leq C(\eta_0) r^{n\alpha} e^{\gamma t} \nor{u(rt)}{L^2(\omega)}^2 .
\ena
 
Given $\lambda \in (0,1)$, integrating this on the interval $t\in [\lambda T/r,T/r]$ yields
\bna
\int_{\lambda T/r}^{T/r}e^{-\frac{\alpha d_m^2}{2(r-1)t}}\nor{u(t)}{L^2}^2 dt
& \leq  &C(\eta_0) r^{n\alpha}  \int_{\lambda T/r}^{T/r}  e^{\gamma t} \nor{u(rt)}{L^2(\omega)}^2 dt \\
& \leq & C(\eta_0) r^{n\alpha}  e^{\gamma \frac{T}{r}}  \int_{\lambda T/r}^{T/r} \nor{u(rt)}{L^2(\omega)}^2 dt 
=  C(\eta_0) r^{n\alpha}  e^{\gamma \frac{T}{r}}  \int_{\lambda T}^{T} \nor{u(s)}{L^2(\omega)}^2 ds,
\ena
after the change of variables $s=rt$. Concerning the left hand-side, we use the decay of the $L^2$ norm of solutions to the heat equation to write
\bnan
\label{e:remplace-r-t}
\nor{u(t)}{L^2(\M)}\geq  \nor{u(T/r)}{L^2(\M)} \geq \nor{u(T)}{L^2(\M)} ,
\enan for all $t\in [\lambda T/r,T/r]$ since $r>1$. Noting also that $t \mapsto e^{-\frac{\alpha d_m^2}{2(r-1)t}}$ is increasing in $t>0$, we have 
\bna
\int_{\lambda T/r}^{T/r}e^{-\frac{\alpha d_m^2}{2(r-1)t}} dt \geq \frac{T(1-\lambda)}{r}e^{-\frac{r\alpha d_m^2}{2(r-1)\lambda T}}.
\ena
Combining the above three estimates yields
\bna
 \frac{T(1-\lambda)}{r}e^{-\frac{r\alpha d_m^2}{2(r-1)\lambda T}} \nor{u(T)}{L^2(\M)}^2 
 \leq C(\eta_0) r^{n\alpha}  e^{\gamma \frac{T}{r}}  \int_{\lambda T}^{T} \nor{u(s)}{L^2(\omega)}^2 ds , 
\ena
that is, for all $\eta>0$, $r>1$, $\lambda \in (0,1)$, and $\alpha>1$, 
\bna
 \nor{u(T)}{L^2(\M)}^2 
 \leq \frac{C(\eta) r^{n\alpha+1}}{T(1-\lambda)}  e^{\frac{2n\alpha K(r-1)}{\sqrt{2}(\alpha-1)} \frac{T}{r}} e^{\frac{r\alpha (d_\omega +\eta)^2}{2(r-1)\lambda T}} \int_{\lambda T}^{T} \nor{u(s)}{L^2(\omega)}^2 ds .
\ena

But $\frac{r}{r-1}=1+\frac{1}{r-1}$ can be made arbitrary close to $1^+$ for large $r$, $\lambda$ close to $1^-$, $\alpha$ close to $1^+$, and $\eta$ to $0^+$, so that $\frac{r\alpha (d_\omega +\eta)^2}{2(r-1)\lambda T} \leq \frac{d_\omega^2+\eps }{2T}$. We have thus proved the first statement.

To be a little more precise, we can choose $\alpha, r$ such that $\frac{1}{r}+\frac{1}{\alpha}=1$. This yields
\bna
 \nor{u(T)}{L^2(\M)}^2 
 \leq \frac{C(\eta) \left( \frac{\alpha}{\alpha-1} \right)^{n\alpha+1}}{T(1-\lambda)}  e^{\frac{2nK}{\sqrt{2}(\alpha-1)}T} e^{\frac{\alpha^2 (d_\omega +\eta)^2}{2 \lambda T}} \int_{\lambda T}^{T} \nor{u(s)}{L^2(\omega)}^2 ds ,
\ena
or, with $\alpha=1+\epsilon$ and $\lambda=1-\epsilon$, we obtain for all $\epsilon \in (0,1)$
\bna
 \nor{u(T)}{L^2(\M)}^2 
& \leq & \frac{C(\eta) \left(\frac{1+\epsilon}{\epsilon}\right)^{(1+\epsilon)n+1}}{T\epsilon}  e^{\frac{2nK}{\sqrt{2}\epsilon}T} e^{\frac{(1+\epsilon)^2}{1-\epsilon} \frac{(d_\omega +\eta)^2}{2 T}} \int_{(1-\epsilon) T}^{T} \nor{u(s)}{L^2(\omega)}^2 ds \\
& \leq & \frac{C(\eta)}{T\epsilon^{2n+2}}  e^{\frac{2nK}{\sqrt{2}\epsilon}T} e^{\frac{(1+\epsilon)^2}{1-\epsilon} \frac{(d_\omega +\eta)^2}{2 T}} \int_{(1-\epsilon) T}^{T} \nor{u(s)}{L^2(\omega)}^2 ds .
\ena
So we have proved the first estimate of the theorem. The second can be obtained similarly by integrating~\eqref{e:before-int-y} in the $x$ variable only, and not in the $y$ variable.
\enp

\begin{remark}
In fact, remark that from~\eqref{e:remplace-r-t} on, we could also put $\nor{u(T/r)}{L^2(\M)}^2$ on the left hand-side of all estimates of the proof, which amounts to $\nor{u(T\frac{\epsilon}{1+\epsilon})}{L^2(\M)}^2$, and, in particular, we have the much stronger statement
\bna
 \nor{u((1-\epsilon)T)}{L^2(\M)}^2 
 \leq \frac{C(\eta) r^{n\alpha+1}}{T\epsilon}  e^{\frac{2nK}{\sqrt{2}\epsilon}T} e^{\frac{(1+\epsilon)^2}{1-\epsilon} \frac{(d_\omega +\eta)^2}{2 T}} \int_{(1-\epsilon) T}^{T} \nor{u(s)}{L^2(\omega)}^2 ds .
\ena

\end{remark}

\begin{remark}
\label{rkexplicitpos}
All constants can be made explicit. We denote by $K := \min \left\{0 , - \min_{x\in \M} Ricci(x) \right\}$. For instance, we have for all $\eta>0$, all 
\bna
 \nor{u(T)}{L^2(\M)}^2 
 \leq \frac{C(\eta) r^{n\alpha+1}}{T(1-\lambda)}  e^{\frac{2nK}{\sqrt{2}(\alpha-1)}T} e^{\frac{\alpha^2 (d_\omega +\eta)^2}{2 \lambda T}} \int_{\lambda T}^{T} \nor{u(s)}{L^2(\omega)}^2 ds ,
\ena

Choosing the constants, we have, for all $\epsilon \in (0,1)$, for all $\eta>0$, 
\bna
\nor{u(T\frac{\epsilon}{1+\epsilon})}{L^2(\M)}^2
& \leq & \frac{C(\eta)}{T\epsilon^{2n+2}}  e^{\frac{2nK}{\sqrt{2}\epsilon}T} e^{(1+\epsilon)^3  \frac{(d_\omega +\eta)^2}{2 T}} \int_{(1-\epsilon) T}^{T} \nor{u(s)}{L^2(\omega)}^2 ds .
\ena

Remark that for non-negatively (Ricci) curved manifolds (this is the case of a convex domain in $\R^n$), then $K=0$ and the constant is 
$ \frac{C(\eta)}{T\epsilon^{2n+2}} e^{(1+\epsilon)^3  \frac{(d_\omega +\eta)^2}{2 T}}$ and hence decays like $1/T$ for $T$ large.
\end{remark}

\appendix

\section{Uniform Lipschitz Carleman estimates}
\label{s:carleman}

In this appendix, we produce Carleman estimates for a Laplace-Beltrami operator on a Riemannian manifold $M$ with boundary $\d M$. It requires the minimum of regularity and seems to be new from this point of view, even if it will not be surprising to specialists of the subject. Moreover, the proof below present several advantages with respect to the existing proof of similar results:
\begin{itemize}
\item it is relatively short;
\item it is completely geometric and, we hope, is relatively readable;
\item as we already said, it requires the minimum of regularity for the metric (in dimension $\geq3$), namely only Lipschitz regularity. Indeed, it is known that in dimension $\geq 3$, local uniqueness does not hold for general elliptic operators (even in divergence form) with $C^{0,\alpha}$ coefficients for all $\alpha <1$, see~\cite{Plis:63} and~\cite{Miller:74}. 

\end{itemize}
The proof, using formulae from Riemannian geometry, is inspired by some Carleman estimates for the Schr\"odinger equation proved by the first author \cite{L:10}.

There have been several works about such Carleman estimates for Lipschitz metric (but without boundary). The oldest result seems to be \cite{AKS:62} for elliptic operators. Another one, which actually falls short from the Lipschitz regularity is the very general result of H\"ormander \cite[Section 8.3]{Hoermander:63} which requires $C^1$ regularity, but applies to much more operators than elliptic ones. A proof for general elliptic operators with order $2m$ and Lipschitz coefficients is written by H\"ormander in~\cite[Proposition~17.2.3]{Hoermander:V3}.
For Lipschitz regularity of the coefficients, we can also mention for instance the recent preprint \cite{NRT:15}, with explicit dependence. Note that there has also been several research on doubling estimates directly on the parabolic equation, see \cite{CRV:02,EV:03} for instance.

\subsection{Toolbox of Riemannian geometry}
The definitions given in this section have a deep geometric meaning (see~\cite{GallotHulinLaf}). We will however only use the associated calculus rules, which we recall below. Note that they are usually written for smooth metrics, but they still make sense for Lipschitz metric, as we shall see below. We follow the notations of \cite{GallotHulinLaf}.

Here and in all estimates below, $M$ is a (not necessarily compact) smooth $d$-dimensional manifold with boundary $\d M$, so that $M = \d M \sqcup \Int(M)$. 

Given $U \subset M$ such that $\ovl{U}$ is compact in $M$ (note that this definition holds not only for open sets of $\Int(M)$), we denote by $L^p (U), H^k(U) , W^{k, \infty}(U)$ the usual Sobolev spaces. These are defined intrinsically once $U$ is fixed, even if the associated norms may depend on the metric or the charts chosen.The notation  $L^p_{\loc} (M), H^k_{\loc}(M) , W^{k, \infty}_{\loc}(M)$ will be used for functions belonging to $L^p (U)$ etc... for any  set $U$ such that $\ovl{U}$ is compact in $M$ (and not $\Int(M)$).

We denote by $g$ a Lipschitz metric on $M$, (that is, $x \mapsto g_x(\cdot , \cdot)$ is a Lipschitz section of the bundle of symmetric bilinear forms on $TM$ that is uniformly bounded from below by a positive constant). 

Given a local regularity space $B$ as above, and $U \subset M$ such that $\ovl{U}$ is compact in $M$, we define 
$$
\mathcal{T}^2_{B}(U) = \Gamma_{B}(T^2 T^*M)|_{U}
$$
to be the space of sections of $2-$tensors on $T^*M$ having regularity $B$ on a neighborhood of $U$. In local charts, such a tensor $t\in \mathcal{T}^2_{B}(M)$ writes $t = (t_{ij})$ with $t_{ij}$ having the regularity of $B$.
Typically, a locally Lipschitz metric $g$ satisfies $g \in \mathcal{T}^2_{W^{1,\infty}_{\loc}}(M)$.

We denote by $\gl{\cdot}{\cdot} = g(\cdot , \cdot )$ the inner product in $TM$. Remark that this notation omits to mention the point $x \in M$ at which the inner products takes place: this allows to write $\gl{X}{Y}$ as a function on $M$ (the dependence on $x$ is omitted here as well) when $X$ and $Y$ are two vector fields on $M$.
We also denote for a vector field $X$, $\gln{X}=\gl{X}{X}$.

We recall that the Riemannian gradient $\nablag$ of a function $f$ is defined by
$$
\gl{\nablag f}{X}= df(X) , \quad \text{ for any vector field }X ,
$$
For a function $f$ on $M$, we denote by $\int f = \int_M f(x) d\Vol_g(x)$ its integral on $M$, where $d\Vol_g(x)$ is the Riemannian density. 
We denote by $\div_g$ the associated divergence, defined on a vector field $X$ by
$$
\int  u \div_g X = -  \int \gl{\nablag u}{X} , \quad \text{for all } u  \in C^\infty_c(\Int (M)) .
$$
We denote by $\Lap= \div_g \nabla_g $ the associated (nonpositive) Laplace-Beltrami operator.
 We also denote by $D$ the Levi-Civita connection associated to the metric $g$ (see~\cite[Chapter~II Section~B]{GallotHulinLaf}).
 
 Let us now recall how these objects write in local coordinates.
 \begin{formule}
 \label{F1}
In coordinates, for $f$ a smooth function and $X=\sum_i X^i\frac{\partial }{\partial x_i}$, $Y=\sum_i Y^i\frac{\partial }{\partial x_i}$ smooth vector fields on $M$, we have
\begin{align*}
\gl{X}{Y}&=\sum_{i=1}^n g_{ij} X^iY^j , \\
 \nablag f&= \sum_{i,j=1}^n g^{ij}(\partial_j f)\frac{\partial }{\partial x_i} , \\
 \int f & = \int f(x) \sqrt{\det g(x)} dx , \\
  \div_g(X) &=\sum_{i=1}^n \frac{1}{\sqrt{\det g}}\partial_i \left(\sqrt{\det g} X_i\right) , \\
 \Lap f &=\sum_{i,j=1}^n \frac{1}{\sqrt{\det g}}\partial_i \left(\sqrt{\det g}g^{ij}\partial_j f\right) , \\
D_{X}Y&=\sum_{i=1}^n\left(\sum_{j=1}^nX^j\frac{\partial Y^i}{\partial x_j}+\sum_{j,k=1}^n\Gamma_{j,k}^iX^jY^k\right)\frac{\partial }{\partial x_i} ,
\end{align*}
where $(g^{-1})_{ij}=g^{ij}$ and the Chritoffel symbols are defined by $$\Gamma_{j,k}^i=\frac{1}{2}\sum_{l=1}^n g^{il}\left(\partial_j g_{kl}+\partial_k g_{lj}-\partial_lg_{jk}\right) ,$$ (see for instance \cite[p71]{GallotHulinLaf}).

Note in particular that the Lipschitz regularity of $g$ writes $g_{ij} \in W^{1,\infty} (M)$, and implies $g^{ij} \in W^{1,\infty} (M)$.
This entails, if $f, X, Y$ are smooth, that $\gl{X}{Y} \in W^{1,\infty} (M)$, $\nablag f$ is a Lipschitz vector field, $ \Lap f \in L^\infty(M)$ and $D_X Y$ is an $L^\infty$ vector field on $M$, since the definitions of $\Lap$ and $D_X$ involve one derivative of the coefficients of $g$.
\end{formule}

In view of the properties of $D_X$, it is natural to set $D_X f = Xf = df(X)$ for a function $f$ on $M$. 
Let us now collect some properties of these objects, that we shall use below.
\begin{formule}
For $f,g$ smooth functions and $X=\sum_i X^i\frac{\partial }{\partial x_i}$, $Y=\sum_i Y^i\frac{\partial }{\partial x_i}$ smooth vector fields on $M$, we have
\begin{align*}
\nablag (fh)& = (\nablag f)h+f(\nablag h) , \\
\div_g (f X)& = \gl{\nablag f}{X}+ f \div_g(X), \\
D_X(fY) &= (Xf) Y + f D_X Y ,  \quad \text{ where } Xf := df(X) \\
D_X (\gl{Y}{Z})&=\gl{D_X Y}{Z}+\gl{ Y}{D_X Z} .
\end{align*}
\end{formule}
That $D_X$ acts on functions as well as on vector fields suggests to extend the definition of $D_X$ to more general vector bundles (see~\cite[Proposition~2.58]{GallotHulinLaf}), and, in particular, for a one-form $\omega$, define (by duality) $D_X \omega$ to be the one-form acting as
$$
 (D_X\omega) ( Y) = X (\omega(Y)) - \omega (D_XY) , \quad \text{ for all vector fields } Y .
$$
This allows to define the Hessian of a function (see~\cite[Exercice~2.65]{GallotHulinLaf}) 
$$
Hess(f)(X,Y)=(D_X df)(Y), \quad \text{ for vector fields } X, Y ,
$$ (which only involves the values of $X$, $Y$ and not their derivatives). In local charts, note that we have 
\bna
Hess(f)(X,Y)=\sum_{i,j}X^iY^j\left[ \partial_{ij}^2f-\Gamma_{ij}^k\partial_k f \right] ,
\ena
which again is in $L^\infty(M)$ for a Lipschitz metric $g$ and $L^\infty$ vector fields $X,Y$. Note also that the Hessian of $f$ is  symmetric, that is $Hess(f)(X,Y)=Hess(f)(Y,X)$.
\begin{lemma}
For any function $f$ and any vector field $X$ and $Y$, we have
$$Hess(f)(X,Y) =\gl{ D_{X}\nablag f}{ Y} .$$
\end{lemma}
\begin{proof} 
According to the above calculus rules, we compute in two different ways the following quantity: 
 $$D_X (\gl{\nablag f}{ Y})=D_X \big(df(Y)\big)=(D_X df)(Y)+df(D_X Y)=Hess(f)(X,Y)+df(D_X Y).$$
We also have
 $$ D_X (\gl{\nablag f}{ Y})= \gl{D_X\nablag f}{ Y}+ \gl{\nablag f}{D_XY}=\gl{D_X\nablag f}{ Y}+df(D_X Y), $$
 which, combined with the previous computation yields the result.
\end{proof}

Finally, we recall an integration by parts formula in the present context.
\begin{formule}[Riemannian Stokes formula] Assume $\d M$ is piecewise $C^1$ and graph-Lipschitz.  Then, for all $f\in H^2_{\loc}(M)$ and $h \in H^1_{\loc}(M)$ one of which being compactly supported, we have 
$$
\int (\Lap f) h  = \int_{\d M} \gl{\nablag f}{\nu}h - \int \gl{ \nablag f}{\nablag h} .
$$
Here, the boundary $\d M$ is endowed with the Riemannian metric induced by $g$, and $\int_{\d M}$ is the integral with respect to the associated surface measure (defined as in Formula~\ref{F1}).
The vector field $\nu$ is the normal vector to $\d M$ which is outgoing. It is defined almost everywhere if $\d M$ is piecewise $C^1$.
In a local coordinate chart $(x_1, \cdots , x_n)$ centered at $0$, and in which $\d M \subset \{x_n=0\}$ and $M \subset \{x_n \leq 0\}$, we have $\nu = \sum_{j=1}^n \frac{g^{jn}}{\sqrt{g^{nn}}} \frac{\d}{\d x_j}$.
With the prescribed regularity of the boundary, the space $L^\infty_{\loc}(\d M)$ is defined intrinsically.
 We denote by $\d_\nu f =\gl{\nablag f}{\nu}$ the normal derivative at the boundary, which is only $L^\infty (\d M)$ since $\d M$ is piecewise $C^1$.

Note that in the above coordinate system, we have $\d_\nu f = \sum_{j=1}^n \frac{g^{jn}}{\sqrt{g^{nn}}} \d_{x_j} f$. In particular, if $f$ satisfies Dirichlet boundary conditions, this is $\d_\nu f  = \sqrt{g^{nn}} \d_{x_n} f$.

Note finally the vector field $X -\gl{X}{\nu} \nu$ is tangential to $\d \M$, so that we may decompose a vector field as its normal and tangential parts. In particular, we shall decompose the gradient $\nablag f = \d_\nu f \nu  +\nabla_T f$, where $\nabla_T f|_{\d \M} \in T\d\M$.
\end{formule}

\subsection{The Carleman estimate}

We stress the fact that functions $u\in C^{\infty}(M)$ are smooth up to the boundary of $M$ (as opposed to functions $u\in C^{\infty}(\Int(M))$). We will first estimate the Carleman conjugate operator in Theorem \ref{thmCarlemancalcul} and then give the desired estimate under appropriate assumptions in Theorem \ref{t:Carleman}.
\begin{theorem}
\label{thmCarlemancalcul}
Assume $g$ is a Lipschitz metric on $M$ and $\d M$ is piecewise $C^1$. 
Let $U$ be an open subset of $M$ such that $\ovl{U}$ is compact (in the topology of $M \supset \d M$) and denote $\Sigma = \d M \cap U$.
Then, for any $f \in W^{1,\infty}(U)$,  $\varphi \in W^{2,\infty}(U)$, $u\in H^2_{\comp}(U)$ and $\tau\geq 0$, we have
\begin{align*}
\int \left|e^{\tau \varphi}\Lap (e^{- \tau \varphi}u)\right|^2 
+ R(u) & \geq 
\tau^3\int \left[2 Hess(\varphi)(\nablag \varphi,\nablag \varphi) + (\Lap  \varphi )\gln{\nablag \varphi}- f \gln{\nablag \varphi} \right]|u|^2  \\
&\quad  + \tau \int 2Hess(\varphi)(\nablag u,\nablag u) - (\Lap \varphi )\gln{\nablag u}+ f \gln{ \nablag u} \\
&\quad  + BT(u) , 
\end{align*}
with boundary terms
\begin{align}
\label{e:def-BT}
BT(u)&= - 2\tau \int_{\Sigma}\gl{\nablag u}{\nu}\gl{\nablag\varphi}{\nablag u}+\tau \int_{\Sigma}\gl{\nablag \varphi}{\nu}\gln{\nablag u}\nonumber\\
&\quad -\tau^3\intS \gl{\nablag\varphi}{\nu}|u|^2\gln{\nablag \varphi}+\tau\intS \gl{\nablag u}{\nu} f u
\end{align}
and remainder $R(u)$ satisfying
\begin{align}
\label{e:def-R(u)}
|R(u)| & \leq \left( \nor{f-\Lap \varphi }{L^{\infty}(U)}^{2} + \frac{1}{2} \nor{\nablag f}{L^\infty(U)}\right)  \tau^2 \nor{u}{L^{2}}^{2} 
+  \frac{1}{2} \nor{\nablag f}{L^\infty(U)} \nor{\nablag u}{L^{2}}^{2}  . 
\end{align}
\end{theorem}
Note that the last term in~\eqref{e:def-BT} is actually of lower order. We keep it here since it vanishes in case of Dirichlet Boundary conditions.

\begin{remark}
It is very important for our purpose to notice that all terms in this identity only involve derivatives of order $0$ or $1$ of the metric. This will be important when we will consider stability issues with respect to Lipschitz perturbations of the metric.
\end{remark}

This identity suggests to introduce and study the following two important quantities, given $X$ a smooth vector field on $M$:
\bna
\B_{g, \varphi, f}(X)  = 2Hess(\varphi)(X,X) - (\Lap \varphi )\gln{X}+f \gln{X} ,\\ 
\E_{g, \varphi, f} = 2 Hess(\varphi)(\nablag \varphi,\nablag \varphi) + (\Lap  \varphi )\gln{\nablag \varphi}-f\gln{\nablag \varphi} . 
\ena
Note that for a Lipschitz metric $g$, we have $\E_{g, \varphi, f} \in L^\infty_{\loc}(M)$ and $\B_{g, \varphi, f}(X) \in  L^\infty_{\loc}(M)$ for any locally bounded vector field $X$.

\begin{remark}
\label{r:f-deltaphi-regularity}
At this level, it would be very tempting to set $F = - \Lap \varphi +f $ and work with the associated simplified expressions of $\B_{g, \varphi, f}(X)$ and $\E_{g, \varphi, f}$. From a conceptual point of view, this is completely fine, see Remark~\ref{r:conceptual-remark} below.
However, since we consider the limiting Lipschitz regularity of the metric, this change of additional function is not admissible. Indeed, the remainder term $R(u)$ in Theorem~\ref{thmCarlemancalcul} requires the regularity $\nablag f \in L^\infty$ and $f= F+ \Lap \varphi$ is already in $L^\infty$ and consumes one derivative of the metric $g$. Having $\nablag f \in L^\infty$ would then require $g$ to be $W^{2,\infty}$.
\end{remark}

We define $\nor{w}{L^2}^2 = \int |w|^2$ (see Formula~\ref{F1} for the notation $\int$) for a function $w$ and $\nor{X}{L^2}^2 = \int \gln{X}$ for a vector field $X$.

We can now state the Carleman estimate.
\begin{theorem}
\label{t:Carleman}
Let $U$ be an open subset of $M$ such that $\ovl{U}$ is compact (in the topology of $M \supset \d M$) and denote $\Sigma = \d M \cap U$.
Assume that the functions $(\varphi, f)$ satisfy: $f \in W^{1,\infty}(U)$,  $\varphi \in W^{2,\infty}(U)$, $\gln{\nablag \varphi} >0$ on $\ovl{U}$, and there exists $C_0>0$ such that for any vector field $X$, we have almost everywhere on $U$:
\begin{align}
\label{e:asspt-Bgeq}
\B_{g, \varphi, f}(X) &\geq 2 C_0 \gln{X}, \\
\label{e:asspt-Egeq}
\E_{g, \varphi, f}  & \geq 2 C_0  \gln{\nablag \varphi} .
\end{align}
Then, denoting $c(\varphi) = \min \left\{ 1, \left(\min_{\ovl{U}} \gln{\nablag \varphi} \right)^{-1}\right\}$, we have the following statements.
\begin{enumerate}
\item For all $\tau \geq \frac{c(\varphi)}{C_0} \left( \nor{f-\Lap \varphi }{L^{\infty}(U)}^{2} + \frac{1}{2} \nor{\nablag f}{L^\infty(U)}\right)$ and all $v \in C^\infty_c(U)$ we have the estimate
\begin{multline}
\label{e:thm-Carleman-brutal}
\frac{C_0}{3} \left( \tau^3 \nor{e^{\tau \varphi}v \nablag \varphi}{L^{2}(U)}^{2}+ \tau\nor{e^{\tau \varphi}\nabla_{g} v}{L^{2}(U)}^{2}\right) \\
\leq \nor{e^{\tau \varphi}\Lap v}{L^2(U)}^2 + \tau \left( \nor{ e^{\tau \varphi}  \nablag v}{L^2(\Sigma)}^2 +  \tau^2 \nor{e^{\tau \varphi} v \nablag \varphi }{L^2(\Sigma)}^2 \right) K_{f,\varphi} ,
\end{multline}
with $K_{f,\varphi} = 3\left( \frac{c(\varphi)}{\tau} \nor{f}{L^\infty(\Sigma)} + 3 \nor{\nablag \varphi}{L^\infty(\Sigma)} \right)$.

\item For all $\tau \geq \frac{c(\varphi)}{C_0} \left( \nor{f-\Lap \varphi }{L^{\infty}(U)}^{2} + \frac{1}{2} \nor{\nablag f}{L^\infty(U)}\right)$ and all $v \in C^\infty_c(U)$ such that $v=0$ on $\Sigma$, we have
\begin{align}
\label{e:thm-Carleman-case-dirichlet}
\frac{C_0}{3} \left( \tau^3 \nor{e^{\tau \varphi}v \nablag \varphi}{L^{2}(U)}^{2}+ \tau\nor{e^{\tau \varphi}\nabla_{g} v}{L^{2}(U)}^{2}\right)
\leq \nor{e^{\tau \varphi}\Lap v}{L^2(U)}^2 + \tau \int_{\Sigma} e^{2\tau \varphi}  \d_\nu \varphi  |\d_\nu v|^2 . 
\end{align}

\item If $\varphi|_\Sigma$ is constant and $- m(\varphi) :=\max_{\Sigma}\d_\nu\varphi < 0$, then setting  $M(\varphi) := \max_\Sigma(-\d_\nu\varphi)>0$, we have for all $\tau \geq  \max \left\{ \frac{c(\varphi)}{C_0} \left( \nor{f-\Lap \varphi }{L^{\infty}(U)}^{2} + \frac{1}{2} \nor{\nablag f}{L^\infty(U)}\right), \frac{\sqrt{\nor{f}{L^\infty(\Sigma)}}}{m(\varphi)} \right\}$ and all $v \in C^\infty_c(U)$, 
\begin{align}
\label{e:Carleman-Jerome}
\nor{e^{\tau \varphi}\Lap v}{L^2}^2 
+ M(\varphi) \tau \int_{\Sigma} e^{2\tau \varphi}\gln{\nabla_T v} 
 & \geq \frac{C_0}{3} \left( \tau^3 \nor{e^{\tau \varphi}v \nablag \varphi}{L^{2}(U)}^{2}+ \tau\nor{e^{\tau \varphi}\nabla_{g} v}{L^{2}(U)}^{2}\right)  \nonumber \\
 & \quad  +  \frac{\tau}{8} \frac{m(\varphi)^3}{M(\varphi)^2} \intS e^{2\tau \varphi}|\d_\nu v |^2 
    + \tau^3\frac{m(\varphi)^3}{4} \intS |v|^2 .
\end{align}
\end{enumerate}
\end{theorem}
\begin{remark}
\label{r:splitting-boundary}
In the last two statements of this result, we assume boundary conditions (either for $v$ or for $\varphi$) on the whole boundary $\Sigma$. Since the integrals involved are local, we could also assume different conditions on parts of the boundary, obtaining the associated terms in the estimates.
\end{remark}

For simplicity, in the proof, we shall denote by
$$
 \nor{u}{H^1_\tau}^2 = \tau^2 \nor{u\nablag \varphi}{L^2}^2 + \nor{\nablag u}{L^2}^2  
$$
the semiclassical norm (recall that $\gln{\nablag \varphi} >0$ here).

\bnp
We first let $v=e^{-\tau \varphi}u$, and apply the estimate of Theorem \ref{thmCarlemancalcul}. The latter, together with our assumption~\eqref{e:asspt-Bgeq}-\eqref{e:asspt-Egeq} (applied almost everywhere in $M$ to $X= \nablag u$) implies for all $\tau \geq 0$ and $u \in C^\infty_c(U)$
\begin{align*}
\nor{e^{\tau \varphi}\Lap (e^{- \tau \varphi}u)}{L^2}^2
+ R(u) & \geq 
2 C_0 \tau^3\nor{u \nablag \varphi}{L^2}^2 + 2 C_0 \tau \nor{\nablag u}{L^2}^2+ BT(u)  \\
&= 2C_0 \tau  \nor{u}{H^1_\tau}^2+ BT(u)  ,
\end{align*}
where $BT(u)$ is defined in~\eqref{e:def-BT} and $R(u)$ estimated in \eqref{e:def-R(u)}. Now, we have 
\begin{align*}
|R(u)| & \leq  c(\varphi) \left( \nor{f-\Lap \varphi }{L^{\infty}}^{2} + \frac{1}{2} \nor{\nablag f}{L^\infty}\right)  \nor{u}{H^1_\tau}^2,
\end{align*}
which implies that if $\tau C_0 \geq c(\varphi) \left( \nor{f-\Lap \varphi }{L^{\infty}}^{2} + \frac{1}{2} \nor{\nablag f}{L^\infty}\right)$, we obtain 
\begin{align}
\label{e:intermediate-with-BT}
\nor{e^{\tau \varphi}\Lap (e^{- \tau \varphi}u)}{L^2}^2 
  \geq 
   C_0 \tau  \nor{u}{H^1_\tau}^2+ BT(u)  .
\end{align}
We now consider the boundary terms. Without any assumption on the boundary, we have
\begin{align*}
|BT(u)|&\leq  3 \tau \nor{\nablag \varphi}{L^\infty(\Sigma)} \left( \nor{\nablag u}{L^2(\Sigma)}^2 + \tau^2\nor{u \nablag \varphi}{L^2(\Sigma)}^2 \right) \\
& \quad + \frac{1}{2} \nor{f}{L^\infty(\Sigma)} \left( \nor{\d_\nu u}{L^{2}(\Sigma)}^{2} + \tau^2 \nor{u}{L^{2}(\Sigma)}^{2} \right) , 
\end{align*}
and hence obtain in this case 
\begin{align*} 
  C_0 \tau  \nor{u}{H^1_\tau}^2  &\leq \nor{e^{\tau \varphi}\Lap (e^{- \tau \varphi}u)}{L^2}^2 \\
 &  + \left( c(\varphi) \nor{f}{L^\infty(\Sigma)} + 3 \tau \nor{\nablag \varphi}{L^\infty(\Sigma)} \right)
\left( \nor{\nablag u}{L^2(\Sigma)}^2 + \tau^2\nor{u \nablag \varphi}{L^2(\Sigma)}^2 \right) .
\end{align*}
Recalling that $u= e^{\tau \varphi}v$, this implies $\nablag u = e^{\tau \varphi} \nablag v + \tau u \nablag \varphi $, and hence
\begin{align}
\label{e:conj-phi-uv}
 \nor{e^{\tau \varphi} \nablag v}{L^2}^2 & \leq 2 \nor{\nablag u}{L^2}^2 + 2 \tau^2 \nor{u \nablag \varphi }{L^2}^2 = 2  \nor{u}{H^1_\tau}^2 , \\
 \nor{\nablag u}{L^2}^2 & \leq 2 \nor{ e^{\tau \varphi}  \nablag v}{L^2}^2 + 2 \tau^2 \nor{e^{\tau \varphi} v \nablag \varphi }{L^2}^2  \nonumber.
\end{align}

The last four estimates imply 
\begin{align*}
\frac{C_0}{3} \left( \tau^3 \nor{e^{\tau \varphi}v}{L^{2}}^{2}+ \tau\nor{e^{\tau \varphi}\nabla_{g} v}{L^{2}}^{2}\right)&  \leq C_0 \tau \nor{u}{H^1_\tau}^2  \\
& \leq \nor{e^{\tau \varphi}\Lap v}{L^2}^2 \\
 &  + 3\left( c(\varphi) \nor{f}{L^\infty(\Sigma)} + 3 \tau \nor{\nablag \varphi}{L^\infty(\Sigma)} \right) \\
& \times \left( \nor{ e^{\tau \varphi}  \nablag v}{L^2}^2 +  \tau^2 \nor{e^{\tau \varphi} v \nablag \varphi }{L^2}^2 \right),
\end{align*}
and hence~\eqref{e:thm-Carleman-brutal}.

Second, we assume the Dirichlet boundary condition $v|_{\Sigma}= 0$. This implies $u|_{\Sigma}= 0$ and $\nablag u|_{\Sigma} = \d_\nu u|_{\Sigma} \d_\nu$, so that we obtain
\begin{align*}
BT(u)&= - \tau \intS \d_\nu \varphi |\d_\nu u|^2 = - \tau \intS\d_\nu \varphi e^{2\tau \varphi}|\d_\nu v|^2 .
\end{align*}
Estimate~\eqref{e:intermediate-with-BT} then reads 
\begin{align*}
\nor{e^{\tau \varphi}\Lap (e^{- \tau \varphi}u)}{L^2}^2 +\intS \d_\nu \varphi e^{2\tau \varphi}|\d_\nu v|^2  \geq 
   C_0 \tau  \nor{u}{H^1_\tau}^2  .
\end{align*}
Using again~\eqref{e:conj-phi-uv} to come back to the variable $v$ yields~\eqref{e:thm-Carleman-case-dirichlet}.

Finally, we consider the case where $\varphi|_{\Sigma}$ is constant and $\d_\nu \varphi \leq -m(\varphi)<0$, in which case we obtain from~\eqref{e:def-BT}:
\begin{align*}
BT(u)&= - 2\tau \int_{\Sigma} \d_\nu \varphi |\d_\nu u|^2 +\tau \int_{\Sigma} \d_\nu \varphi \gln{\nablag u} 
 -\tau^3\intS (\d_\nu \varphi)^3  |u|^2 +\tau\intS \d_\nu u   f u \\
 &= - \tau \int_{\Sigma} \d_\nu \varphi |\d_\nu u|^2 + \tau \int_{\Sigma} \d_\nu \varphi \gln{\nabla_T u} 
 -\tau^3\intS (\d_\nu \varphi)^3  |u|^2 +\tau\intS \d_\nu u   f u .
\end{align*}
Estimate~\eqref{e:intermediate-with-BT} then reads 
\begin{align*}
\nor{e^{\tau \varphi}\Lap (e^{- \tau \varphi}u)}{L^2}^2 
+ \tau \int_{\Sigma} \d_\nu \varphi |\d_\nu u|^2 - \tau \int_{\Sigma} \d_\nu \varphi \gln{\nabla_T u} 
 + \tau^3\intS (\d_\nu \varphi)^3  |u|^2 - \tau\intS \d_\nu u   f u
  \geq 
   C_0 \tau  \nor{u}{H^1_\tau}^2  ,
\end{align*}
and hence, using $-M(\varphi)\leq \d_\nu \varphi \leq -m(\varphi)<0$,
\begin{align*}
\nor{e^{\tau \varphi}\Lap (e^{- \tau \varphi}u)}{L^2}^2 
+ M(\varphi) \tau \int_{\Sigma}\gln{\nabla_T u} 
- \tau\intS \d_\nu u   f u
  \geq 
   C_0 \tau  \nor{u}{H^1_\tau}^2  
   + m(\varphi) \tau \int_{\Sigma} |\d_\nu u|^2 
    + \tau^3m(\varphi)^3 \intS  |u|^2 .
\end{align*}
Now, we estimate
$$
\left| \intS \d_\nu u   f u \right| \leq \nor{f}{L^\infty(\Sigma)} \nor{u}{L^2(\Sigma)} \nor{\d_\nu u}{L^2(\Sigma)} 
\leq \frac{\nor{f}{L^\infty(\Sigma)} }{2 m(\varphi)\tau}\nor{\d_\nu u}{L^2(\Sigma)}^2
+ \frac{\nor{f}{L^\infty(\Sigma)} m(\varphi)\tau}{2} \nor{u}{L^2(\Sigma)}^2 ,
$$
so that, for $\tau \geq \frac{\sqrt{\nor{f}{L^\infty(\Sigma)}}}{m(\varphi)}$ this term is absorbed in the right handside, and we obtain
\begin{align*}
\nor{e^{\tau \varphi}\Lap (e^{- \tau \varphi}u)}{L^2}^2 
+ M(\varphi) \tau \int_{\Sigma}\gln{\nabla_T u} 
  \geq 
   C_0 \tau  \nor{u}{H^1_\tau}^2  
   + \frac{m(\varphi)}{2} \tau \int_{\Sigma} |\d_\nu u|^2 
    + \tau^3\frac{m(\varphi)^3}{2} \intS  |u|^2 .
\end{align*}
Recalling that $u= e^{\tau \varphi}v$ and $\nabla_T \varphi|_{\Sigma}=0$, this implies $\nabla_T u= e^{\tau \varphi} \nabla_Tv$ and $e^{2\tau \varphi}|\d_\nu v |^2 \leq 2 |\d_\nu u|^2 + 2\tau^2 |\d_\nu\varphi |^2 |u|^2$, hence
$$
\frac{1}{4} \frac{m(\varphi)^2}{M(\varphi)^2} \intS e^{2\tau \varphi}|\d_\nu v |^2 \leq \frac12 \int_{\Sigma} |\d_\nu u|^2  + \tau^2\frac{m(\varphi)^2}{2} \intS  |u|^2 .
$$
Finally using again~\eqref{e:conj-phi-uv} with the las two estimates implies~\eqref{e:Carleman-Jerome}, which concludes the proof of the theorem.

\enp

\bnp[Proof of Theorem \ref{thmCarlemancalcul}]
The statement of the theorem is a lower bound for the $L^2$ norm of the quantity $e^{\tau\varphi} \Lap (e^{-\tau\varphi}u)$, which we may compute as 
\bna
P_{\varphi} u := e^{\tau\varphi} \Lap (e^{-\tau\varphi}u) = \Lap u - 2\tau\gl{\nablag \varphi}{\nablag u}- \tau(\Lap \varphi) u+\tau^2\gln{\nablag \varphi}u .
\ena
We then decompose the conjugated operator $P_{\varphi}$ as 
$$
P_{\varphi}  = Q_2 + Q_1 
$$
with 
\bna
&&Q_1u:= - 2\tau\gl{\nablag \varphi }{\nablag u} - \tau f  u\\
&&Q_2u:=\Lap u+\tau^2\gln{\nablag \varphi}u - \tau(\Lap \varphi) u + \tau f u=\widetilde{Q_2}u+R_{2}u
\ena
where $\widetilde{Q_2}$ is the principal part of $Q_2$, that is
$$
\widetilde{Q_2}u=\Lap u+\tau^2\gln{\nablag \varphi}u , \quad \text{and} \quad  R_2 u =  \tau(- \Lap \varphi + f ) u .
$$
Now, we write ($\left\| \cdot \right\|$ denotes the $L^2$ norm for short)
\begin{align*}
2\left\|P_{\varphi}  u\right\|^2+2\left\|R_{2}u\right\|^2 & \geq \left\|P_{\varphi} u-R_{2}u\right\|^2  = \left\|Q_1u+\widetilde{Q_2}u\right\|^2, 
\end{align*}
where we estimate the remainder as 
\bnan
\label{e:estimR2}
\left\|R_{2}u\right\|^2\leq \tau^{2}\nor{f-\Lap \varphi }{L^{\infty}}^{2}  \nor{u}{L^{2}}^{2} .
\enan
Hence, we are left to produce a lower bound for
$$
\left\|Q_1u+\widetilde{Q_2}u\right\|^2=  \left\|Q_1u\right\|^2+\left\|\widetilde{Q_2}u\right\|^2+2\Re (Q_1u,\widetilde{Q_2}u).
$$
Now, remark that the all differential operators $P_{\varphi} , Q_1 , \widetilde{Q_2}$ involved have real coefficients. Hence, if we consider complex valued functions $u = u_R +i u_I$, we have $\|P_{\varphi} u\|^2= \|P_{\varphi} u_R \|^2 + \|P_{\varphi} u_I\|^2, \|u\|^2= \|u_R \|^2 + \|u_I\|^2$ so that proving $\|P_{\varphi} u\|^2 \geq c_0 \|u \|^2$ for real valued functions $u$ implies the same inequality for complex valued ones.
As a consequence, we only prove the result for a real valued function $u$, and associated real inner product. We now provide an explicit computation for $(Q_1u,\widetilde{Q_2}u)$, which is the key step in the proof.
\begin{lemma}
\label{l:computeM1M2}
For all functions $\varphi \in W^{2,\infty}_{\loc} (M)$, $f \in W^{1,\infty}_{\loc} (M)$ and $u \in H^2_{\comp} (M)$, we have  
\begin{align*}
(Q_1u,\widetilde{Q_2}u)
&= \tau^3\int \left(2 Hess(\varphi)(\nablag \varphi,\nablag \varphi) + (\Lap  \varphi )\gln{\nablag \varphi} - f\gln{\nablag \varphi} \right) |u|^2 \\
& \quad + \tau \int 2Hess(\varphi)(\nablag u,\nablag u) - (\Lap \varphi )\gln{\nablag u}+f\gln{ \nablag u} \\
& \quad + \tau \int u \gl{ \nablag u}{\nablag f}+ Boundary~terms , 
\end{align*}
with 
\begin{align*}
Boundary~terms&= - 2\tau \intS \gl{\nablag u}{ \nu}\gl{\nablag \varphi}{\nablag u} + \tau \intS \gl{\nablag \varphi}{ \nu}\gln{\nablag u}\\
& \quad - \tau^3\intS \gl{\nablag \varphi}{\nu}|u|^2\gln{\nablag \varphi}
- \tau \intS \gl{\nablag u}{\nu} f u .
\end{align*}
\end{lemma}

To conclude the proof of the theorem, we now simply write
\begin{equation}
\label{e:MRgeqM1M2}
2\left\|P_{\varphi} u\right\|^2+2\left\|R_{2}u\right\|^2 \geq \left\|Q_1u+\widetilde{Q_2}u\right\|^2 \geq 2 (Q_1u,\widetilde{Q_2}u) .
\end{equation}
In the estimates of Lemma~\ref{l:computeM1M2}, the remainder term is 
\begin{align*}
R_3 (u)  = - \tau \int u \gl{ \nablag u}{\nablag f}, \qquad 
|R_3 (u)|  \leq \frac{\nor{\nablag f}{L^{\infty}}}{2}\left(\nor{\nablag u}{L^{2}}^{2}+\tau^{2}\nor{u}{L^{2}}^{2}\right) ,
\end{align*}
which, combined with~\eqref{e:MRgeqM1M2}, \eqref{e:estimR2} and Lemma~\ref{l:computeM1M2}, concludes the proof of the Theorem with
\bna
R(u)=\left\|R_{2}u\right\|^2 + |R_3 (u)| .
\ena
\enp

\bnp[Proof of Lemma~\ref{l:computeM1M2}]
We have 
\begin{align}
\label{e:defM1M2JKL}
 (Q_1u,\widetilde{Q_2}u)=\int (- 2 \tau\gl{\nablag \varphi}{\nablag u}  -  \tau f u)(\Lap u +\tau^2\gln{\nablag \varphi}u)
 = - \tau( 2 J + 2 \tau^{2}K + L) , 
\end{align}
with
\begin{align*}
J = \int \gl{\nablag \varphi}{\nablag u}  \Lap u , \\ 
K = \int \gl{\nablag \varphi}{\nablag u}  \gln{\nablag \varphi}u ,  \\
L = \int f u (\Lap u +\tau^2\gln{\nablag \varphi}u) .
\end{align*}
We now perform one (and only one, which is the more we can do with the Lipschitz regularity of $g$) integration by parts in each of these integrals. Firstly, we compute $J$ as 
\bna 
J=\int_{\Sigma}\gl{\nablag u}{\nu}\gl{\nablag \varphi}{\nablag u}-\int \gl{\nablag u}{\nablag (\gl{\nablag \varphi}{\nablag u})} .
\ena
But, we also have
\begin{align*}
\gl{\nablag u}{\nablag (\gl{\nablag \varphi}{\nablag u})}& =D_{\nablag u}(\gl{\nablag \varphi}{\nablag u})\\
& =\gl{D_{\nablag u}\nablag \varphi}{\nablag u}+\gl{\nablag \varphi}{D_{\nablag u}\nablag u} \\
& = Hess(\varphi)(\nablag u,\nablag u) + Hess(u)(\nablag u,\nablag \varphi) ,
\end{align*}
so that
 \bna 
J = \int_{\Sigma}\gl{\nablag u}{\nu}\gl{\nablag \varphi}{\nablag u}-\int Hess(\varphi)(\nablag u,\nablag u)-\int Hess(u)(\nablag u,\nablag \varphi) .
\ena
But we notice that
\begin{align}
\label{e:comput-grad-norm-square}
\gl{\nablag \varphi }{\nablag \gln{\nablag u}}&=d(\gln{\nablag u})(\nablag \varphi ) =D_{\nablag \varphi}(\gl{\nablag u}{\nablag u}) \nonumber \\
& = \gl{D_{\nablag \varphi}\nablag u}{\nablag u}+\gl{\nablag u}{D_{\nablag \varphi}\nablag u}=2 \gl{D_{\nablag \varphi}\nablag u}{\nablag u} \nonumber \\
&=2Hess(u)(\nablag \varphi, \nablag u) ,
\end{align}
so that we have in particular
\bna
 2\int Hess(u)(\nablag \varphi, \nablag u)=\int \gl{\nablag \varphi }{\nablag \gln{\nablag u}}=-\int (\Lap \varphi )\gln{\nablag u}+ \int_{\Sigma}\gl{\nablag \varphi}{ \nu}\gln{\nablag u} .
\ena
Coming back to $J$ this finally implies the expression
\begin{align}
\label{e:comput-J}
2 J & =2 \int_{\Sigma}\gl{\nablag u}{ \nu}\gl{\nablag \varphi}{\nablag u}-2 \int Hess(\varphi)(\nablag u,\nablag u) \nonumber  \\
& \quad +\int (\Lap \varphi )\gln{\nablag u}- \int_{\Sigma}\gl{\nablag \varphi}{\nu}\gln{\nablag u} .
\end{align}
Secondly, remarking that $\nablag |u|^2=2u \nablag u$, we write $K$ as 
\bna
K = \int \gl{\nablag \varphi}{\nablag u} \gln{\nablag \varphi}u=\frac{1}{2}\int \gln{\nablag \varphi} \gl{\nablag \varphi}{ \nablag |u|^2} .
\ena
An integration by parts yields
 \begin{align*}
\int (\Lap \varphi)|u|^2\gln{\nablag \varphi}& =\intS  \gl{\nablag \varphi }{\nu}|u|^2\gln{\nablag \varphi} - \int \gl{\nablag \varphi }{\nablag \left(|u|^2\gln{\nablag \varphi}\right)} \\
& =\intS  \gl{\nablag \varphi }{\nu}|u|^2\gln{\nablag \varphi} \\
& \quad  - \int \gln{\nablag \varphi} \gl{\nablag \varphi }{ \nablag |u|^2} -\int  |u|^2 \gl{\nablag \varphi }{\nablag \gln{\nablag \varphi}} .
\end{align*}
Combining these two formulas, we obtain
\begin{align}
\label{e:comput-K}
2K&=-\int (\Lap \varphi)|u|^2\gln{\nablag \varphi}+\intS \gl{\nablag \varphi}{\nu}|u|^2\gln{\nablag \varphi} - \int \gl{\nablag \varphi }{\nablag \gln{\nablag \varphi}}|u|^2 \nonumber\\
&=-\int (\Lap \varphi)|u|^2\gln{\nablag \varphi}+\intS \gl{\nablag \varphi}{\nu}|u|^2\gln{\nablag \varphi} -2 \int Hess(\varphi)(\nablag \varphi,\nablag\varphi)|u|^2 ,
\end{align}
where we have used as in~\eqref{e:comput-grad-norm-square} that 
$$ \gl{\nablag \varphi}{\nablag \gln{\nablag \varphi}}=D_{\nablag \varphi}\gl{\nablag \varphi}{\nablag \varphi }=2 \gl{D_{\nablag \varphi}\nablag \varphi}{\nablag \varphi} =2 Hess(\varphi)(\nablag \varphi,\nablag \varphi) .$$
Thirdly, let us compute $L$ with one integration by parts as 
\begin{align}
\label{e:comput-L}
L&=\int  f u(\Lap u +\tau^2\gln{\nablag \varphi}u) \nonumber \\
& =  \int_{\Sigma} \gl{\nablag u}{\nu}f u  - \int \gl{ \nablag u}{\nablag (f u)} + \tau^2\int \gln{\nablag \varphi}f |u|^2 \nonumber \\
&=  \int_{\Sigma} \gl{\nablag u}{\nu}f u  - \int f \gln{ \nablag u}  + \tau^2\int \gln{\nablag \varphi}f |u|^2  - \int u \gl{ \nablag u}{\nablag f} .
\end{align}
Coming back to~\eqref{e:defM1M2JKL} and combining the computations of $J, K, L$ in~\eqref{e:comput-J}-\eqref{e:comput-K}-\eqref{e:comput-L}, we have obtained the statement of Lemma~\ref{l:computeM1M2}.
\enp

\begin{remark}
\label{r:conceptual-remark}
We wish to compare the above proof with the more usual proofs of Carleman estimates. Note first that the fact that operators and functions are real-valued implies, for $u \in C^\infty_c(\Int(M))$ that $(Q_1 u, Q_2 u) =  (Q_2 Q_1 u, u) = -(u, Q_1 Q_2 u) = \frac12 \left( [Q_2,Q_1]u , u\right)$. Note also that the principal symbol of the conjugate operator $P_\varphi$ is given by 
$$
p_{\varphi}(x,\xi)  = \sigma(P_\varphi)(x,\xi) = \gln{\xi^\sharp} - \tau^2 \gln{\nablag \varphi} + 2 i\tau \gl{ \nablag \varphi}{\xi^\sharp}
= |\xi|_{g^*}^2 - \tau^2 |d\varphi |_{g^*}^2 + 2i \tau \langle d\varphi , \xi \rangle_{g^*} ,
$$
where $g^*$ is the dual metric on $T^*M$, i.e. $g^* = (g^{ij})$, and $\xi^\sharp$ is defined by $\gl{\xi^\sharp}{X}=\xi(X)$.

Here, a computation shows that we have
$$
\{ \Re p_{\varphi},\Im p_{\varphi}\}(x, \xi) = 4 \tau Hess(\varphi)(\xi^\sharp,\xi^\sharp) + 4 \tau^3Hess(\varphi)(\nablag \varphi,\nablag \varphi) .
$$
As a consequence, the important quantity in the Carleman estimate of Theorem~\ref{t:Carleman} is 
\begin{align*}
\B_{g, \varphi, f}(\xi^\sharp) + \tau^2 \E_{g, \varphi, f} & = (f - \Lap \varphi ) \left(\gln{\xi^\sharp} - \tau^2 \gln{\nablag \varphi}\right)\\
& \quad + 2Hess(\varphi)(\xi^\sharp,\xi^\sharp) +  2 \tau^2 Hess(\varphi)(\nablag \varphi,\nablag \varphi)  \\
& = (f - \Lap \varphi ) \Re p_{\varphi} +\frac{1}{2 \tau}\{ \Re p_{\varphi},\Im p_{\varphi}\} .
\end{align*}
The main assumption under which the Carleman estimate of Theorem~\ref{t:Carleman} holds is hence the existence of a function $F=F(x)$ (of the position variable only) so that 
\bnan
\label{e:symbolic-assumption}
F \Re p_{\varphi} +\frac{1}{2\tau}\{ \Re p_{\varphi},\Im p_{\varphi}\}\geq C ( |\xi |^{2}+\tau^{2}).
\enan
The choice of $F$ under the form $F= f - \Delta \varphi$ is only made in order not to consume regularity of the metric $g$, see above Remark~\ref{r:f-deltaphi-regularity}.

Of course, Assumption~\eqref{e:symbolic-assumption} is stronger than the usual subellipticity of the H\"ormander theorem~\cite[Section~23]{Hoermander:V4}:
$$
\{ \Re p_{\varphi},\Im p_{\varphi}\}>0 \quad \text{ on the set } \quad \left\{ \Re p_{\varphi}  = 0,  \Im p_{\varphi} = 0\right\} .
$$
The proof of the H\"ormander theorem~\cite[Section~23.3]{Hoermander:V4} then uses a symbol $F(x, \xi)$ instead of just a function $F(x)$, for instance having the form $F(x,\xi)=\frac{\Re p_{\varphi}}{\xi^{2}+\tau^{2}}$. 

We refer to~\cite[Section~3.1]{LeLe:09} for a related discussion regarding the Furskiov-Imanuvilov approach to Carleman estimates.
\end{remark}

\subsection{Constructing weight functions via convexification}
In this section, we explain how to construct weight functions $(\varphi, f)$ that satisfy the Assumption of Theorem~\ref{t:Carleman}, via the usual convexification procedure. In the present context (as opposed to the usual situation), this also requires a smart choice of the function $f$.
\begin{lemma}[Explicit convexification]
\label{l:explicit-convexification-G}
Let $\Psi \in W^{2,\infty}_{\loc}(M;\R)$ and $G \in W^{2,\infty}_{\loc}(\R)$, and choose 
\bnan
\label{e:choice-phi-f}
\varphi = G(\Psi)\quad \text{ and } \quad  f = 2 G''(\Psi) \gln{\nablag \Psi}.
\enan
Then we have 
\begin{align*}
\B_{g, \varphi, f}(X) & = 2 G'(\Psi) Hess(\Psi)(X,X) + 2G''(\Psi) \left|\gl{\nablag \Psi }{ X}\right|^2
+ \left(G''(\Psi) \gln{\nablag \Psi} -G'(\Psi) \Lap \Psi \right) \gln{X},\\
\E_{g, \varphi, f} & = G'(\Psi)^2\Big[2 G'(\Psi) Hess(\Psi)(\nablag\Psi,\nablag\Psi)+ G''(\Psi)\left| \nablag \Psi\right|_g^4 
+ G'(\Psi) \Lap \Psi \gln{\nablag \Psi}\Big].
\end{align*}
\end{lemma}
To state the next corollary, for $B \in \mathcal{T}^2_{L^\infty_{\loc}}(M)$ a $L^\infty_{\loc}$ section of bilinear forms on $TM$, we define $|B|_g (x) = \sup_{X \in T_x M\setminus 0}\frac{|B(x,X,X)|}{\gln{X}}$ which yields a $L^\infty_{\loc}$ function on $M$. 
\begin{corollary}
\label{c:explicit-convexification-exp}
Let $\Psi \in W^{2,\infty}_{\loc}(M;\R)$, $\lambda>0$ and define $\varphi, f$ as in~\eqref{e:choice-phi-f} with $G(t)=e^{\lambda t}$. Then, for any $\lambda >0$ and any vector field $X$, we have almost everywhere on $M$
\begin{align*}
\B_{g, \varphi, f}(X) &\geq
\lambda e^{\lambda \Psi} \gln{X}  \left( \lambda \gln{\nablag \Psi} - 2 |Hess(\Psi)|_g - \Lap \Psi \right) ,\\
\E_{g, \varphi, f}  & \geq
\lambda e^{\lambda\Psi} \gln{\nablag \varphi} \left(  \lambda \left| \nablag \Psi\right|_g^2 - 2 | Hess(\Psi)|_g +  \Lap \Psi \right) .
\end{align*}
\end{corollary}

\bnp[Proof of Corollary~\ref{c:explicit-convexification-exp}]
With this choice of $G$, Lemma~\ref{l:explicit-convexification-G} gives
$$
\B_{g, \varphi, f}(X)  = \lambda e^{\lambda \Psi}  \left[ 2 Hess(\Psi)(X,X) + 2\lambda  \left|\gl{\nablag \Psi }{ X}\right|^2
- \Lap \Psi \gln{X} + \lambda \gln{\nablag \Psi}  \gln{X} \right] ,
$$
together with 
$$
\E_{g, \varphi, f}  =\lambda^3 e^{3\lambda\Psi} \Big[2 Hess(\Psi)(\nablag\Psi,\nablag\Psi)+ \lambda \left| \nablag \Psi\right|_g^4 
+  \Lap \Psi \gln{\nablag \Psi}\Big] ,
$$
which yields the sought result.
\enp

\bnp[Proof of Lemma~\ref{l:explicit-convexification-G}]
We first have $d \varphi = G'(\Psi) d \Psi$ and $\nablag \varphi = G'(\Psi) \nablag\Psi.$
We then compute the Hessian and the Laplacian as 
\begin{align*}
Hess(\varphi)(X,Y)&=\gl{D_X \nablag \varphi }{ Y}= \gl{D_X (G'(\Psi) \nablag \Psi) }{Y} \\
&= G'(\Psi) \gl{D_X \nablag \Psi }{ Y} + G''(\Psi)d\Psi(X)\gl{\nablag \Psi}{ Y}\\
&= G'(\Psi) Hess(\Psi)(X,Y)+ G''(\Psi) \gl{\nablag \Psi }{ X} \gl{\nablag \Psi}{Y} ,
\end{align*}
and 
\bna
\Lap \varphi= \div_g ( G'(\Psi) \nablag\Psi )= G'(\Psi) \Lap \Psi+ G''(\Psi) \gln{\nablag \Psi} .
\ena
In particular, we have 
\begin{align*}
Hess(\varphi)(\nablag \varphi,\nablag \varphi)&= G'(\Psi) Hess(\Psi)(\nablag\varphi,\nablag\varphi)+G''(\Psi) \left|\gl{\nablag \Psi }{\nablag \varphi}\right|^2\\
&= G'(\Psi)^2  \left[ G'(\Psi)Hess(\Psi)(\nablag\Psi,\nablag\Psi)+G''(\Psi)\left|\gln{\nablag \Psi}\right|^2\right] ,
\end{align*}
together with
$$
\Lap \varphi \gln{\nablag \varphi}=  G'(\Psi)^2 \gln{\nablag \Psi} \left(  G'(\Psi) \Lap \Psi+  G''(\Psi)\gln{\nablag \Psi} \right)
$$
As a consequence, we obtain
\begin{align*}
\B_{g, \varphi, f}(X) 
& = 2 G'(\Psi) Hess(\Psi)(X,X) + 2G''(\Psi) \left|\gl{\nablag \Psi }{ X}\right|^2  \\
& \quad + \left( -G'(\Psi) \Lap \Psi - G''(\Psi) \gln{\nablag \Psi} + f \right) \gln{X},
\end{align*}
as well as
\begin{align*}
\E_{g, \varphi, f}(x) & =
G'(\Psi)^2\Big[2 G'(\Psi) Hess(\Psi)(\nablag\Psi,\nablag\Psi)+2 G''(\Psi)\left| \nablag \Psi\right|_g^4 \\
& \quad +\left(G'(\Psi) \Lap \Psi+G''(\Psi) \gln{\nablag \Psi} - f \right)\gln{\nablag \Psi}\Big].
\end{align*}
Now, recalling the choice $f=2 G''(\Psi) \gln{\nablag \Psi}$ concludes the proof of the lemma.
\enp

\begin{remark}
Note that in this proof, the choice $f =\alpha G''(\Psi) \gln{\nablag \Psi}$ yields a useful lower bound only if $\alpha \in (1,3)$. See also~\cite[Section~3.1]{LeLe:09} for a related discussion. 
\end{remark}

\subsection{Uniformity with respect to the metric}
Until this point, all calculations are exact for a fixed metric. In the present section, we prove uniform estimates in a class of metrics. 
For this, even though the manifold with boundary $M$ is not assumed compact, we will consider only an open subsets $U$ of $M$ such that $\ovl{U}$ is compact in $M$ (not in $\Int(M)$).
On the compact set $\K$, the spaces $W^{k,\infty}(\K)$ are defined intrinsically, even if the associated norms may depend on the metric or the charts chosen.
We fix one of these norms $\| \cdot \|_{W^{1,\infty} (\K)}$ for functions on $M$, as well as for forms on $M$ (still denoted $\| \cdot \|_{W^{1,\infty} (\K)}$).

Now, given a reference metric $g_0$ and two constants $D \geq \eps >0$, we consider the class 
$$
\Gamma_{\eps, D}(\K, g_0) = \left\{ g \text{ metric in } \mathcal{T}^2_{W^{1,\infty}_{\loc}}(M), \quad \|g\|_{W^{1,\infty} (\K)} \leq D, \quad \eps g_0 \leq g \leq D g_0\right\} .
$$
\begin{lemma}
\label{l:uniform-subellipticity}
Let $U$ be an open subset of $M$ such that $\ovl{U}$ is compact (in the topology of $M \supset \d M$) and denote $\Sigma = \d M \cap U$.
Given a metric $g_0 \in \mathcal{T}^2_{W^{1,\infty}_{\loc}}(M)$, $D \geq \eps >0$, and a function $\Psi \in W^{2, \infty} (U)$ such that $|\nabla_{g_0} \Psi |_{g_0}^2 >0$ on $\ovl{U}$, there exists $C_0>0$ and $\lambda>0$ such that for any $g \in \Gamma_{\eps, D}(\ovl{U}, g_0)$, the functions $\varphi=e^{\lambda \Psi}$, $f=2\lambda^{2} \gln{\nablag \Psi}$ satisfy
\begin{align}
\B_{g, \varphi, f}(X) &\geq 2 C_0 \gln{X}, \text{ for all vector fields $X$}, \label{e:B-unif-bound-below} \\
\E_{g, \varphi, f}  & \geq 2 C_0 \gln{\nablag\varphi} ,  \label{e:E-unif-bound-below} 
\end{align}
almost everywhere in $U$.
\end{lemma}
Note that the constant $C_0$ involved is explicitely computable in terms of $D$ and $\eps$, which we do not write for the sake of readability. Yet, if one is interested in obtaining explicit constants, the choice $G(t) = e^{\lambda t}$ of convexifying function is probably not the best one.
\bnp
Denote by $g^* = (g^{ij})$ the metric on $T^*M$ induced by $g$. For $g \in \Gamma_{\eps, D}(\ovl{U}, g_0)$, we have $\frac{1}{\eps} g_0^* \leq g^* \leq \frac{1}{D} g_0^*$. 
With this notation, we have 
\begin{equation}
\label{e:equiv-norms-g}
\frac{1}{\eps} |\nabla_{g_0} \Psi|_{g_0}^2=\frac{1}{\eps} |d \Psi |_{g^*_0}^2
 \leq \gln{\nablag \Psi} = |d \Psi |_{g^*}^2 
 \leq \frac{1}{D} |d \Psi |_{g^*_0}^2  =  \frac{1}{D} |\nabla_{g_0} \Psi|_{g_0}^2 ,
\end{equation}
where $|\omega |_{g^*}^2 =\langle \omega , \omega \rangle_{g^*}$ is the cotangent squared norm.
Next, using the uniform $W^{1,\infty}(U)$ bound in $\Gamma_{\eps, D}(\ovl{U}, g_0) $, we have
$$
|\Lap \Psi| \leq C(\eps, D) \|\Psi\|_{W^{2,\infty}(U)}, \quad | Hess(\Psi)|_g  \leq C(\eps, D) \|\Psi\|_{W^{2,\infty}(U)} .
$$
Now, the compactness of $\ovl{U}$ with the assumption yields $c_0>0$ such that $|\nabla_{g_0} \Psi |_{g_0}^2 \geq c_0$ everywhere on $\ovl{U}$.
According to Corollary~\ref{c:explicit-convexification-exp} and the above two estimates, we obtain for any $\lambda >0$ and any vector field $X$ 
\begin{align*}
\B_{g, \varphi, f}(X) &\geq
\lambda e^{\lambda \Psi} \gln{X}  \left( \lambda \gln{\nablag \Psi} - 2 |Hess(\Psi)|_g - \Lap \Psi \right) ,\\
&\geq
\lambda e^{\lambda \min_\K \Psi} \gln{X}  \left( \lambda \frac{c_0}{\eps} - 3 C(\eps, D) \|\Psi\|_{W^{2,\infty}(U)}\right) ,
\end{align*}
which yields~\eqref{e:B-unif-bound-below} when taking $\lambda$ large enough. Similarly, \eqref{e:E-unif-bound-below} follows from taking $\lambda$ large enough in
\begin{align*}
\E_{g, \varphi, f}  & \geq
\lambda e^{\lambda \min_\K \Psi }  \gln{\nablag\varphi}  \frac{c_0}{\eps}   \left( \lambda \frac{c_0}{\eps} - 3 C(\eps, D) \|\Psi\|_{W^{2,\infty}(U)}\right) .
\end{align*}
\enp

We directly deduce the following uniform Carleman estimate in the class $\Gamma_{\eps, D}(\ovl{U}, g_0)$. We only state it with the Dirichlet boundary condition here for conciseness (the case without boundary condition writes the same).

\begin{theorem}[Uniform Lipschitz Carleman estimate]
\label{t:Carlemanuniform}
Let $U$ be an open subset of $M$ such that $\ovl{U}$ is compact (in the topology of $M \supset \d M$) and denote $\Sigma = \d M \cap U$.
Given a metric $g_0 \in \mathcal{T}^2_{W^{1,\infty}_{\loc}}(M)$, $D \geq \eps >0$, and a function $\Psi \in W^{2, \infty} (U)$ such that $|\nabla_{g_0} \Psi |_{g_0}^2 >0$ on $\ovl{U}$, there exist $\lambda>0 ,C_1>0, \tau_0 >0$ such that for $\varphi=e^{\lambda \Psi}$ and for any $g \in \Gamma_{\eps, D}(\ovl{U}, g_0)$, 
for all $\tau \geq \tau_0$ and all $v \in C^\infty_c(U)$ such that $v=0$ on $\Sigma$, we have
\begin{align}
\label{e:thm-Carleman-case-dirichlet-bis}
C_1\left( \tau^3 \nor{e^{\tau \varphi}v }{L^{2}(U)}^{2}+ \tau\nor{e^{\tau \varphi}\nabla_{g} v}{L^{2}(U)}^{2}\right)
\leq \nor{e^{\tau \varphi}\Lap v}{L^2(U)}^2 + \tau \int_{\Sigma} e^{2\tau \varphi}  \d_\nu \varphi  |\d_\nu v|^2 ,
\end{align}
\begin{align}
\label{e:thm-Carleman-case-dirichlet-ter}
C_1\left( \tau^3 \nor{e^{\tau \varphi}v }{L^{2}(U)}^{2}+ \tau\nor{e^{\tau \varphi}\nabla_{g_0} v}{L^{2}(U)}^{2}\right)
\leq \nor{e^{\tau \varphi}\Lap v}{L^2(U)}^2 + \tau \int_{\Sigma} e^{2\tau \varphi}  \d_\nu \varphi  |\d_\nu v|^2 . 
\end{align}
\end{theorem}
Note that in the second inequality~\eqref{e:thm-Carleman-case-dirichlet-ter}, we implicitely wrote $$\nor{e^{\tau \varphi}\nabla_{g_0} v}{L^{2}(U)}^{2} =\int_U e^{2\tau \varphi}\left|\nabla_{g_0} v \right|_{g_0}^2 d\Vol_{g_0}$$ in the left handside, which does no longer depend on the metric $g$. Hence, the sole dependence on the metric $g$ in~\eqref{e:thm-Carleman-case-dirichlet-ter} is through $\Delta_g$ and $\d_\nu$.

\bnp
We choose $f=2\lambda^{2} \gln{\nablag \Psi}$ and according to Lemma~\ref{l:uniform-subellipticity}, the bounds~\eqref{e:asspt-Bgeq}-\eqref{e:asspt-Egeq} with constant $C_0$ are satisfied for $\lambda$ large enough uniformly in the class $g \in \Gamma_{\eps, D}(\ovl{U}, g_0)$.
According to Theorem~\ref{t:Carleman}, this implies~\eqref{e:thm-Carleman-case-dirichlet-bis} with $C_1 = \frac{C_0}{3}c(\varphi)$ for all 
 $\tau \geq \tau_0(g)$ with $\tau_0(g) = \frac{c(\varphi)}{C_0} \left( \nor{f-\Lap \varphi }{L^{\infty}(U)}^{2} + \frac{1}{2} \nor{\nablag f}{L^\infty(U)}\right)$ with $c(\varphi) = \min \left\{ 1, \left(\min_{\ovl{U}} \gln{\nablag \varphi} \right)^{-1}\right\}$.
Now,~\eqref{e:equiv-norms-g} implies that 
$$ \min \left\{ 1, \eps \left(\min_{\ovl{U}}|\nabla_{g_0} \Psi|_{g_0}^2\right)^{-1}\right\} \leq 
c(\varphi)\leq \min \left\{ 1, D \left(\min_{\ovl{U}}|\nabla_{g_0} \Psi|_{g_0}^2\right)^{-1}\right\}
$$ uniformly for $g \in \Gamma_{\eps, D}(\ovl{U}, g_0)$, and, similarly
$$
\tau_0(g) \leq C(\eps, D) \|\Psi\|_{W^{2,\infty}(U)} ,
$$
uniformly for $g \in \Gamma_{\eps, D}(\ovl{U}, g_0)$. This concludes the proof of~\eqref{e:thm-Carleman-case-dirichlet-bis}. The proof of~\eqref{e:thm-Carleman-case-dirichlet-ter} follows again from~\eqref{e:equiv-norms-g} (applied to $v$) and the fact that $d\Vol_{g_0} \leq \eps^{-d/2} d\Vol_{g}$ (recall that $d=\dim M$).
\enp

Note that for the application that we have in Proposition~\ref{p:unif-interp-boundary-metric} below, it is sufficient to have some stability results in the following sense. If some interpolation inequality or Carleman inequality is true for some metric, it is still true for any metric in a suitable neighborhood, which can be obtained as a byproduct of our results.

\subsection{Uniform interpolation estimate at the boundary}

In this section, we consider a very particular case of the above Carleman estimate to prove a local interpolation estimate in a neighborhood of a boundary point for metrics $g$ in the neighborhood of the constant flat metric.
The manifold $M$ considered is $\R^{n+1}_+ = \R^{n} \times \R_+$ (that is, $d=n+1$) and the reference metric is $g_0= \id$

 Note that the above sections prove much more than needed for this argument.
 
Below, we denote $B_r = B(0,r) \subset \R^{n+1}$ and $B_r^+ = B(0,r)\cap \R^{n+1}_+$.

\begin{proposition}
\label{p:unif-interp-boundary-metric}
There exists $\e>0$, $C>0$ and $\alpha_0\in (0,1)$ so that for any metric $\mathfrak{g} \in \Gamma_{\eps, D}(\ovl{B}_{\R^n}(0,2), \id)$, we have 
\begin{equation*}
\|v\|_{H^1(B_1^+)} \leq C \left(\|(-\d_s^2-\Delta_\mathfrak{g} ) v\|_{L^2(B_2^+)} +  \| \partial_s v|_{s=0} \|_{L^2(B_2\cap \{0\}\times \R^{n})}  \right)^{\alpha_0} \|v\|_{H^1(B_2^+)}^{1-\alpha_0}, 
\end{equation*}
for any $v\in H^2(B_2^+)$ such that $v|_{s=0}=0$.
\end{proposition}

\begin{proof}
In the proof, we shall denote by $\x = (s, x) \in \R_+ \times \R^n$ the overall variable and recall that all balls are centered at zero.
We choose a point  $\x^a = (-a , 0, \cdots , 0 ) \notin \overline{\R^{n+1}_+}$. 
We define the weight function $\Psi (\x)=  - |\x-\x^a|$, which is smooth and satisfies $\Psi <0$ and $d\Psi \neq 0$ in $\ovl{B_2^+}$.

For $a$ sufficiently small, there exist $0<\rho_1 < \rho_2 $ such that we have
\bnan
\label{e:defW1W2}
\ovl{B_1^+} \subset W_1\subset \ovl{W_1}  \subset W_2 \subset \ovl{W_2} \subset B_2^+, \quad \text{with} \quad   W_j = \{\psi > -\rho_j\} \cap \overline{\R^{n+1}_+}, \quad j=1,2.
\enan

As a consequence of Theorem~\ref{t:Carlemanuniform}, there exist $\lambda>0 ,C_1>0, \tau_0 >0$ such that for $\varphi=e^{\lambda \Psi}$ and for any $g =\id \otimes \mathfrak{g} \in \Gamma_{\eps, D}(\ovl{B_2^+}, \id)$, for all $\tau \geq \tau_0$ and all $u \in C^\infty_c(B_2^+)$ such that $u=0$ on $\{s=0\}$, we have
\begin{align}
\label{e:thm-Carleman-case-dirichlet-interp}
C_1\left( \tau^3 \nor{e^{\tau \varphi} u}{L^{2}(B_2^+)}^{2}+ \tau\nor{e^{\tau \varphi}\nabla u}{L^{2}(B_2^+)}^{2}\right)
\leq \nor{e^{\tau \varphi}\Lap u}{L^2(B_2^+)}^2 + \tau \int_{\{s=0\}} e^{2\tau \varphi}  \d_\nu \varphi  |\d_\nu u|^2 . 
\end{align}
Here, the ball, the gradient and the volume element are taken w.r.t. the Euclidean metric. Moreover, the normal vector-field $\d_\nu$ is that associated to the metric $g =\id \otimes \mathfrak{g}$, and hence $\d_\nu = -\d_s$ (and does not depend on $\mathfrak{g}$). The sole dependence on the metric in~\eqref{e:thm-Carleman-case-dirichlet-interp} is thus in $\Lap = \d_s^2 +\Delta_\mathfrak{g}$.

Note that levelsets of $\varphi$ are those of $\psi$ i.e. pieces of spheres.
Note also that we have $\varphi \leq \varphi(0)$ on $B_2^+$ and define
$$
\varphi(0) > \varphi_1 := \min_{\ovl{B_1^+}} \varphi > \varphi_1'  := \min_{\ovl{W_1}} \varphi = e^{- \lambda \rho_1} = \max_{\ovl{W_2}\setminus W_1} \varphi  ,
$$
which only depend the geometric setting (not on the metric).

We let $\chi \in C^\infty_c(\R^{n+1})$ such that, with $W_j$ as in~\eqref{e:defW1W2}, $\chi=1$ on $W_1$ and  $\chi=0$ on $B_2^+ \setminus W_2$, and apply~\eqref{e:thm-Carleman-case-dirichlet-interp} to $u = \chi v \in C^\infty_c(B_2^+)$ with $v \in C^\infty(B_2^+)$ satisfies $v|_{s=0}=0$. We have $\d_\nu u |_{s=0}= - \chi|_{s=0} \d_s v|_{s=0}$ since $v|_{s=0}=0$ and hence
\bna
\int_{\{s=0\}} e^{2\tau \varphi}  \d_\nu \varphi  |\d_\nu u|^2 \leq C e^{2\tau \varphi(0)}\| \chi|_{s=0} \partial_{s} v|_{s=0}\|_{L^2(W_2 \cap\{s=0\})}^2  .
\ena
Using that $\chi=1$ on $W_1 \supset B_1^+$, we have that
\begin{align*}
 \tau^3 \nor{e^{\tau \varphi} u}{L^{2}(B_2^+)}^{2}+ \tau\nor{e^{\tau \varphi}\nabla u}{L^{2}(B_2^+)}^{2}
& \geq  \tau^3 \nor{e^{\tau \varphi} u}{L^{2}( B_1^+)}^{2}+ \tau\nor{e^{\tau \varphi}\nabla u}{L^{2}( B_1^+)}^{2} \\
 & \geq  \tau e^{2 \tau \varphi_1} \| v \|_{H^1( B_1^+)}^2  .
\end{align*}
Finally, we have $ \Lap \chi v = \chi \Lap v + [\Lap, \chi] v$, where $[\Lap , \chi]$ (recall $\Lap =  \partial_s^2 + \Delta_{\mathfrak{g}}$) is a first order differential operator with $L^\infty$ coefficients supported in $\ovl{W_2} \setminus W_1$, and such that $\nor{[\Lap , \chi]}{L^2 \to L^2} \leq C D$ on that set uniformly for $\mathfrak{g}\in \Gamma_{\eps, D}(\ovl{B}_{\R^n}(0,2), \id)$. Moreover, we have $\varphi \leq \varphi_1'$ on $\ovl{W_2} \setminus W_1$. Thus, we have
\begin{align*}
 \nor{e^{\tau \varphi}\Lap u}{L^2(B_2^+)}^2  & \leq   \| e^{\tau \varphi} \chi \Lap v \|_{L^2(B_2^+)}^2 +  \| e^{\tau \varphi}  [\Lap , \chi] v \|_{L^2(B_2^+)}^2 \\
 & \leq e^{2 \tau \varphi(0)} \| \Lap v \|_{L^2(B_2^+)}^2 + CD e^{2 \tau \varphi_1'} \| v \|_{H^1(B_2^+)}^2  .
\end{align*}
Combining the last three estimates with~\eqref{e:thm-Carleman-case-dirichlet-interp}, we find that there is $C,\tau_0>0$ such that for all
 $g =\id \otimes \mathfrak{g} \in \Gamma_{\eps, D}(\ovl{B_2^+}, \id)$, for all $\tau \geq \tau_0$ and all $v \in C^\infty(B_2^+)$ such that $v=0$ on $\{s=0\}$, we have
\bna
e^{2\tau \varphi_1 } \| v \|_{H^1( B_1^+)}^2 \leq C e^{2\tau \varphi(0)} \left( \| \partial_{s} v|_{s=0}\|_{L^2(B_2^+ \cap\{s=0\})}^2 
+  \| \Lap v \|_{L^2(B_2^+)}^2 \right) + C e^{2 \tau \varphi_1'} \| v \|_{H^1(B_2^+)}^2  .
\ena
Recalling that $\varphi(0)> \varphi_1 > \varphi_1'$ and after an optimization in the parameter $\tau$ (see~\cite{Robbiano:95}), this yields the result of the lemma.
\end{proof}

\subsection{A uniform Lebeau-Robbiano spectral inequality}

In this section, we give a proof of Theorem~\ref{t:uniform-LR-ineq-metric}. For this, we follow the strategy of proof of~\cite[Section~2]{BHLR:10} with our uniform Carleman estimates (Theorem~\ref{t:Carlemanuniform}). The original proof of~\cite{LR:95} also works (see the above Section~\ref{s:unif-LR-ineq}) but is less straightforward in the present setting.
We recall that $\M$ is the ambient compact manifold with boundary $\d \M$, and set $M= [0,S_0] \times \M$, having piecewise $C^1$ and graph-Lipschitz boundary $\d M = \{0\}\times \M \cup \{S_0\}\times \M\cup [0,S_0] \times \d \M$. We denote by $(s, x)$ the variable in $M$. the metric is $g= \id \otimes \mathfrak{g}$. Note finally that $\d_\nu = \d_{\nu_x}$ on $[0,S_0]\times \d \M$, were $\nu_x$ denotes here the outward unit normal to $\M$ at $\d \M$, that $\d_\nu = \d_s$ on $\{S_0\}\times \M$, and that $\d_\nu = - \d_s$ on $\{0\}\times \M$.

\begin{lemma}
\label{l:psi-jerome}
Let $\mathfrak{g}_0 \in \mathcal{T}^2_{W^{1,\infty}(\M)}$ be a metric on $\M$. Then, there exists a function $\psi \in C^2(M;\R)$ and $c>0$ such that 
\begin{equation*}
\begin{array}{ll}
|\nabla_{\mathfrak{g}_0}\psi|_{\mathfrak{g}_0} \geq c \text{ in } M , & \quad \d_{\nu_x}\psi <0 \text{ on }[0,S_0] \times \d\M  , \\
\d_s \psi \geq c \text{ on } \{0\} \times (\M \setminus \omega),  & \quad \nabla_{\mathfrak{g}_0}\psi = 0 \text{ and } \d_s \psi \leq -c \text { on } \{S_0\} \times \M .
\end{array}
\end{equation*}
\end{lemma}
We refer to~\cite[Appendix~C]{BHLR:10} for the proof of this result.
With this weight function in hand, we obtain the following global uniform Carleman estimate.

\begin{theorem}[Global uniform Lipschitz Carleman estimate]
\label{t:Carleman-global}
Given a metric $\mathfrak{g}_0 \in \mathcal{T}^2_{W^{1,\infty}(\M)}$, and $\Psi$ as in Lemma~\ref{l:psi-jerome}, for any $D \geq \eps >0$, there exist $\lambda>0 ,C_1>0, \tau_0 >0$ such that for $\varphi=e^{\lambda \Psi}$ and for any $\mathfrak{g} \in \Gamma_{\eps, D}(\M, \mathfrak{g}_0)$, 
for all $\tau \geq \tau_0$ and all $v \in H^2([0,S_0] \times \M)$ such that $v=0$ on $\{0\}\times \M \cup [0,S_0] \times \d \M$, we have with $M=[0,S_0] \times \M$ and $g= \id \otimes \mathfrak{g}$,
\begin{align}
\label{e:thm-Carleman-final-jerome}
 &\tau^3 \nor{e^{\tau \varphi}v }{L^{2}(M)}^{2}+ \tau\nor{e^{\tau \varphi}\nabla_{g} v}{L^{2}(M)}^{2} \nonumber\\
&+   \tau e^{2\tau \varphi(S_0)}  \left( \int_{\M} |\d_s v (S_0, \cdot)|^2 
    +  \tau^2 \int_{\M} |v(S_0, \cdot)|^2 \right)  + \tau \int_{\M \setminus \omega} e^{2\tau \varphi(0,\cdot)}  |\d_s v(0,\cdot)|^2 \nonumber \\
& \leq C \left( \nor{e^{\tau \varphi}(-\d_s^2 - \Delta_{\mathfrak{g}}) v}{L^2(M)}^2
 + \tau \int_{\omega} e^{2\tau \varphi(0,\cdot)}  |\d_s v(0, \cdot)|^2 
 +  \tau e^{2\tau \varphi(S_0)} \int_{\M} |\nabla_{\mathfrak{g}} v(S_0, \cdot)|_{\mathfrak{g}}^2 \right). 
\end{align}
\end{theorem}
\bnp
We use the Carleman estimates~\eqref{e:thm-Carleman-case-dirichlet}-\eqref{e:Carleman-Jerome} together with Remark~\ref{r:splitting-boundary} and Lemma~\ref{l:uniform-subellipticity} for the uniformity in the metric. More precisely, on the boundary $\{0\}\times \M \cup [0,S_0] \times \d \M$, the Dirichlet boundary condition is prescribed and the only boundary term is $+ \tau \int_{\Sigma} e^{2\tau \varphi}  \d_\nu \varphi  |\d_\nu v|^2$, according to~\eqref{e:thm-Carleman-case-dirichlet}. That $ \d_\nu \varphi \leq -c <0$ on $\{0\}\times (\M \setminus \omega) \cup [0,S_0] \times \d \M$ implies that the associated integral is dominated on that set, whereas the only observation term on that part of the boundary is $- \tau \int_{\omega} e^{2\tau \varphi(0,\cdot)}  \d_s \varphi(0,\cdot)  |\d_s v(0,\cdot)|^2$.

Now, on the part $\{S_0\} \times \M$ of the boundary, we have the observation term $\tau \int_{\Sigma} e^{2\tau \varphi}\gln{\nabla_T v} = \tau e^{2\tau \varphi(S_0)} \int_{\M} |\nabla_{\mathfrak{g}} v(S_0, \cdot)|_{\mathfrak{g}}^2$. On the other side of the inequality, we have the two observed terms 
$$
 \frac{\tau}{8} \frac{m(\varphi)^3}{M(\varphi)^2}\intS e^{2\tau \varphi}|\d_\nu v |^2 
    + \tau^3\frac{m(\varphi)^3}{4} \intS |v|^2 
    \geq C  \tau e^{2\tau \varphi(S_0)}  \left( \int_{\M} |\d_s v (S_0, \cdot)|^2 
    +  \tau^2 \int_{\M} |v(S_0, \cdot)|^2 \right) .
$$
Finally, we are left with the existence of $C,\tau_0>0$ such that for all $v \in H^2([0,S_0]\times \M)$, $\mathfrak{g} \in \Gamma_{\eps, D}(\M, \mathfrak{g}_0)$, and $\tau \geq \tau_0$, we have~\eqref{e:thm-Carleman-final-jerome}.
\enp

From Theorem~\ref{t:Carleman-global}, we now deduce a proof of Theorem~\ref{t:uniform-LR-ineq-metric}, following closely (and carefully)~\cite[Proof of Theorem~1.1]{BHLR:10}.
\bnp[Proof of Theorem~\ref{t:uniform-LR-ineq-metric}]
Given $w \in E_{\leq \lambda}^{\mathfrak{g}}$ take the function
$$
v(s) = \frac{\sinh(s \sqrt{-\Delta_{\mathfrak{g}}})}{\sqrt{-\Delta_{\mathfrak{g}}}} \Pi_+^{\mathfrak{g}} w + s \Pi_0^{\mathfrak{g}} w ,
$$
where $\Delta_{\mathfrak{g}}$ is the Dirichlet Laplacian, $\Pi_0^{\mathfrak{g}}$ the orthogonal projector on $\ker(\Delta_{\mathfrak{g}})$ and $\Pi_+^{\mathfrak{g}} = \id-\Pi_0^{\mathfrak{g}}$, that is $v$ is the unique solution to 
$$
(-\d_s^2-\Delta_{\mathfrak{g}})v = 0 , \quad v|_{(0,S_0)\times \d M}=0,  \quad (v, \d_s v)|_{s=0} = (0 , w) .
$$
We may now apply~\eqref{e:thm-Carleman-final-jerome}, keeping only the last term in the left hand-side:
\begin{align*}
 e^{2\tau \varphi(S_0)}  \tau^3 \int_{\M} |v(S_0, \cdot)|^2  
\leq C \left(\tau \int_{\omega} e^{2\tau \varphi(0,\cdot)}  |\d_s v(0, \cdot)|^2 
 +  \tau e^{2\tau \varphi(S_0)} \int_{\M} |\nabla_{\mathfrak{g}} v(S_0, \cdot)|_{\mathfrak{g}}^2 \right). 
\end{align*}
Now, we have 
$$
\int_{\omega} e^{2\tau \varphi(0,\cdot)}  |\d_s v(0, \cdot)|^2  \leq e^{2 \tau \sup_\M \varphi(0,\cdot)} \nor{w}{L^2(\omega)}^2,
$$
together with (using an integration by parts, together with $w \in E_{\leq \lambda}^{\mathfrak{g}}$),
$$
\int_{\M} |\nabla_{\mathfrak{g}} v(S_0, \cdot)|_{\mathfrak{g}}^2 =  \left(- \Delta_{\mathfrak{g}} v(S_0, \cdot), v(S_0, \cdot) \right)_{L^2(\M, d \Vol_{\mathfrak{g}})} \leq \lambda \left( v(S_0, \cdot), v(S_0, \cdot) \right)_{L^2(\M, d \Vol_{\mathfrak{g}})} .
$$
The last three inequalities imply for all $\tau \geq \tau_0$
\begin{align*}
 \tau^2 \nor{v(S_0, \cdot)}{L^2}^2
\leq C \left( e^{2\tau  \left( \sup_\M \varphi(0,\cdot)- \varphi(S_0) \right)}  \nor{w}{L^2(\omega)}^2 
 + \lambda \nor{v(S_0, \cdot)}{L^2}^2\right),
\end{align*}
and hence, when choosing $\tau =\max\{2 \sqrt{\lambda} , \tau_0\}$, we obtain 
$$
\nor{v(S_0, \cdot)}{L^2}^2
\leq C e^{4 \sqrt{\lambda}  \left( \sup_\M \varphi(0,\cdot)- \varphi(S_0) \right)}  \nor{w}{L^2(\omega)}^2 .
$$
Finally, using $\frac{\sinh(S_0\ell)}{\ell}\geq S_0$ and the orthogonality of the eigenfunctions, we also have 
$$
\int_{\M} |v(S_0, \cdot)|^2  = \left(  \frac{\sinh^2(S_0 \sqrt{-\Delta_{\mathfrak{g}}})}{-\Delta_{\mathfrak{g}}} \Pi_+^{\mathfrak{g}} w, \Pi_+^{\mathfrak{g}} w \right)_{L^2(\M, d \Vol_{\mathfrak{g}})} + S_0^2 \| \Pi_0^{\mathfrak{g}} w\|_{L^2(\M, d \Vol_{\mathfrak{g}})} ^2 \geq S_0^2 \| w\|_{L^2}^2 .
$$
The last two inequalities conclude the proof of the theorem.
\enp

\section{Local behavior of vanishing functions}
In this appendix, we give an explicit link between the different definitions of the vanishing rate of a function.
\begin{lemma}
\label{lmordrevanishinL2}
Let $f\in C^{\infty}(B_{\R^n}(0,1))$ and assume that there are $C,D>0$ such that we have uniformly for $0<r<1$ the estimate
\bnan
\label{e:vanishorderD}
\nor{f}{L^2(B_{\R^n}(0,r))}\leq C r^D .
\enan
Then, we have $\partial^{\alpha}f(0)=0$ for all $|\alpha|<D-n/2$.

Conversely, assume $f\in C^{\infty}(B_{\R^n}(0,1))$ satisfies $\partial^{\alpha}f(0)=0$ for all $|\alpha|\leq k$, $k \in \N$. Then we have \eqref{e:vanishorderD} with $D= k+1+n/2$.
\end{lemma}
\bnp
Denote by $k = \inf \{ |\alpha| , \d^{\alpha}f(0) \neq 0\} \in \N \cup \{\infty\}$ and, in case $k<\infty$, write the Taylor expansion of $f$ at zero as $f=P_k+R_k$ with $P_k$ homogeneous of degree $k$ and $|R_k|\leq C|x|^{k+1}$. We obtain 
\bna
\nor{P_k}{L^2(B(0,r))}=r^{n/2+k}\nor{P_k}{L^2(B(0,1))}, \quad \text{and}\quad 
\nor{R_k}{L^2(B(0,r))}\leq Cr^{n/2+k+1}.
\ena
Using~\eqref{e:vanishorderD} for $r$ small implies $n/2+k\geq D$ and thus $\partial^{\alpha}f(0)=0$ for all $|\alpha|<D-n/2$.

Conversely, if $\partial^{\alpha}f(0)=0$ for all $|\alpha|\leq k$, then we have $|f(x)|\leq C |x|^{k+1}$ and thus $$\nor{f}{L^2(B_{\R^n}(0,r))}\leq C  \nor{|x|^{k+1}}{L^2(B_{\R^n}(0,r))} \leq Cr^{k+1+n/2}.$$
\enp

\small
\bibliographystyle{alpha}
\bibliography{bibli}
\end{document}

%% file: mymacros.tex
\newtheorem{lemma}{Lemma}[section]
\newtheorem{theorem}[lemma]{Theorem}
\newtheorem{proposition}[lemma]{Proposition}
\newtheorem{corollary}[lemma]{Corollary}

\newtheorem{conjecture}[lemma]{Conjecture}
\theoremstyle{definition}

\newtheorem{definition}[lemma]{Definition}
\newtheorem{remark}[lemma]{Remark}

\makeatletter
\def\keywords{
    \vspace{1ex}
    \noindent
    \if@twocolumn
      \small{\bf  Keywords}\/---$\!$    \else
      \begin{center}\small\ {\bf Keywords}\end{center}\quotation\small
    \fi}
\def\endkeywords{\vspace{0.6em}\par\if@twocolumn\else\endquotation\fi
    \normalsize\rm}
\makeatother

\renewcommand{\O}{\ensuremath{\mathcal O}}

\newcommand{\calR}{\ensuremath{\mathcal R}}

\renewcommand{\L}{\ensuremath{\mathcal L}}

\newcommand{\calS}{\ensuremath{\mathcal S}}

\newcommand{\ID}{\ensuremath{\mathbb D}}
\newcommand{\x}{\ensuremath{\bf x}}

\DeclareMathOperator{\Sp}{Sp}
\DeclareMathOperator{\loc}{loc}
\DeclareMathOperator{\comp}{comp}

\DeclareMathOperator{\Vol}{Vol}
\DeclareMathOperator{\Loc}{Loc}
\DeclareMathOperator{\Int}{Int}

\newcommand{\mb}[1]{\ensuremath{\mathbb{#1}}}
\newcommand{\N}{{\mb{N}}}

\newcommand{\R}{{\mb{R}}}
\newcommand{\C}{{\mb{C}}}

\newcommand{\w}{\ensuremath{w}}
\newcommand{\W}{\ensuremath{W}}
\newcommand{\y}{\ensuremath{y}}

\newcommand{\eps}{\varepsilon}

\newcommand{\M}{\ensuremath{\mathcal M}}
\newcommand{\T}{\ensuremath{\mathbb T}}

\newcommand{\B}{\ensuremath{\mathscr B}}

\let \Re \relax
\DeclareMathOperator{\Re}{Re}
\let \Im \relax
\DeclareMathOperator{\Im}{Im}

\newcommand{\ovl}[1]{\overline{#1}}

\renewcommand{\S}{\ensuremath{\mathbb S}}

\DeclareMathOperator{\supp}{supp}

\DeclareMathOperator{\dist}{dist}

\DeclareMathOperator{\vect}{span}

\DeclareMathOperator{\id}{Id}

\renewcommand{\d}{\ensuremath{\partial}}

\let \div \relax
\DeclareMathOperator{\div}{div}

\def\x{x}

\newcommand{\E}{\mathscr E}
\newcommand{\Z}{\mathbb Z}

\renewcommand{\H}{\ensuremath{\mathcal H}}
\newcommand{\K}{\ensuremath{\mathcal K}}

\newcommand{\rouge}[1]{{\color{red} #1}}